\documentclass[12pt]{amsart}
\usepackage{oldgerm}
\usepackage{latexsym}
\usepackage{amsmath}
\usepackage{amscd}
\usepackage{youngtab}
\usepackage{amssymb} % \mathfrak , \mathbb
\usepackage{amsmath}
\usepackage{amsthm}
\usepackage{latexsym}
\usepackage[mathscr]{eucal}
\usepackage{xypic}
\usepackage{tikz-cd}
\usepackage{easywncy}
\newlength{\margins}
\usepackage[top=\margins,bottom=\margins,left=\margins,right=\margins]{geometry}
\usepackage{xcolor}

\makeatletter
\@namedef{subjclassname@2020}{%
  \textup{2020} Mathematics Subject Classification}
\makeatother
\begin{document}
\title [Harder's conjecture II: Extended version] { Harder's conjecture II: Extended version with full proofs
}  
\author[H. Atobe]{Hiraku ATOBE}
\address{Hiraku  Atobe, Graduate school of mathematics, Kyoto University, Kitashirakawa, Kyoto, 606-8502, Japan}
\email{atobe@math.kyoto-u.ac.jp}
\author[M. Chida]{ Masataka CHIDA}
\address{Masataka Chida, School of Science and Technology for Future Life, Tokyo Denki University,
5 Senju Asahi-cho, Adachi-ku, Tokyo 120-8551, Japan}
\email{chida@mail.dendai.ac.jp}
\author [T. Ibukiyama] {Tomoyoshi IBUKIYAMA} 
\address{Tomoyoshi Ibukiyama, Department of Mathematics, Graduate School of Science, Osaka University, 
Machikaneyama 1-1, Toyonaka, Osaka, 560-0043 Japan}
\email{ibukiyam@math.sci.osaka-u.ac.jp}
\author[H. Katsurada]{Hidenori KATSURADA}
\address{Hidenori Katsurada, Department of Mathematics, Hokkaido University, Kita 10, Nishi 8, Kita-ku, Sapporo, Hokkaido, 060-0810, Japan and Muroran Institute of Technology, Mizumoto 27-1, Muroran, Hokkaido, 050-8585  Japan}
\email{hidenori@muroran-it.ac.jp}
 \author[T. Yamauchi]{ Takuya YAMAUCHI}
\address{Takuya Yamauchi, Mathematical Institute, Tohoku University,  6-3, Aoba, Aramaki, Aoba-Ku, Sendai 980-8578, JAPAN}
\email{takuya.yamauchi.c3@tohoku.ac.jp}

\thanks{Atobe was supported by JSPS KAKENHI Grant Number JP19K14494. Chida was supported by JSPS KAKENHI Grant Number JP18K03202. Ibukiyama was supported by JSPS KAKENHI Grant Number JP19K03424, JP20H00115. Katsurada was  supported by KAKENHI Grant Number  JP21K03152, JP24K06660.  Yamauchi was supported by JSPS KAKENHI Grant Number JP19H01778. }

%%%

%%%

%\footnotetext[1]{Partially supported by Grant-in-Aid for Scientific Research, the Ministry of Education, Science, Sports  and Culture, Japan}
\date{June 15, 2026}

\maketitle
%Greek letters
\newcommand{\Bell}{\ensuremath{\boldsymbol\ell}}

\newcommand{\alp}{\alpha}
\newcommand{\bet}{\beta}
\newcommand{\gam}{\gamma}
\newcommand{\del}{\delta}
\newcommand{\eps}{\epsilon}
\newcommand{\zet}{\zeta}
\newcommand{\tht}{\theta}
\newcommand{\iot}{\iota}
\newcommand{\kap}{\kappa}
\newcommand{\lam}{\lambda}
\newcommand{\sig}{\sigma}
\newcommand{\ups}{\upsilon}
\newcommand{\ome}{\omega}
\newcommand{\vep}{\varepsilon}
\newcommand{\vth}{\vartheta}
\newcommand{\vpi}{\varpi}
\newcommand{\vrh}{\varrho}
\newcommand{\vsi}{\varsigma}
\newcommand{\vph}{\varphi}
\newcommand{\Gam}{\Gamma}
\newcommand{\Del}{\Delta}
\newcommand{\Tht}{\Theta}
\newcommand{\Lam}{\Lambda}
\newcommand{\Sig}{\Sigma}
\newcommand{\Ups}{\Upsilon}
\newcommand{\Ome}{\Omega}

\def\depth{{\rm depth}}
\def\bbD{{\Bbb D}}
\def\p{{\partial}}
%fraktur letters

\newcommand{\frka}{{\mathfrak a}}    \newcommand{\frkA}{{\mathfrak A}}
\newcommand{\frkb}{{\mathfrak b}}    \newcommand{\frkB}{{\mathfrak B}}
\newcommand{\frkc}{{\mathfrak c}}    \newcommand{\frkC}{{\mathfrak C}}
\newcommand{\frkd}{{\mathfrak d}}    \newcommand{\frkD}{{\mathfrak D}}
\newcommand{\frke}{{\mathfrak e}}    \newcommand{\frkE}{{\mathfrak E}}
\newcommand{\frkf}{{\mathfrak f}}    \newcommand{\frkF}{{\mathfrak F}}
\newcommand{\frkg}{{\mathfrak g}}    \newcommand{\frkG}{{\mathfrak G}}
\newcommand{\frkh}{{\mathfrak h}}    \newcommand{\frkH}{{\mathfrak H}}
\newcommand{\frki}{{\mathfrak i}}    \newcommand{\frkI}{{\mathfrak I}}
\newcommand{\frkj}{{\mathfrak j}}    \newcommand{\frkJ}{{\mathfrak J}}
\newcommand{\frkk}{{\mathfrak k}}    \newcommand{\frkK}{{\mathfrak K}}
\newcommand{\frkl}{{\mathfrak l}}    \newcommand{\frkL}{{\mathfrak L}}
\newcommand{\frkm}{{\mathfrak m}}    \newcommand{\frkM}{{\mathfrak M}}
\newcommand{\frkn}{{\mathfrak n}}    \newcommand{\frkN}{{\mathfrak N}}
\newcommand{\frko}{{\mathfrak o}}    \newcommand{\frkO}{{\mathfrak O}}
\newcommand{\frkp}{{\mathfrak p}}    \newcommand{\frkP}{{\mathfrak P}}
\newcommand{\frkq}{{\mathfrak q}}    \newcommand{\frkQ}{{\mathfrak Q}}
\newcommand{\frkr}{{\mathfrak r}}    \newcommand{\frkR}{{\mathfrak R}}
\newcommand{\frks}{{\mathfrak s}}    \newcommand{\frkS}{{\mathfrak S}}
\newcommand{\frkt}{{\mathfrak t}}    \newcommand{\frkT}{{\mathfrak T}}
\newcommand{\frku}{{\mathfrak u}}    \newcommand{\frkU}{{\mathfrak U}}
\newcommand{\frkv}{{\mathfrak v}}    \newcommand{\frkV}{{\mathfrak V}}
\newcommand{\frkw}{{\mathfrak w}}    \newcommand{\frkW}{{\mathfrak W}}
\newcommand{\frkx}{{\mathfrak x}}    \newcommand{\frkX}{{\mathfrak X}}
\newcommand{\frky}{{\mathfrak y}}    \newcommand{\frkY}{{\mathfrak Y}}
\newcommand{\frkz}{{\mathfrak z}}    \newcommand{\frkZ}{{\mathfrak Z}}

%math boldface latters

\newcommand{\bfa}{{\mathbf a}}    \newcommand{\bfA}{{\mathbf A}}
\newcommand{\bfb}{{\mathbf b}}    \newcommand{\bfB}{{\mathbf B}}
\newcommand{\bfc}{{\mathbf c}}    \newcommand{\bfC}{{\mathbf C}}
\newcommand{\bfd}{{\mathbf d}}    \newcommand{\bfD}{{\mathbf D}}
\newcommand{\bfe}{{\mathbf e}}    \newcommand{\bfE}{{\mathbf E}}
\newcommand{\bff}{{\mathbf f}}    \newcommand{\bfF}{{\mathbf F}}
\newcommand{\bfg}{{\mathbf g}}    \newcommand{\bfG}{{\mathbf G}}
\newcommand{\bfh}{{\mathbf h}}    \newcommand{\bfH}{{\mathbf H}}
\newcommand{\bfi}{{\mathbf i}}    \newcommand{\bfI}{{\mathbf I}}
\newcommand{\bfj}{{\mathbf j}}    \newcommand{\bfJ}{{\mathbf J}}
\newcommand{\bfk}{{\mathbf k}}    \newcommand{\bfK}{{\mathbf K}}
\newcommand{\bfl}{{\mathbf l}}    \newcommand{\bfL}{{\mathbf L}}
\newcommand{\bfm}{{\mathbf m}}    \newcommand{\bfM}{{\mathbf M}}
\newcommand{\bfn}{{\mathbf n}}    \newcommand{\bfN}{{\mathbf N}}
\newcommand{\bfo}{{\mathbf o}}    \newcommand{\bfO}{{\mathbf O}}
\newcommand{\bfp}{{\mathbf p}}    \newcommand{\bfP}{{\mathbf P}}
\newcommand{\bfq}{{\mathbf q}}    \newcommand{\bfQ}{{\mathbf Q}}
\newcommand{\bfr}{{\mathbf r}}    \newcommand{\bfR}{{\mathbf R}}
\newcommand{\bfs}{{\mathbf s}}    \newcommand{\bfS}{{\mathbf S}}
\newcommand{\bft}{{\mathbf t}}    \newcommand{\bfT}{{\mathbf T}}
\newcommand{\bfu}{{\mathbf u}}    \newcommand{\bfU}{{\mathbf U}}
\newcommand{\bfv}{{\mathbf v}}    \newcommand{\bfV}{{\mathbf V}}
\newcommand{\bfw}{{\mathbf w}}    \newcommand{\bfW}{{\mathbf W}}
\newcommand{\bfx}{{\mathbf x}}    \newcommand{\bfX}{{\mathbf X}}
\newcommand{\bfy}{{\mathbf y}}    \newcommand{\bfY}{{\mathbf Y}}
\newcommand{\bfz}{{\mathbf z}}    \newcommand{\bfZ}{{\mathbf Z}}

%caligraphic letters

%\newcommand{\cal}{\fam2}
\newcommand{\cala}{{\mathcal A}}
\newcommand{\calb}{{\mathcal B}}
\newcommand{\calc}{{\mathcal C}}
\newcommand{\cald}{{\mathcal D}}
\newcommand{\cale}{{\mathcal E}}
\newcommand{\calf}{{\mathcal F}}
\newcommand{\calg}{{\mathcal G}}
\newcommand{\calh}{{\mathcal H}}
\newcommand{\cali}{{\mathcal I}}
\newcommand{\calj}{{\mathcal J}}
\newcommand{\calk}{{\mathcal K}}
\newcommand{\call}{{\mathcal L}}
\newcommand{\calm}{{\mathcal M}}
\newcommand{\caln}{{\mathcal N}}
\newcommand{\calo}{{\mathcal O}}
\newcommand{\calp}{{\mathcal P}}
\newcommand{\calq}{{\mathcal Q}}
\newcommand{\calr}{{\mathcal R}}
\newcommand{\cals}{{\mathcal S}}
\newcommand{\calt}{{\mathcal T}}
\newcommand{\calu}{{\mathcal U}}
\newcommand{\calv}{{\mathcal V}}
\newcommand{\calw}{{\mathcal W}}
\newcommand{\calx}{{\mathcal X}}
\newcommand{\caly}{{\mathcal Y}}
\newcommand{\calz}{{\mathcal Z}}

%math script

\newcommand{\scra}{{\mathscr A}}
\newcommand{\scrb}{{\mathscr B}}
\newcommand{\scrc}{{\mathscr C}}
\newcommand{\scrd}{{\mathscr D}}
\newcommand{\scre}{{\mathscr E}}
\newcommand{\scrf}{{\mathscr F}}
\newcommand{\scrg}{{\mathscr G}}
\newcommand{\scrh}{{\mathscr H}}
\newcommand{\scri}{{\mathscr I}}
\newcommand{\scrj}{{\mathscr J}}
\newcommand{\scrk}{{\mathscr K}}
\newcommand{\scrl}{{\mathscr L}}
\newcommand{\scrm}{{\mathscr M}}
\newcommand{\scrn}{{\mathscr N}}
\newcommand{\scro}{{\mathscr O}}
\newcommand{\scrp}{{\mathscr P}}
\newcommand{\scrq}{{\mathscr Q}}
\newcommand{\scrr}{{\mathscr R}}
\newcommand{\scrs}{{\mathscr S}}
\newcommand{\scrt}{{\mathscr T}}
\newcommand{\scru}{{\mathscr U}}
\newcommand{\scrv}{{\mathscr V}}
\newcommand{\scrw}{{\mathscr W}}
\newcommand{\scrx}{{\mathscr X}}
\newcommand{\scry}{{\mathscr Y}}
\newcommand{\scrz}{{\mathscr Z}}

%math Bbb

\newcommand{\AAA}{{\mathbb A}} %not \AA
\newcommand{\BB}{{\mathbb B}}
\newcommand{\CC}{{\mathbb C}}
\newcommand{\DD}{{\mathbb D}}
\newcommand{\EE}{{\mathbb E}}
\newcommand{\FF}{{\mathbb F}}
\newcommand{\GG}{{\mathbb G}}
\newcommand{\HH}{{\mathbb H}}
\newcommand{\II}{{\mathbb I}}
\newcommand{\JJ}{{\mathbb J}}
\newcommand{\KK}{{\mathbb K}}
\newcommand{\LL}{{\mathbb L}}
\newcommand{\MM}{{\mathbb M}}
\newcommand{\NN}{{\mathbb N}}
\newcommand{\OO}{{\mathbb O}}
\newcommand{\PP}{{\mathbb P}}
\newcommand{\QQ}{{\mathbb Q}}
\newcommand{\RR}{{\mathbb R}}
\newcommand{\SSS}{{\mathbb S}} %not \SS
\newcommand{\TT}{{\mathbb T}}
\newcommand{\UU}{{\mathbb U}}
\newcommand{\VV}{{\mathbb V}}
\newcommand{\WW}{{\mathbb W}}
\newcommand{\XX}{{\mathbb X}}
\newcommand{\YY}{{\mathbb Y}}
\newcommand{\ZZ}{{\mathbb Z}}

%typewriter

\newcommand{\tta}{\hbox{\tt a}}    \newcommand{\ttA}{\hbox{\tt A}}
\newcommand{\ttb}{\hbox{\tt b}}    \newcommand{\ttB}{\hbox{\tt B}}
\newcommand{\ttc}{\hbox{\tt c}}    \newcommand{\ttC}{\hbox{\tt C}}
\newcommand{\ttd}{\hbox{\tt d}}    \newcommand{\ttD}{\hbox{\tt D}}
\newcommand{\tte}{\hbox{\tt e}}    \newcommand{\ttE}{\hbox{\tt E}}
\newcommand{\ttf}{\hbox{\tt f}}    \newcommand{\ttF}{\hbox{\tt F}}
\newcommand{\ttg}{\hbox{\tt g}}    \newcommand{\ttG}{\hbox{\tt G}}
\newcommand{\tth}{\hbox{\tt h}}    \newcommand{\ttH}{\hbox{\tt H}}
\newcommand{\tti}{\hbox{\tt i}}    \newcommand{\ttI}{\hbox{\tt I}}
\newcommand{\ttj}{\hbox{\tt j}}    \newcommand{\ttJ}{\hbox{\tt J}}
\newcommand{\ttk}{\hbox{\tt k}}    \newcommand{\ttK}{\hbox{\tt K}}
\newcommand{\ttl}{\hbox{\tt l}}    \newcommand{\ttL}{\hbox{\tt L}}
\newcommand{\ttm}{\hbox{\tt m}}    \newcommand{\ttM}{\hbox{\tt M}}
\newcommand{\ttn}{\hbox{\tt n}}    \newcommand{\ttN}{\hbox{\tt N}}
\newcommand{\tto}{\hbox{\tt o}}    \newcommand{\ttO}{\hbox{\tt O}}
\newcommand{\ttp}{\hbox{\tt p}}    \newcommand{\ttP}{\hbox{\tt P}}
\newcommand{\ttq}{\hbox{\tt q}}    \newcommand{\ttQ}{\hbox{\tt Q}}
\newcommand{\ttr}{\hbox{\tt r}}    \newcommand{\ttR}{\hbox{\tt R}}
\newcommand{\tts}{\hbox{\tt s}}    \newcommand{\ttS}{\hbox{\tt S}}
\newcommand{\ttt}{\hbox{\tt t}}    \newcommand{\ttT}{\hbox{\tt T}}
\newcommand{\ttu}{\hbox{\tt u}}    \newcommand{\ttU}{\hbox{\tt U}}
\newcommand{\ttv}{\hbox{\tt v}}    \newcommand{\ttV}{\hbox{\tt V}}
\newcommand{\ttw}{\hbox{\tt w}}    \newcommand{\ttW}{\hbox{\tt W}}
\newcommand{\ttx}{\hbox{\tt x}}    \newcommand{\ttX}{\hbox{\tt X}}
\newcommand{\tty}{\hbox{\tt y}}    \newcommand{\ttY}{\hbox{\tt Y}}
\newcommand{\ttz}{\hbox{\tt z}}    \newcommand{\ttZ}{\hbox{\tt Z}}

\newcommand{\rank}{\mathrm{rank}}

\newcommand{\phm}{\phantom}
\newcommand{\ds}{\displaystyle }
\newcommand{\smallstrut}{\vphantom{\vrule height 3pt }}
\def\bdm #1#2#3#4{\left(
\begin{array} {c|c}{\ds{#1}}
 & {\ds{#2}} \\ \hline
{\ds{#3}\vphantom{\ds{#3}^1}} &  {\ds{#4}}
\end{array}
\right)}
\newcommand{\wtd}{\widetilde }
\newcommand{\bsl}{\backslash }
\newcommand{\GL}{{\mathrm{GL}}}
\newcommand{\SL}{{\mathrm{SL}}}
\newcommand{\GSp}{{\mathrm{GSp}}}
\newcommand{\PGSp}{{\mathrm{PGSp}}}
\newcommand{\SP}{{\mathrm{Sp}}}
\newcommand{\SO}{{\mathrm{SO}}}
\newcommand{\SU}{{\mathrm{SU}}}
\newcommand{\Ind}{\mathrm{Ind}}
\newcommand{\Hom}{{\mathrm{Hom}}}
\newcommand{\Ad}{{\mathrm{Ad}}}
\newcommand{\Sym}{{\mathrm{Sym}}}
\newcommand{\Mat}{\mathrm{M}}
\newcommand{\sgn}{\mathrm{sgn}}
\newcommand{\trs}{\,^t\!}
\newcommand{\iu}{\sqrt{-1}}
\newcommand{\oo}{\hbox{\bf 0}}
\newcommand{\ono}{\hbox{\bf 1}}
\newcommand{\smallcirc}{\lower .3em \hbox{\rm\char'27}\!}
\newcommand{\bAf}{\bA_{\hbox{\eightrm f}}}
\newcommand{\thalf}{{\textstyle{\frac12}}}
\newcommand{\shp}{\hbox{\rm\char'43}}
\newcommand{\Gal}{\operatorname{Gal}}
\newcommand{\ev}{\mathrm{ev}}

\newcommand{\bdel}{{\boldsymbol{\delta}}}
\newcommand{\bchi}{{\boldsymbol{\chi}}}
\newcommand{\bgam}{{\boldsymbol{\gamma}}}
\newcommand{\bome}{{\boldsymbol{\omega}}}
\newcommand{\bpsi}{{\boldsymbol{\psi}}}
\newcommand{\blam}{{\boldsymbol{\lambda}}}
\newcommand{\GK}{\mathrm{GK}}
\newcommand{\EGK}{\mathrm{EGK}}
\newcommand{\MGK}{\mathrm{MGK}}
\newcommand{\ord}{\mathrm{ord}}
\newcommand{\nd}{\mathrm{nd}}
\newcommand{\diag}{\mathrm{diag}}
\newcommand{\ua}{{\underline{a}}}
\newcommand{\ub}{{\underline{b}}}
\newcommand{\uc}{{\underline{c}}}
\newcommand{\ud}{{\underline{d}}}
\newcommand{\ue}{{\underline{e}}}
\newcommand{\uf}{{\underline{f}}}
\newcommand{\ug}{{\underline{g}}}
\newcommand{\uh}{{\underline{h}}}
\newcommand{\ui}{{\underline{i}}}
\newcommand{\uj}{{\underline{j}}}
\newcommand{\uk}{{\underline{k}}}
\newcommand{\ul}{{\underline{l}}}
\newcommand{\um}{{\underline{m}}}
\newcommand{\un}{{\underline{n}}}
\newcommand{\ZZn}{\ZZ_{\geq 0}^n}
\newcommand{\uzet}{{\underline{\zeta}}}
\newcommand{\St}{\mathrm{St}}
\newcommand{\Spin}{\mathrm{Spin}}
\newcommand{\alg}{\mathrm{alg}}

\newtheorem{theorem}{Theorem}[section]
\newtheorem{lemma}[theorem]{Lemma}
\newtheorem{proposition}[theorem]{Proposition}
\newtheorem{corollary}[theorem]{Corollary}
\newtheorem{conjecture}[theorem]{Conjecture}
\newtheorem{definition}[theorem]{Definition}
\newtheorem{remark}[theorem]{{\bf Remark}}
\theoremstyle{plain}

%%%%%%%%%%%%%%%%title%%%%%%%%%%%%%%%%%%%
%

\def\mattwono(#1;#2;#3;#4){\begin{array}{cc}
                               #1  & #2 \\
                               #3  & #4
                                      \end{array}}

\def\mattwo(#1;#2;#3;#4){\left(\begin{matrix}
                               #1 & #2 \\
                               #3  & #4
                                      \end{matrix}\right)}
 \def\smallmattwo(#1;#2;#3;#4){\left(\begin{smallmatrix}
                               #1 & #2 \\
                               #3  & #4
                                      \end{smallmatrix}\right)}                                     
                                      
 \def\matthree(#1;#2;#3;#4;#5;#6;#7;#8;#9){\left(\begin{matrix}
                               #1 & #2  & #3\\
                               #4  & #5 & #6\\
                               #7  & #8 &#9 
                                      \end{matrix}\right)}                                     
                                      
\def\mattwo(#1;#2;#3;#4){\left(\begin{matrix}
                               #1 & #2 \\
                               #3  & #4
                                      \end{matrix}\right)}  

\def\rowthree(#1;#2;#3){\begin{matrix}
                               #1   \\
                               #2  \\
                               #3
                                      \end{matrix}}  
\def\columnthree(#1;#2;#3){\begin{matrix}
                               #1   &   #2  &  #3
                                      \end{matrix}}  
                                      
\def\rowfive(#1;#2;#3;#4;#5){\begin{array}{lllll}
                               #1   \\
                               #2  \\
                               #3 \\
                               #4 \\
                               #5                              
                                      \end{array}} 

\def\columnfive(#1;#2;#3;#4;#5){\begin{array}{lllll}
                               #1   &   #2  &  #3 & #4 & #5
                                \end{array}}

\def\mattwothree(#1;#2;#3;#4;#5;#6){\begin{matrix}
                               #1 & #2  & #3  \\
                               #4 & #5  & #6
                                      \end{matrix}}  
\def\matthreetwo(#1;#2;#3;#4;#5;#6){\begin{array}{lc}
                               #1  & #2  \\
                               #3  & #4 \\
                               #5  & #6
                                      \end{array}}  
\def\columnthree(#1;#2;#3){\begin{matrix}
                               #1 & #2 & #3  
                                  \end{matrix}}  
\def\rowthree(#1;#2;#3){\begin{matrix}
                               #1 \\
                                #2 \\
                                #3  
                                  \end{matrix}}  
                                        
\begin{abstract}
Let $f$ be a primitive form of weight $2k+j-2$ for $\SL_2(\ZZ)$, and let $\frkp$ be a prime ideal of the Hecke field of $f$. We denote by $\SP_m(\ZZ)$ the Siegel modular group of degree $m$. Suppose that $k \equiv 0 \mod 2, \ j \equiv 0 \mod 4$ and that $\frkp$ divides the algebraic part of $L(k+j,f)$. Put ${\bf k}=(k+j/2,k+j/2,j/2+4,j/2+4)$. Then under certain easily checkable conditions, we prove that  there exists a Hecke eigenform $F$ in the space of modular forms of weight $(k+j,k)$ for $\SP_2(\ZZ)$ such that $[\scri_2(f)]^{\bf k}$ is congruent to $\scra^{(I)}_4(F)$ modulo $\frkp$. Here, $[\scri_2(f)]^{\bf k}$ is the Klingen-Eisenstein lift of the Saito-Kurokawa lift $\scri_2(f)$ of $f$  to the space of modular forms of weight ${\bf k}$ for $\SP_4(\ZZ)$, and $\scra^{(I)}_4(F)$ is a certain lift of $F$ to the space of cusp forms of weight ${\bf k}$ for $\SP_4(\ZZ)$. As an application, we  prove 
Harder's conjecture on the congruence  between the Hecke eigenvalues of $F$ and some quantities related to the Hecke eigenvalues of $f$. This version gives proofs of Lemmas \ref{lem.lattice-change} and \ref{lem.irreducibility-criterion}  and Corollaries \ref{cor.no-factor} and \ref{cor.no-factor2} in  the paper arXiv:2306.07582v2.
\end{abstract} 
   \tableofcontents                               
\section{Introduction}
This is a sequel to our previous paper \cite{A-C-I-K-Y23}.  Let $f$ be a primitive form with respect to $\SL_2(\ZZ)$. In the previous paper, we proposed a conjecture on  the congruence between the Klingen-Eisenstein lift of the Duke-Imamoglu-Ikeda lift of $f$ and a certain lift of a vector valued Hecke eigenform with respect to  $\SP_2(\ZZ)$. This conjecture implies  Harder's conjecture.
We proved the above conjecture in some cases. In this paper,
we prove the above conjecture in a more general setting.
We explain this more precisely. For a non-increasing sequence  ${\bf k}=(k_1,\ldots,k_n)$ of non-negative integers we denote by $M_{\bf k}(\SP_n(\ZZ))$ and $S_{\bf k}(\SP_n(\ZZ))$ the spaces of modular forms and cusp forms of weight ${\bf k}$ (or, weight $k$, if ${\bf k}=(\overbrace{k,\ldots,k}^n)$) for $\SP_n(\ZZ)$, respectively. (For the definition of modular forms, see Section 2.) For two Hecke eigenforms $F$ and $G$ and a prime ideal $\frkp$ of the Hecke field $\QQ(F)$
in $M_{\bf k}(\SP_m(\ZZ))$, we say that $G$ is congruent to $F$ modulo $\frkp$ and write
\[G \equiv_{{\rm ev}} F \mod \frkp\]
if 
\[\lambda_G(T) \equiv\lambda_F(T) \mod {\frkp'}\]
for any integral Hecke operator $T$, where $\lambda_F(T)$ and $\lambda_G(T)$ are the eigenvalues of the Hecke operator $T$ on $F$ and $G$, respectively and $\frkp'$ is a prime ideal of the composite field $\QQ(F)\QQ(G)$ lying above $\frkp$. (As for the definition of integral Hecke operator, see Section 3.)
Let $f(z)=\sum_{m=1}^{\infty}a(m,f)\exp(2\pi \sqrt{-1}mz)$ be a primitive form in $S_{2k+j-2}(\SL_2(\ZZ))$, and suppose that a `big prime' divides the algebraic part of $L(k+j,f)$.  Let $\scri_2(f)$ be the Duke-Imamoglu-Ikeda lift of $f$  to the space of cusp forms of weight ${j \over 2}+k$ for $\SP_2(\ZZ)$. Let $k$ and $j$ be positive even integers such that $k,j  \ge 4$ and $j \equiv 0 \mod 4$. For a sequence 
\begin{align*}
& {\bf k}=\Bigl({j \over 2}+k,{ j \over 2}+k,{j \over 2}+4,{j \over 2}+4 \Bigr),
\end{align*}
 let  $[\scri_2(f)]^{{\bf k}}$ be the Klingen-Eisenstein lift of $\scri_2(f)$ to $M_{{\bf k}}(\SP_4(\ZZ))$. 
In the previous  paper, we proposed a conjecture concerning the congruence between  the Klingen-Eisenstein lift of the Duke-Imamoglu-Ikeda lift and a certain lift of Hecke eigenforms of degree two.  The following conjecture is a special case of it. 
\begin{conjecture} \label{conj.enhanced-Harder}
Under  the above notation and assumption,
 there exists a Hecke eigenform $F$ in $S_{(k+j,k)}(\SP_2(\ZZ))$ such that 
\[\scra^{(I)}_4(F) \equiv_{\rm ev}[\scri_2(f)]^{{\bf k}} \mod \frkp.\]
Here, $\scra_4^{(I)}(F)$ is the lift of $F$ to $S_{{\bf k}}(\SP_4(\ZZ))$, called the lift of type $\scra^{(I)}$, which will be given in Theorem \ref{th.atobe1}. 
\end{conjecture}
In \cite{A-C-I-K-Y23}, we proved that the above conjecture implies  Harder's conjecture. That is, we proved the following conjecture from the above conjecture:
\begin{conjecture} \label{conj.Harder}
Under the above notation and assumption, 
there exists a Hecke eigenform $F$ in $S_{(k+j,k)}(\SP_2(\ZZ))$ such that
\[\lambda_F(T(\ell)) \equiv a(\ell,f)+\ell^{k-2}+\ell^{j+k-1} \mod {\frkp'} \]
for any prime number $\ell$, where $\lambda_F(T(\ell))$ is the eigenvalue of the Hecke operator $T(\ell)$ on $F$, and $\frkp'$ is a prime ideal of $\QQ(f) \cdot \QQ(F)$ lying above $\frkp$. 
\end{conjecture}
Moreover, in \cite{A-C-I-K-Y23},  in the case
$(k,j)=(10,4),(14,4)$ and $(4,24)$, we confirmed Conjecture \ref{conj.enhanced-Harder}, and therefore Conjecture \ref{conj.Harder}. To do this, we needed some results on non-congruence for modular forms, and we confirmed  them case by case. In this paper, we prove  several auxiliary results on non-congruence for modular forms, and prove  Conjecture \ref{conj.enhanced-Harder} in general under certain conditions that can be easily checked for any concrete examples.
\begin{theorem}\label{th.intro-main-result} (cf. Theorems \ref{th.main-result} and \ref{th.main-result2}.)
Suppose that $k$ is even and $j \equiv 0 \mod 4$. Then, under certain  conditions explained later in Theorem \ref{th.main-result}  or in Theorem \ref{th.main-result2}  later, 
Conjecture \ref{conj.enhanced-Harder}, and therefore Conjecture \ref{conj.Harder} hold.
\end{theorem}
The conditions in Theorems \ref{th.main-result} and \ref{th.main-result2} are mainly concerned with the $\frkp$-indivisibility of the Fourier coefficients of the Klingen-Eisenstein series and the critical values of the Hecke $L$-function  related to $f$, where $\frkp$ and $f$ are those in Conjecture \ref{conj.enhanced-Harder}. We have an algorithm for computing these values (cf. Remark \ref{rem.confirm-assumption}, (1) and Remark \ref{rem.confirm-assumption2}, (2).) 
Therefore,  we can easily check whether these conditions hold or not  for any concrete examples (cf. Theorem \ref{th.main-result2}), and  we confirm these conjectures  in many  concrete examples (cf. Theorem \ref{th.examples}).
We briefly describe our methods. There are two key ingredients  for proving the main results. One is the congruence for the Klingen-Eisenstein lifts as developed in the previous paper.
Indeed, under the above assumption, we can prove that there exists a Hecke eigenform  $G \in M_{{\bf k}}(\SP_4(\ZZ))$ such that $G$ is not a constant multiple of $[\scri_2(f)]^{{\bf k}}$ and 
\[ G \equiv_{\rm ev} [\scri_2(f)]^{{\bf k}}    \mod {\frkp} \]
 (cf. Theorem \ref{th.main-congruence}).  Therefore, to prove the above conjecture, it suffices to show that $G$ is a lift of type $\scra^{(I)}$. To show this, we consider non-congruence between Galois representations. To be more precise, suppose that $G$ is a cusp form and that the above congruence holds. Let $K$ be an algebraic number field of finite degree containing the Hecke fields $\QQ(G)$ and $\QQ(f)$, and $\frkP$ a prime ideal of $K$ lying above $\frkp$. Let $\rho_{G,\St}:\Gal(\bar \QQ/\QQ) \to \GL_9(K_\frkP)$ be the Galois representation attached to $G$ in Theorem \ref{th.Galois-st}, and $\rho_f$ be the $2$-dimensional Galois representation attached to $f$ in Theorem \ref{th.Galois-spin}. Moreover, suppose that the reduction $\bar \rho_f$ of $\rho_f$ mod $\frkP$ is absolutely irreducible.  Then, the semi-simplification $\bar \rho_{G,\St}^{\mathrm{ss}}$ of $\rho_{G,\St}$ mod $\frkP$ is expressed as 
\[(*) \quad \bar \rho_{G,\St}^{\mathrm{ss}}=\bar \rho_f(k+j/2-1) \oplus \bar \rho_f(k+j/2-2) \oplus \bar \chi^{j/2+1} \oplus \bar \chi^{-j/2-1} \oplus \bar \chi^{j/2} \oplus \bar \chi^{-j/2} \oplus \bar{\bf{1}},\]
where $\chi=\chi_p$ is the $p$-adic cyclotomic character.  Then we can show
the following theorem.
\begin{theorem}\label{th.intro-noncongruence}(cf. Theorem \ref{th.noncongruence}).
The congruence relation (*)  is impossible if $G$ is not the lift of type $\scra^{(I)}$ under certain conditions.
\end{theorem}
 Similarly, we can show non-congruence between non-cuspidal Hecke eigenforms in $M_{{\bf k}}(\SP_4(\ZZ))$ and $[\scri_2(f)]^{{\bf k}}$ (cf. Theorem \ref{th.noncongruence-Phi}).

This paper is organized as follows. In Sections 2 and 3, we give a brief  summary of  Siegel modular forms and several $L$-values, respectively. In Section 4, we explain the Galois representations for modular forms which are necessary for the present paper.  In Section 5,  we introduce several lifts, and among other things define the lift of type $\scra^{(I)}$ of a vector valued modular form in $S_{(k+j,k)}(\SP_2(\ZZ))$, and state our main results. In Section 6, we review the result on  the congruence for vector valued Klingen-Eisenstein series following \cite{A-C-I-K-Y23}. In particular, we explain  how the assumption that $\frkp$ divides the algebraic part of $L(k+j,f)$ for $f \in S_{2k+j-2}(\SL_2(\ZZ))$ gives the congruence between  $[\scri_n(f)]^{{\bf k}}$ and another Hecke eigenform in $M_{{\bf k}'}(\SP_{2n}(\ZZ))$. Moreover, we review how to compute Fourier coefficients of the Klingen-Eisenstein series in question. In Section 7, we provide 
several preliminary results for proving Theorem \ref{th.main-result}, and in Section 8, we prove it.
In Section 9, we prove Theorem \ref{th.main-result2}.  In Section 10, we give examples that confirm our conjecture and also  Harder's. In Appendix A, we give a proof of Lemma \ref{lem.comparison-of-periods} which compares Kato's period with the Petersson inner product of elliptic modular forms.

\smallskip

\noindent\textbf{Acknowledgments.} The authors thank Masanori Asakura, Siegfried B\"ocherer, Ga\"etan Chenevier, Neil Dummigan, G\"unter Harder, Tamotsu Ikeda, Kai-Wen Lan, Chul-hee Lee, Tadashi Ochiai, Yuichiro Taguchi, Nobuki Takeda and Seidai Yasuda for valuable comments. They also thank the referee for his/her valuable comments and suggestions, which greatly helped to improve the manuscript.

\smallskip

\indent
{\sc Notation.}  
Let $R$ be a commutative ring. We denote by $R^{\times}$ the unit group of $R$.
We denote by $M_{mn}(R)$ the set of
$m \times n$-matrices with entries in $R.$ In particular put $M_n(R)=M_{nn}(R).$   Put $\GL_m(R) = \{A \in M_m(R) \ | \ \det A \in R^\times \},$ where $\det
A$ denotes the determinant of a square matrix $A$. We sometimes write $|A|$ instead of $\det A$. For an $m \times n$-matrix $X$ and an $m \times m$-matrix
$A$, we write $A[X] = {}^t \!X A X,$ where $^t \!X$ denotes the
transpose of $X$. Let $\Sym_n(R)$ denote
the set of symmetric matrices of degree $n$ with entries in
$R.$ Furthermore, if $R$ is an integral domain of characteristic different from $2,$ let  $\calh_n(R)$ denote the set of half-integral matrices of degree $n$ over $R$, that is, $\calh_n(R)$ is the subset of symmetric
matrices of degree $n$ with entries in the field of fractions of $R$ whose $(i,j)$-component belongs to
$R$ or ${1 \over 2}R$ according as $i=j$ or not.  
  For a subset $S$ of $M_n(R)$ we denote by $S^{\nd}$ the subset of $S$
consisting of non-degenerate matrices. If $S$ is a subset of $\Sym_n(\RR)$ with $\RR$ the field of real numbers, we denote by $S_{>0}$ (resp. $S_{\ge 0}$) the subset of $S$
consisting of positive definite (resp. semi-positive definite) matrices. The group  
$\GL_n(R)$ acts on the set $\Sym_n(R)$ in the following way:
\[
\GL_n(R) \times \Sym_n(R) \ni (g,A) \longmapsto A[g] \in \Sym_n(R).
\]
Let $G$ be a subgroup of $\GL_n(R).$ For a $G$-stable subset ${\mathcal B}$ of $\Sym_n(R)$  we denote by $\calb/G$ the set of equivalence classes of $\calb$ under the action of  $G.$ We sometimes use the same symbol $\calb/G$ to denote a complete set of representatives of $\calb/G.$ We abbreviate $\calb/\GL_n(R)$ as $\calb/\!\!\sim$ if there is no fear of confusion. Let $R'$ be a subring of $R$. Then two symmetric matrices $A$ and $A'$ with
entries in $R$ are said to be equivalent over $R'$ with each
other and write $A \sim_{R'} A'$ if there is
an element $X$ of $\GL_n(R')$ such that $A'=A[X].$ We also write $A \sim A'$ if there is no fear of confusion. We denote by $O$ the zero matrix. We sometimes write $0$ instead of $O$. If we need to specify its size, we write $O_{m,n}$ or $O_n$.
For square matrices $X$ and $Y$ we write $X \bot Y = \begin{pmatrix} X & O\\ O &Y \end{pmatrix}.$

We put ${\bf e}(x)=\exp(2 \pi \sqrt{-1} x)$ for $x \in {\CC},$ and for a prime number $p$ we denote by ${\bf e}_p(*)$ the continuous additive character of $\QQ_p$ such that ${\bf e}_p(x)= {\bf e}(x)$ for $x \in \ZZ[p^{-1}].$

Let $K$ be an algebraic number field, and $\frkO=\frkO_K$ the ring of integers in $K$. For a prime ideal $\frkp$ we denote by $K_{\frkp}$ and $\frkO_{\frkp}$ the $\frkp$-adic completion of $K$ and $\frkO$, respectively, and put $\frkO_{(\frkp)}=\frkO_{\frkp} \cap K$.
For a prime ideal $\frkp$ of $\frkO$, we denote by $\ord_{\frkp}(*)$ the additive valuation of $K_{\frkp}$ normalized so that $\ord_{\frkp}(\vpi)=1$ for a prime element $\vpi$ of $K_{\frkp}$. Moreover for any element $a, b \in \frkO_{(\frkp)}$ 
we write $b \equiv a \mod {\frkp}$ if $\ord_{\frkp}(a-b) >0$.
For $a \in K_\frkp$, we sometimes say that $\frkp$ divides $a$ if $\ord_\frkp(a)>0$.

\section{Siegel modular forms}
In this section we review Siegel modular forms. For details, see \cite[Section 2]{A-C-I-K-Y23}.
We denote by $\HH_{n}$ the Siegel upper half 
space of degree $n$:
\[
\HH_{n}=\{Z\in M_{n}(\CC) \mid Z=\,^{t}\!Z=X+\sqrt{-1}Y,\ X,Y\in M_{n}(\RR),Y>0\}.
\]  
For any ring $R$ and any natural integer $n$, we define the
group $\GSp_n(R)$ by
\[
\GSp_n(R)=\{g\in M_{2n}(R) \ | \ gJ_{n}\,^{t}\! g=\nu(g) J_{n} \text{ with some } \nu(g) \in R^{\times}\},
\]
where $J_{n}=\left(\begin{smallmatrix} 0_{n} & -1_{n} \\ 1_{n} & 0_{n} 
\end{smallmatrix}\right)$.  We call $\nu(g)$ the symplectic similitude of $g$. 
We also define the  symplectic group of degree $n$ over $R$ by 
\[
\SP_n(R)=\{g \in \GSp_n(R) \ | \  \nu(g)=1 \}.
\]
In particular, if $R$ is a subfield of $\RR$,
we define 
\[\GSp_n(R)^+=\{g \in \GSp_n(R) \ | \  \nu(g)>0 \}.
\]
 We put $\varGamma^{(n)}=\SP_n(\ZZ)$ for the sake of simplicity, and call it the Siegel modular group of degree $n$.
Now we define vector valued Siegel modular forms of $\varGamma^{(n)}$. 
Let  $(\rho,V)$ be a polynomial representation of $\GL_n(\CC)$ 
on a finite dimensional complex vector space $V$. We fix a Hermitian inner
product $\langle  \, ,  \, \rangle$ on $V$ such that
\[\langle \rho(g)v,w\rangle=\langle v,\rho({}^t\bar g)w \rangle \quad \text{ for } g \in \GL_n(\CC), v,w \in V.\]
For  any $V$-valued function $F$ on $\HH_{n}$, 
and for any $g=\left(\begin{smallmatrix} A & B \\ C & D \end{smallmatrix}
 \right)
\in \GSp(n,\RR)^+$, we put $J(g,Z)=CZ+D$ and 
\[
F|_{\rho}[g]=\rho(J(g,Z))^{-1}F(gZ).
\]

We say that $F$ is a holomorphic Siegel modular form of weight $\rho$ 
with respect to $\varGamma^{(n)}$ if $F$ is holomorphic on $\HH$ 
and $F|_{\rho}[\gamma]=F$ for any $\gamma\in \varGamma^{(n)}$ 
(with the extra condition of holomorphy at all the cusps  if $n=1$). 
We denote by $M_{\rho}(\varGamma^{(n)})$ the space of 
modular forms of weight $\rho$ with respect to $\varGamma^{(n)}$.
A modular form $F \in M_{\rho}(\varGamma^{(n)})$ has the following Fourier expansion
\[F(Z)=\sum_{T \in \calh_n(\ZZ)_{\ge 0}}a(T,F){\bf e}(\mathrm{tr}(TZ)) \quad (a(T,F) \in V),\]
where $\mathrm{tr}(A)$ is the trace of a matrix $A$. We say that $F$ is a cusp form 
 if we have
$a(T,F)=0$ unless $T$ is positive definite. We denote by  $S_{\rho}(\varGamma^{(n)})$ the subspace of $M_{\rho}(\varGamma^{(n)})$ consisting of cusp forms.  
For $F, G \in M_{\rho}(\varGamma^{(n)})$ the Petersson inner product is defined by
\begin{align*}
&(F,G)= \int_{\varGamma^{(n)} \backslash \HH_n} \langle \rho(\sqrt{Y})F(Z),\rho(\sqrt{Y})G(Z)\rangle \det (Y)^{-n-1} dZ,
\end{align*}
where $Y=\mathrm{Im}(Z)$ and $\sqrt{Y}$ is a positive definite symmetric matrix such that $\sqrt{Y}^2=Y$.
This integral converges if either $F$ or $G$ belongs to $S_{\rho}(\varGamma^{(n)})$.

Let $\lambda=(k_1,k_2,\ldots)$ be a finite sequence of non-negative integers such that $k_i \ge k_{i+1}$ for all $i$ and $k_m=0$ for some  $m$. We call this a dominant integral weight.
We call the largest integer $m$ such that $k_m \not=0$ the depth of $\lambda$ and denote it by $\depth(\lambda)$. It is well-known that the set of dominant integral weights $\lambda$ of $\depth (\lambda) \le n$ corresponds bijectively to the set of irreducible polynomial representations of $\GL_n(\CC)$ (cf. \cite{Weyl}). We denote this representation by $(\rho_{n,\lambda},V_{n,\lambda})$.
We also denote it by $(\rho_{\bf k},V_{\bf k})$ with ${\bf k}=(k_1,\ldots,k_n)$ and call it the irreducible polynomial representation of $\GL_n(\CC)$ of  highest weight ${\bf k}$. We then set $M_{\bf k}(\varGamma^{(n)})=M_{\rho_{\bf k}}(\varGamma^{(n)})$ and 
$S_{\bf k}(\varGamma^{(n)})=S_{\rho_{\bf k}}(\varGamma^{(n)})$. We say $F$ is a modular form of weight ${\bf k}$ if it is a modular form of weight $\rho_{\bf k}$. If ${\bf k}=(\overbrace{k,\ldots,k}^n)$, we simply write $M_k(\varGamma^{(n)})=M_{{\bf k}}(\varGamma^{(n)})$ and 
$S_k(\varGamma^{(n)})=S_{{\bf k}}(\varGamma^{(n)})$. For the ${\bf k}=(k_1,\ldots,k_n)$ above, put ${\bf k}'
=(k_1-k_n,\ldots,k_{n-1}-k_n,0)$. Then, we have  $\rho_{\bf k} \cong \det^{k_n} \otimes \rho_{{\bf k}'}$ with $(\rho_{{\bf k}'},V_{ {\bf k}'})$ an irreducible polynomial representation of highest weight ${\bf k}'$. Here we understand that $(\rho_{{\bf k}'},V_{{\bf k}'})$ is the trivial representation on $\CC$ if $k_1=\cdots=k_{n-1}=k_n$. Moreover, we may regard an element $F \in M_k(\varGamma)$ as a $V_{{\bf k}'}$-valued holomorphic  function on $\HH$  such that 
\[F|_{\det^{k_n} \otimes \rho_{{\bf k}'}}[\gamma]=F\]
 for any $\gamma\in \varGamma^{(n)}$ 
(with the extra condition of holomorphy at all the cusps  if $n=1$). Then $F$ has the following Fourier expansion
\[F(Z)=\sum_{T \in \calh_n(\ZZ)_{\ge 0}}a(T,F){\bf e}(\mathrm{tr}(TZ)) \quad (a(T,F) \in V_{ {\bf k}'}).\]

For a representation $(\rho,V)$ of $\GL_n(\CC)$, we denote by $\frkF(\HH_n,V)$ the set of Fourier series $F(Z)$ on $\HH_n$ with values in $V$ of the following form:
\[
F(Z)=\sum_{A \in \calh_n(\ZZ)_{\ge 0}} a(A,F) {\bf e}(\mathrm{tr}(AZ)) \quad (Z \in \HH_n, \ a(A,F) \in V).
\]
For $F(Z) \in \frkF(\HH_n,V)$ we define $\Phi(F)(Z_1)=\Phi_r^n(F)(Z_1) \quad (Z_1 \in \HH_r)$ as
\[\Phi(F)(Z_1)=\lim_{\lambda \to \infty} F\Bigl(\begin{pmatrix} Z_1 & O \\ O & \sqrt{-1} \lambda 1_{n-r} \end{pmatrix}\Bigr).\]
We make the convention that $\frkF(\HH_0,V)=V$ and $\Phi^n_0(F)=a(O_n,F)$. Then, $\Phi(F)$ belongs to $\frkF(\HH_r,V)$. We denote by $\widetilde \frkF(\HH_n,(\rho,V))$
the subset of $\frkF(\HH_n,V)$ consisting of elements $F(Z)$ 
such that the following condition is satisfied:
\begin{align*} a(A[g],F)=\rho(g)a(A,F) \text{ for any } g \in \GL_n(\CC). \tag{K0} \end{align*}
Now  let $\Bell=(l_1,\ldots,l_n)$ be a dominant integral weight of depth $m$. Then, we review a realization $V_{\Bell}$ in terms of  bideterminants (cf. \cite{Ibukiyama-Takemori19}). 
Let $U=(u_{ij})$ be an $m \times n$ matrix of variables. For a positive integer $a$ let $\mathcal{SI}_{n,a}$ denote the set of strictly increasing sequences of positive integers not greater than $n$ of length $a$. For each $J=(j_1,\ldots,j_a) \in  \mathcal{SI}_{n,a}$ we define $U_J$ as
\[\begin{vmatrix} u_{1,j_1} & \ldots & u_{1,j_a} \\
\vdots & \vdots & \vdots \\
u_{a,j_1} &\ldots & u_{a,j_a}
\end{vmatrix}.\]
Then we say that a polynomial $P(U)$ in $U$ is a bideterminant of weight $ {\bf \ell}$ if
$P(U)$ is of the following form:
\[P(U)=\prod_{i=1}^r \prod_{j=1}^{l_i-l_{i+1}} U_{J_{i,j}},\]
where $(J_{i,1},\ldots,J_{i,l_i-l_{i+1}}) \in \mathcal{SI}_{n,m_i}^{l_i-l_{i+1}}$. Here we understand that $\prod_{j=1}^{l_i-l_{i+1}}U_{J_{i,j}}=1$ if $l_i=l_{i+1}$. Let $\mathcal {BD}_{{\Bell}}$ be the set of all bideterminants of weight ${\Bell}$. Here we make the convention that 
$\mathcal {BD}_{{\Bell}}=\{1 \}$ if ${\bf \ell}=(0,\ldots,0)$. 
For a commutative ring  $R$ and an $R$-algebra $S$ let $S[U]_{{\Bell}}$ denote the $R$-module of   all $S$-linear combinations of $P(U)$ for $P(U) \in  \mathcal {BD}_{{\Bell}}$. 
Then we can define an action of $\GL_n(\CC)$ on $\CC[U]_{\Bell}$ as
\[\GL_n(\CC) \times \CC[U]_{\Bell} \ni (g,P(U)) \mapsto P(U \ {}^t\! g) \in \CC[U]_{\Bell},\]
and we can take the $\CC$-vector space $\CC[U]_{{\Bell}}$  as a representation space $V_{\Bell}$ of $\rho_{\Bell}$ under this action. Now let $m \le n-1$ be a non-negative integer and  $U=(u_{ij})$ be an $m \times n$ matrix of variables. Let ${\bf k}=(k_1,\ldots,k_n)$ with $k_1 \ge \cdots \ge k_m>k_{m+1}=\cdots=k_n$ and ${\bf k}'=(k_1-k_{m+1},\ldots,k_m-k_{m+1},\overbrace{0,\ldots,0}^{n-m})$. Here we make the convention that ${\bf k}-(k_1,\ldots,k_1)$ and ${\bf k}'=(0,\ldots,0)$
if $m=0$. Under this notation and convention,
 $M_{{\bf k}}(\varGamma^{(n)})$ can be regarded as a $\CC$-subvector space of $\mathrm{Hol}(\HH_n)[U]_{{\bf k}'}$, where $\mathrm{Hol}(\HH_n)$ denotes the ring of holomorphic functions on $\HH_n$. Moreover, the Fourier expansion of $F(Z)$ can be expressed as
\[F(Z)=\sum_{A \in \calh_n(\ZZ)_{\ge 0}} a(A,F){\bf e}(\mathrm{tr}(AZ)),\]
where $a(A,F)=a(A,F;U) \in \CC[U]_{{\bf k}'}$.
Let $r$ be an integer such that $m \le r \le n$ and let ${\bf l}=(k_1,\ldots,k_{r-1},k_r)$ and ${\bf l}'=(k_1-k_{m+1},\ldots,k_r-k_{m+1},\overbrace{0,\ldots,0}^{r-m})$. For the $m \times n$ matrix $U$, let $U^{(r)}=(u_{ij})_{1 \le i \le m, 1 \le j \le r}$ and put $W'=\CC[U^{(r)}]_{{\bf l}'}$. Then we can define a representation $(\rho',W')$ of $\GL_r(\CC)$. The representations 
$(\rho',V')$ and $(\tau',W')$ satisfy the following conditions:
\begin{itemize}
\item[(K1)] $W' \subset V_{{\bf k}'}$;
\item[(K2)] $\rho'\Bigl( \left(\begin{smallmatrix} g_1 & g_2 \\ O & g_4 \end{smallmatrix} \right)\Bigr)v=\tau'(g_1)v$ for $\left(\begin{smallmatrix} g_1& g_2 \\ O & g_4 \end{smallmatrix} \right) \in \GL_n(\CC)$ with $g_1 \in \GL_r(\CC)$ and $v \in W'$;
\item[(K3)]  If $v \in V_{{\bf k}'}$ satisfies the condition
\[\rho_{{\bf k}'}\Bigl(\left(\begin{smallmatrix}1_r & O \\ O & h\end{smallmatrix}\right)\Bigr)v=v \text{ for any } h \in \GL_{n-r}(\CC),\]
then $v$ belongs to $W'$.
\end{itemize}
Then, $\Phi=\Phi_r^n$ is a mapping from $\widetilde \frkF(\HH_n,V)$  to $\widetilde \frkF(\HH_r,W)$ and it induces a mapping from
$M_{\rho}(\varGamma^{(n)})$ to $M_{\tau}(\varGamma^{(r)})$,
where $\rho=\det^{k_n} \otimes \rho_{{\bf k}'}$ and $\tau=\det^{k_n} \otimes \tau'$. Let ${\Delta}_{n,r}$
be the subgroup of ${\varGamma}^{(n)}$ defined by
$${\Delta}_{n,r} := \left\{ \, \begin{pmatrix} * & * \\ 
O_{(n-r,n+r)} & * \\ \end{pmatrix} \in {\varGamma}^{(n)} \, \right\} . $$
For $F\in S_{\bf l}(\varGamma^{(r)})$ the Klingen-Eisenstein lift  $[F]_{\tau}^{\rho}(Z,s)$ of $F$
to  $M_{{\bf k}}({\varGamma}^{(n)})$
 in the sense of
\cite{Langlands76} and Klingen \cite{Klingen67} is defined by
\begin{align*} [F]_{\tau}^{\rho}(Z,s):= \sum_{\gamma \in {\Delta}_{n,r}\backslash {\varGamma}^{(n)}} \Bigl({\det \mathrm{Im}(Z) \over \det \mathrm{Im} (\mathrm{pr}_r^n(Z))}\Bigr)^s F(\mathrm{pr}_r^n(Z))|_\rho \gamma .
\end{align*}
Here $\mathrm{pr}_r^n(Z)=Z_1$ for 
$Z=\begin{pmatrix} Z_1&Z_2 \\ {}^t\!Z_2&Z_4 \end{pmatrix} \in \HH_n$
with $Z_1 \in \HH_r, Z_4 \in \HH_{n-r}, Z_2 \in M_{r,n-r}(\CC)$.
We also write $[F]_{\tau}^{\rho}$ as $[F]_{\bf l}^{\bf k}$ or $[F]^{\bf k}$.
Suppose that $k_n$ is even and $2\mathrm{Re}(s) +k_n >n+r+1$. Then, by \cite{Klingen67}, 
 $[F]^{{\bf k}}(Z,s)$
converges absolutely and
uniformly on $\HH_n$ and holomorphic at $s=0$ as a function of $s$. We put $[F]_\tau^{\rho}(Z)=[F]_\tau^\rho(Z,0)$.
For a positive even integer $k$, we define  $E_{n,k}(Z,s)$ as
\[E_{n,k}(Z,s)=\sum_{\gamma \in \Delta_{n,0} \backslash \varGamma^{(n)}} (\det \mathrm{Im}(Z))^s|_\gamma\]
and call it the Siegel-Eisenstein series of weight $k$ with respect to $\varGamma^{(n)}$. The Siegel-Eisenstein series $E_{n,k}(Z,s)$ can be continued meromorphically to the whole $s$-plane as a function of $s$ and
holomorphic at $s=0$ as a function of $s$. We put $E_{n,k}(Z)=E_{n,k}(Z,0)$.

Let $\Bell=(l_1,\ldots,l_n)$ be a dominant integral weight of depth $m$. Let $\widetilde V=\widetilde V_{\Bell}=\QQ[U]_{\Bell}$. Then $(\rho_l|\GL_n(\QQ),\widetilde V)$ is a representation of $\GL_n(\QQ)$ and $\widetilde V \otimes \CC=V_{\Bell}$. 
We consider a $\ZZ$ structure of $V_{\Bell}$. We fix a basis $\cals=\cals_{\Bell}=\{P\}$ of $\ZZ[U]_{\Bell}$.
 Let $K$ be a number field, and $\frkO$ the ring of integers in $K$. For  a prime ideal  $\frkp$ of $\frkO$ and 
$a=\sum_{P \in \cals_{\Bell}} a_P P(U) \in K[U]_{\Bell}$ with $a_P \in K$, define 
\[\ord_{\frkp}(a)= \min_{P \in \cals_{\Bell}} \ord_{\frkp}(a_P).\]
For a subring $R$ of $\CC$, we denote by $M_{\bf k}(\varGamma^{(n)})(R)$ the $R$-submodule of $M_{{\bf k}}(\varGamma^{(n)})$ consisting of all elements $F$ such that $a(T,F) \in R[U]_{{\bf k}'}$ for all $T \in \calh_n(\ZZ)_{\ge 0}$.

We consider tensor products of modular forms. 
Let $n_1$ and $n_2$ be positive integers such that $n_1 \ge n_2$.
Let ${\bf k}_1=(k_1,\ldots,k_m,k_{m+1},\ldots,k_{n_1})$ and ${\bf k}_1=(k_1,\ldots,k_m,k_{m+1},\ldots,k_{n_2})$ be non-increasing sequences of integers such that $k_{m} >k_{m+1}=\cdots=k_{n_1}=l$.
Then 
$(\rho_{{\bf k}_1} \otimes \rho_{{\bf k}_2}, V_{1} \otimes V_{2})$ is a representation of $\GL_{n_1}(\CC) \times \GL_{n_2}(\CC)$. Put ${\bf k}_1'=(k_1-l,\ldots,k_m-l,\overbrace{0,\ldots,0}^{n_1-m})$ and
${\bf k}_2'=(k_1-l,\ldots,k_m-l,\overbrace{0,\ldots,0}^{n_2-m})$.
Then, $\rho_{{\bf k}_1} \otimes \rho_{{\bf k}_2}=\det^l \otimes
\rho_{{\bf k}_1'} \otimes \det^l \otimes \rho_{{\bf k}_2'}$
with $(\rho_{\bf {k}_i'},V_{i}')$ a polynomial representation of highest weight ${\bf k}_i'$ for $i=1,2$.
To make our formulation smooth, we sometimes regard 
a modular form of scalar weight $k$ for $\varGamma^{(n)}$ as a function with values in the one-dimensional vector space spanned by $\det U^k$, where $U$ is an $n \times n$ matrix of variables. 
 Let $U_1$ and $U_2$ be $m \times n_1$ and $m \times n_2$ matrices, respectively, of variables and for a commutative ring $R$ and an $R$-algebra $S$ let
\begin{align*}
&S[U_1,U_2]_{{\bf k}_1',{\bf k}_2'}\\
&=\left\{\sum_j P_j(U_1)P_j(U_2) \quad (\text{ finite sum }) \text{ with } P_j (U_i) \in S[U_i]_{{\bf k}_i'} \ (i=1,2) \right\}.
\end{align*}
Here we make the convention that $P_i(U_i) =(\det U_i)^{k_1-l}$ if $n_1=m$ and $k_1=\cdots=k_m$. 
Then, as a representation space $W=W_{{\bf k}_1',{\bf k}_2'}$ of 
$\rho_{{\bf k}_1'} \otimes \rho_{{\bf k}_2'}$ we can take  $\CC[U_1,U_2]_{{\bf k}_1',{\bf k}_2'}$.
Let 
\[\widetilde W=\widetilde W_{{\bf k}_1',{\bf k}_2'}=\QQ[U_1,U_2]_{{\bf k}_1',{\bf k}_2'}. \]
Then $\widetilde W \cong \widetilde V_{1}' \otimes \widetilde V_{2}'$ and  $\widetilde W \otimes_{\QQ} \CC=W$.
Let 
\[M=M_{{\bf k}_1',{\bf k}_2'}=\ZZ[U_1,U_2]_{{\bf k}_1',{\bf k}_2'}. \]
We note that 
\[M=\left \{\sum_{P_1 \in \cals_{{\bf k}_1'}, P_2 \in \cals_{{\bf k}_2'}}a_{P_1,P_2}P_1(U_1)P_2(U_2) \right\}.\]
Here we make the convention that $P_1(U_1)=\det U_1^{k_1-l}$ if
$n_1=m$ and $k_1=\cdots=k_m$. Therefore, $M$ is a lattice of $\widetilde V$ and $M \cong L_1 \otimes L_2$ with 
$L_i=\ZZ[U_i]_{{\bf k}_i'}\ (i=1,2).$
Thus  $(\rho_{{\bf k}_1} \otimes \rho_{{\bf k}_2}, V_1 \otimes V_2 )$ has also a $\QQ$-structure and $\ZZ$-structure and we can define $\ord_{\frkp}(a \otimes b)$ for $a \otimes b \in \widetilde  W_K$.  If $\dim_{\CC} V_1=1,$ then we sometimes identify $V_1,\widetilde V_1$ and $L_1$ with $\CC,\QQ$ and $\ZZ$, respectively, and for $a, b \in V_1$ and $w \in V_2$, we write
$a \otimes b$ and $a \otimes w$ as $ab$ and  $aw$, respectively through the identifications $V_1 \otimes V_1 \cong V_1$ and $V_1 \otimes V_2 \cong V_2 \otimes V_1 \cong V_2$. 
 $M_{{\bf k}_1}(\varGamma^{(n_1)}) \otimes M_{{\bf k}_2}(\varGamma^{(n_2)})$ is regarded as a $\CC$-subspace of
$(\mathrm{Hol}(\HH_{n_1}) \otimes \mathrm{Hol}(\HH_{n_2}))[U_1,U_2]_{{\bf k}_1',{\bf k}_2'}$.

\section{Special values of $L$-functions attached to modular  forms}
In this section we review several $L$-values of modular forms that appear in this article. For the details, see \cite[Section 3]{A-C-I-K-Y23}.
Let ${\bf L}_n={\bf L}(\varGamma^{(n)},\GSp_n^+(\QQ) \cap M_{2n}(\ZZ))$ be the Hecke algebra over $\ZZ$ associated to the Hecke pair $(\varGamma^{(n)},\GSp_n^+(\QQ) \cap M_{2n}(\ZZ))$ and ${\bf L}_n(\QQ)={\bf L}(\varGamma^{(n)},\GSp_n^+(\QQ))$ be the Hecke algebra over $\QQ$ associated to the Hecke pair $(\varGamma^{(n)},\GSp_n^+(\QQ))$.
\begin{remark}
In \cite[Section 3]{A-C-I-K-Y23}, we defined ${\bf L}_n(\QQ)$ as ${\bf L}_n \otimes_\ZZ \QQ$. But it is not appropriate. Actually ${\bf L}_n \otimes_\ZZ \QQ$ does not contain $[a^{-1}]_n$ with a positive integer $a$.   Therefore,  we changed it as  above (cf. \cite{K-L26}). Moreover, we changed all the ${\bf L}_n(\CC)$ on \cite[page 1349]{A-C-I-K-Y23} and ${\bf L}_n$ on \cite[page 1349, line -7, and page 1350, line 11]{A-C-I-K-Y23} to the ${\bf L}_n(\QQ)$ in the present paper.
\end{remark}
Let ${\bf k}=(k_1,\ldots,k_n)$ be a non-increasing sequence of non-negative integers. 
For an element  $T=\varGamma^{(n)} g \varGamma^{(n)} \in {\bf L}_n(\QQ)$, let 
\[T=\bigsqcup_{i=1} ^r \varGamma^{(n)} g_i\]
be the coset decomposition. Then, for a modular form $F \in M_{{\bf k}}(\varGamma^{(n)})$  we define $F|T$ as
\[F|T=\nu(g)^{k_1+\cdots+k_n-n(n+1)/2}\sum_{i=1}^r F|_{\rho_{\bf k}}g_i.\]
This defines an action of the Hecke algebra ${\bf L}_n(\QQ)$ on $M_{{\bf k}}(\varGamma^{(n)})$. The operator $F|T$ with $T \in {\bf L}_n(\QQ)$ is called the Hecke operator. 
We say that $F$ is a Hecke eigenform if $F$ is a common eigen function of all Hecke operators $T \in {\bf L}_n(\QQ)$. We note that  $F$ is a Hecke eigenform if $F$ is a common eigen function of all Hecke operators $T \in {\bf L}_n$.
Then we have 
\[F|T =\lambda_F(T) F \text{ with } \lambda_F(T) \in \CC  \text{ for any } T \in {\bf L}_n(\QQ).\]
We call $\lambda_F(T)$ the Hecke eigenvalue of $T$ with respect to $F$. 
For a Hecke eigenform $F$ in $M_{{\bf k}}(\varGamma^{(n)})$, we denote by  $\QQ(F)$ the field generated over $\QQ$ by all the Hecke eigenvalues $\lambda_F(T)$ with $T \in {\bf L}_n(\QQ)$ and call it the  Hecke field of $F$. For two Hecke eigenforms $F$ and $G$ we sometimes write $\QQ(F,G)=\QQ(F)\QQ(G)$.
We say that an element $T \in {\bf L}_n(\QQ)$ is integral with respect to  $M_{{\bf k}}(\varGamma^{(n)})$ if $F|T \in M_{{\bf k}}(\varGamma^{(n)})(\ZZ)$ for any $F \in  M_{{\bf k}}(\varGamma^{(n)})(\ZZ)$. 
We denote by ${\bf L}_n^{({\bf k})}$ the subset of ${\bf L}_n(\QQ)$  consisting of all integral elements with respect to  $M_{{\bf k}}(\varGamma^{(n)})$. 
By \cite[Proposition 3.1]{A-C-I-K-Y23}, we have
${\bf L}_n \subset {\bf L}_n^{({\bf k})}$ for any ${\bf k}$.
For  a non-zero rational number $a$, we define an element 
$[a]=[a]_n$ of ${\bf L}_n(\QQ)$ by $[a]_n=\varGamma^{(n)}(a 1_{n})\varGamma^{(n)}.$
For each integer $m$ define an element $T(m)$ of ${\bf L}_n$
by 
$$T(m)=\sum_{d_1,\ldots,d_n,e_1,\ldots,e_n}\varGamma^{(n)}(d_1 \bot \cdots \bot d_n \bot e_1 \bot \cdots \bot e_n)\varGamma^{(n)},$$ where
$d_1,\ldots,d_n,e_1,\ldots,e_n$ run over all positive integers satisfying
$$d_i|d_{i+1}, \ e_{i+1}|e_i \ (i=1,\ldots,n-1), d_n|e_n,d_ie_i=m \
(i=1,\ldots,n).$$ Furthermore, for $i=1,\ldots,n$ and a prime number $p$
 put
$$T_i(p^2)=\varGamma^{(n)}(1_{n-i} \bot p1_i \bot p^21_{n-i} \bot p1_i)\varGamma^{(n)}.$$
As is well-known, ${\bf L}_n(\QQ)$ is generated over $\QQ$ by $T(p),
T_i(p^2) \ (i=1,\ldots,n),$ and $[p^{-1}]_n$ for all $p$. We note that 
$T_n(p^2)=[p]_n$. 
Let $F$ be a Hecke eigenform in $M_{{\bf k}}(\varGamma^{(n)})$.  As is well-known, $\QQ(F)$ is a totally real algebraic number field of finite degree.  
   Let ${\bf L}_{n,p}={\bf L}( \varGamma^{(n)}, \GSp_n({\QQ})^+  \cap \GL_{2n}({\ZZ}[p^{-1}]))$ be the Hecke algebra associated with the pair
   $(\varGamma^{(n)},\GSp_n({\QQ})^+  \cap  \, \GL_{2n}({\ZZ}[p^{-1}])).$
 ${\bf L}_{n,p}$ can be considered as a subalgebra of ${\bf L}_{n}(\QQ),$ and is generated over ${\QQ}$ by $T(p)$ and $T_i(p^2) \ (i=1,2,\ldots, n),$ and $[p^{-1}]_n.$

We write  
$\Gamma_{\CC}(s)=2(2\pi)^{-s}\Gamma(s)$. 
Let 
$$f(z)=\sum_{m=0}^{\infty} a(m,f){\bf e}(mz)$$
 be a primitive form in $S_k(\SL_2(\ZZ))$, that is 
let $f$ be a Hecke eigenform whose first Fourier coefficient is $1$. 
For a prime number $p$ let $\beta_{1,p}(f)$ and $\beta_{2,p}(f)$ be complex numbers such that $\beta_{1,p}(f)+\beta_{2,p}(f)=a(p,f)$ and
$\beta_{1,p}(f)\beta_{2,p}(f)=p^{k-1}$.
Then for a Dirichlet character $\chi$ we define  Hecke's $L$-function $L(s,f)$ twisted by $\chi$ defined as
$$L(s,f,\chi)=\prod_p\bigl((1-\beta_{1,p}(f)\chi(p)p^{-s})(1-\beta_{2,p}(f)\chi(p)p^{-s})\bigr)^{-1}.$$
 Let $f$ be a primitive form in ${S}_{k}(\SL_2(\ZZ)).$  Then Shimura \cite{Shimura77} showed that there exist two complex numbers 
$c_{\pm}(f)$, uniquely determined up to ${\QQ}(f)^{\times}$ multiple such that the following holds:
\bigskip

\noindent
(AL) \quad The value $\displaystyle {\Gamma_{\CC}(l) \sqrt{-1}^l L(l,f,\chi) \over \tau(\chi)c_s(f)}$
 belongs to  ${\QQ}(f)(\chi)$  for any integer $1 \le l \le k-1$ and a Dirichlet character $\chi,$ where $\tau(\chi)$ is the Gauss sum of $\tau$, and  $s=s(l,\chi)=+$ or $-$ according as $\chi(-1)=(-1)^l$ or $(-1)^{l-1}.$

\bigskip

%We call $c_{\pm}(f)$ Deligne's period. %
We write
 $${\bf L}(l,f,\chi;c_s(f))={\Gamma_{\CC}(l)  \sqrt{-1}^l L(l,f,\chi) \over \tau(\chi)c_s(f)}.$$
We sometimes write $c_{s(\l,\chi)}(f)=c_{s(l)}(f)$ and ${\bf L}(l,f,\chi;c_{s(l,\chi)}(f))={\bf L}(l,f;c_{s(l)}(f))$ if $\chi$ is the principal character.
We note that the above value depends on the choice of $c_{\pm}(f),$ and to formulate our conjecture we choose them in an appropriate way. We will discuss this topic  in Section 7. However, if $(\chi \eta)(-1)=(-1)^{l+m}$, then $s:=s(l,\chi)=s(m,\eta)$ and the ratio $\displaystyle {{\bf L}(l,f,\chi;c_s(f)) \over {\bf L}(m,f,\eta;c_s(f))}$ does not depend on the choice of $c_s(f)$, which will be denoted by $\displaystyle {{\bf L}(l,f,\chi) \over {\bf L}(m,f,\eta)}$. 

For two positive integers $l_1,l_2 \le k-1$ and Dirichlet characters $\chi_1,\chi_2$ such that $\chi_1(-1)\chi_2(-1)=(-1)^{l_1+l_2+1},$ the value 
$${\Gamma_{\CC}(l_1)\Gamma_{\CC}(l_2)L(l_1,f,\chi_1)L(l_2,f,\chi_2) \over \sqrt{-1}^{l_1+l_2+1}\tau((\chi_1 \chi_2)_0) (f , f )}$$
belongs to ${\QQ}(f)(\chi_1, \chi_2)$, where
$(\chi_1\chi_2)_0$ is the primitive character associated with $\chi_1\chi_2$.
 (cf. \cite[Theorem 4]{Shimura76}.) We denote this value by ${\bf L}(l_1,l_2;f;\chi_1,\chi_2).$
In particular, we put
\[{\bf L}(l_1,l_2;f)={\bf L}(l_1,l_2;f;\chi_1,\chi_2)\]
if $\chi_1$ and $\chi_2$ are  the principal character. This value does not depend upon the choice of ${c_{\pm}(f)}$. 

Let $f$ be a primitive form in $S_k(\SL_2(\ZZ))$. 
Let $f_1,\ldots,f_d$ be a basis of $S_k(\SL_2(\ZZ))$ consisting of primitive forms with $f_1=f$ and 
let $\frkD_f$ be the ideal of $\QQ(f)$ generated by all $\prod_{i=2}^d (\lambda_{f_i}(T(m))-\lambda_f(T(m)))$'s ($m \in \ZZ_{>0}$).
Moreover, for each prime number $p$, let  $\frkD_f^{(p)}$ be the ideal of $\QQ(f)$ generated by all $\prod_{i=2}^d (\lambda_{f_i}(T(m))-\lambda_f(T(m)))$'s ($m \in \ZZ_{>0}, \ p \nmid m$). Clearly we have $\frkD_f^{(p)} \subset \frkD_f$.
  For a prime ideal $\frkp$ of an algebraic number field, let $p_{\frkp}$ be a prime number such that $(p_{\frkp})=\ZZ \cap \frkp$.

 Let $F$ be  a Hecke eigenform in $M_{\bf k}(\varGamma^{(n)})$, and for a prime number $p$ we take the $p$-Satake parameters 
 $\alpha_0(p),\alpha_1(p),\ldots,\alpha_n(p)$ of $F$ so that 
$$\alpha_0(p)^2\alpha_1(p) \cdots \alpha_n(p)=p^{k_1+\cdots +k_n-n(n+1)/2}.$$
We define  the polynomial $L_p(X,F,{\rm Sp})$ by
$$L_p(X,F,{\rm Sp})=(1-\alpha_0(p)X)\prod_{r=1}^n\prod_{1 \le i_1<\cdots<i_r}(1-\alpha_0(p)\alpha_{i_1}(p)\cdots \alpha_{i_r}(p)X)$$
and the spinor $L$ function $L(s,F,\mathrm {Sp})$ by
$$L(s,F,\mathrm {Sp})=\prod_p L_p(p^{-s},F,{\rm Sp})^{-1}.$$
We note that  $L(s,f,{\rm Sp})$ is Hecke's $L$-function $L(s,f)$
if $f$ is a primitive form. In this case we write $L_p(s,f)$ for $L_p(s,f,{\rm Sp})$.
We also define the polynomial $L_p(X,F,{\rm St})$ by
$$(1-X)\prod_{i=1}^n (1-\alpha_i(p)X)(1-\alpha_i(p)^{-1}X)$$
and 
the standard $L$-function $L(s,F,{\rm St})$ by 
$$L(s,F,{\rm St})=\prod_p L_p(p^{-s},F,{\rm St})^{-1}.$$
For a Hecke eigenform $F \in S_k(\varGamma^{(r)})$ put
\[{\bf L}(s,F,\St)=\Gamma_{\CC}(s)\prod_{i=1}^r \Gamma_{\CC}(s+k-i){ L(s,F,\St) \over (F, F)}.\]
We note that for a positive integer $m \le k-r$
\[{\bf L}(m,F,\St)=A_{r,k,m} {L(m,F,\St) \over \pi^{r(k+m)+m-r(r+1)/2} (F,  F)}\]
with a non-zero rational number $A_{r,k,m}$ such that 
$\ord_p(A_{r,k,m})=0$ if $p \ge 2k$.

\section{Galois representations}
In this section we recall the results from \cite{Chenevier-Lannes19} for Galois representations 
attached to Siegel cusp forms and observe some properties. Fix an algebraic closure $\bar \QQ$ of $\QQ$ in $\CC$. 
From now on, we fix an isomorphism $\iota_p:\CC\simeq \bar \QQ_p$ and
 a pair $(i_\infty,i_p)$ of two embeddings 
$i_\infty:\bar \QQ \to \CC$ and $i_p:\bar \QQ \to\bar \QQ_p$ such that they commute with $\iota_p$.
From now on we put $G_{\QQ}=\Gal(\bar \QQ/\QQ)$.  We also put $G_{\QQ_{\ell}}=\Gal(\bar \QQ_\ell/\QQ_{\ell})$ for each prime number $\ell$, and let $I_{\ell}$ denote  the inertia subgroup of $G_{\QQ_{\ell}}$. We sometimes identify $I_{\ell} \subset G_{\QQ_{\ell}}$ with their images in $G_{\QQ}$ and write  $G_{\QQ_{\ell}}$ as $D_l$.
Let $\rho_{\bf k}$ be the algebraic representation of $\GL_n(\CC)$ with the highest weight ${\bf k}=(k_1,\ldots,k_n)$
satisfying the condition 
\begin{equation}\label{wcond}
k_1\ge \cdots \ge k_n\ge n+1. 
\end{equation}
Let $F$ be a Hecke eigenform in $S_{\bf k}(\varGamma^{(n)})$ and $\pi=\pi_F$ be the corresponding 
automorphic cuspidal representation of $\mathrm{Sp}_n(\AAA_\QQ)$ where $\AAA_\QQ$ is the ring of adeles of $\QQ$. Throughout this section, we assume Arthur's classification \cite{Arthur} (Arthur-Langlands conjecture) for $\mathrm{Sp}_n(\AAA_\QQ)$ explained 
in \cite[Appendix A]{A-C-I-K-Y23}. (See also \cite[Section 6.4, p.166] {Chenevier-Lannes19}.)
The infinitesimal character of $\pi_\infty$ is indexed by the set 
$${\rm Weights}(\pi):=\{k_1-1,k_2-2,\ldots,k_n-n,0,-(k_n-n),\ldots, -(k_1-1)\}.$$
By the assumption, it consists of distinct $(2n+1)$ integers. 
By Arthur's classification, we can write the global Arthur parameter $\psi(\pi,{\rm St})$ as 
\begin{equation}%\label{SR}
\psi(\pi,{\rm St})=\ds\bigoplus_{i=1}^r \pi_i[d_i]
\end{equation}
where 
\begin{enumerate}
\renewcommand{\labelenumi}{(\roman{enumi})}
\item for each $i$ $(1\le i \le r)$, $\pi_i$ is an irreducible unitary cuspidal automorphic self-dual representation of $\mathrm{PGL}_{n_i}(\AAA_\QQ)$ 
which is unramified at any finite place of $\QQ$,
\item $\ds\sum_{i=1}^r n_id_i=2n+1$. 
\end{enumerate}
For such a $\pi_i$, we denote by $w(\pi_i)$ the motivic weight of $\pi$ and ${\rm Weights}(\pi_i)$ the set of weights of 
the infinitesimal character of $\pi_{i,\infty}$ (see Section 8.2.6, p.\ 195 of \cite{Chenevier-Lannes19}). We also denote by 
$\rho_{\pi_i}=\rho_{\pi_i,p}:G_{\QQ} \to \GL_{n_i}(\bar \QQ_p)$  
the corresponding $p$-adic Galois representation attached to $\pi_i$ (cf. Theorem 2.1.1, p.\ 538 of \cite{Barnet-Lamb-Gee-Taylor14}). Let $(\rho,V)$ be 
a crystalline $p$-adic representation of $G_{\QQ_p}$. 
Let $\{{\rm Fil}^i\}_{i\in\ZZ}$ be the decreasing filtration on $D_{{\rm crys}}(V):=
(V\otimes_{\QQ_p}B_{{\rm crys}})^{G_{\QQ_p}}$. Then the Hodge-Tate weights $\mathrm {HT}(\rho)=\mathrm {HT}(V)$ are covered by the multi-set over the set 
$$\{j\in \ZZ\ |\ {\rm Fil}^{-j}(D_{{\rm crys}}(V))/{\rm Fil}^{-j+1}(D_{{\rm crys}}(V))\neq 0\}$$
such that for each $j\in \ZZ$ in the above set, the multiplicity at $j$ is given by 
the dimension of ${\rm Fil}^{-j}(D_{{\rm crys}}(V))/{\rm Fil}^{-j+1}(D_{{\rm crys}}(V))$. 
For example, let $\chi=\chi_p$ be the $p$-adic cyclotomic character  so that $\chi_p(\mathrm{Frob}_{\ell})=\ell$ for any prime number $\ell \not=p$. Then $\mathrm{HT
}(\chi_p)=\{1\}$. For each $a\in \ZZ$, we write $\chi_p^a$ for the $a$-th power of $\chi_p$ and 
$\bar{\chi}^a_p$ for the reduction of $\chi^a_p$ modulo a prime over $p$. 
\begin{definition}{\rm (}\cite[Definition 7.3 ]{Brown07} or 
\cite[p.670]{Diamond-Flach-Guo04}{\rm )}
\label{def.short-crystalline}
Let $(\rho,V)$ be a crystalline $p$-adic representation of $G_{\QQ_p}$. 
We say that $V$ is short crystalline if 
\begin{itemize}
\item [(1)] ${\rm Fil}^{0}(D_{{\rm crys}}(V))=D_{{\rm crys}}(V)$ and ${\rm Fil}^p(D_{{\rm crys}}(V))=0$; 
\item [(2)] if $V'$ is a non-zero quotient of $V$, then $V'\otimes_{\QQ_p}\QQ_p(p-1)$ is ramified. 
\end{itemize}

%We remark that 
%\cite[Definition 7.3 ]{Brown07} or one in \cite[p.670]{Diamond-Flach-Guo04} is given with respect to 
%the geometric Frobenius element. Therefore, the above definition has the different shape but essentially the same. 
\end{definition} 
\begin{remark}\label{rem.short-crystalline1}
Let $V$ be a crystalline representation. The condition (1) above is 
true if and only if ${\rm HT}(V)\subset [-(p-1),0]$. 
If there exists a non-zero quotient $V'$ such that $V'\otimes_{\QQ_p}\QQ_p(p-1)$ is unramified, 
then ${\rm HT}(V)$ has to contain $-(p-1)$. Hence if ${\rm HT}(V)\subset [-(p-2),0]$, then $V$ is short crystalline. The condition (2) above seems to be true if $V$ comes from a reasonable motive over $\QQ$. 
However, no substantial result except for a few cases (see \cite{Ghate-Vatsal11}) is available.
\end{remark}
Let $K$ be a number field in $\bar\QQ$ containing the Hecke field $\QQ(F)$ of $F$ and  $\frkp$ a prime ideal of $K$ dividing a rational prime $p$.   We say that a Galois representation $\rho:G_{\QQ} \longrightarrow \GL_m(K_\frkp)$ is crystalline (resp. short crystalline) at $p$ if $\rho|{D_p}$ is  crystalline (resp. short crystalline).
By \cite[Corollary 8.2.19, p. 201]{Chenevier-Lannes19}, we have the following fact:
\begin{theorem}\label{th.Galois-st}
Keep the notation and the assumptions as above. Then, there exists a continuous 
semi-simple Galois representation 
$$\rho_{F,\St}:G_{\QQ} \to \GL_{2n+1}(K_\frkp)$$
such that 
\begin{enumerate}
\item $\rho_{F,\St}$ is unramified at any  prime number $\ell \not=p$. Moreover, we have  
\begin{align*}
&\det(I_{2n+1}-X \rho_{F,\St}({\rm Frob}_{\ell}^{-1}) )          \\
&=\prod_{i=1}^r\prod_{j=0}^{d_i-1}
\det(I_{n_i}-{\ell}^{j+\frac{1-d_i}{2}}X\rho_{\pi_i}({\rm Frob}_{\ell}^{-1}))\\
&=L_{\ell}(X,F,\St),
\end{align*}
where $\mathrm{Frob}_{\ell}$ is the arithmetic Frobenius at $\ell$;
\item  $\rho_{F,\St}$  is crystalline at $p$ of Hodge-Tate weights 
$\{\lambda \ |\ \lambda\in {\rm Weights}(\pi)\}\subset [-k_1+1, k_1-1]$. 
In particular, if $p-1>2k_1-2$, then $(\chi^{1-k_1} \otimes \rho_{F,\St})$ is short crystalline at $p$.  
\end{enumerate}
\end{theorem} 
\begin{proof} As mentioned, it follows from \cite[Corollary 8.2.19]{Chenevier-Lannes19} except for the following things. 
A priori, the image takes the values in $\bar \QQ_p$ with a fixed embedding $K_\frkp \hookrightarrow \bar \QQ_p$. 
Applying \cite[TH\'EOR\`EME 2]{Carayol91} to $\rho_{F,\St}$, up to isomorphism, we may have a  
Galois representation with values in $\GL_{2n+1}(K_\frkp)$ as desired. 

The fact that $\chi^{1-k_1} \otimes \rho_{F,\St}$ satisfies the short crystalline condition follows from Remark \ref{rem.short-crystalline1} 
since the Hodge-Tate weights of $\chi^{1-k_1} \otimes \rho_{F,\St}$ are given by 
distinct, decreasing $(2n+1)$ integers in $[-(2k_1-2),0]$. 
\end{proof}
Let ${\bf k}=(k_1,\ldots,k_r,k_{r+1},\ldots,k_n)$ and ${\bf l}=(k_1,\ldots,k_r)$ with $k_1 \ge \cdots \ge k_r \ge k_{r+1} \ge\cdots \ge k_n$.
For a Hecke eigenform $F \in M_{\bf k}(\varGamma^{(n)})$ such that $\Phi_r^n(F)$ is a Hecke eigenform in $S_{\bf l}(\varGamma^{(r)})$,
we define the Galois representation $\rho_{F,\St}$ for $F$ as
\[\rho_{F,\St}=\rho_{\Phi_r^n(F),\St} \oplus \bigoplus_{i=r+1}^n (\chi^{k_i-i} \oplus \chi^{i-k_i}).\]
We note that $\rho_{F,\St}$ is unramified at any prime number $\ell \not=p$ and
\[\det (1_{2n+1}-\rho_{F,\St}(\mathrm{Frob}_{\ell}^{-1})X)=L_{\ell}(X,F,\St).\]
For $n=1,2$, one can also construct the $\frkp$-adic Galois representation attached to $F$ which 
takes the values in $\GL_{2^n}(K_\frkp)$. For $n=1$ we refer \cite{Deligne68}, \cite{Ribet}, \cite{Kato04} and for $n=2$ 
we refer \cite{Laumon05}, \cite{Wei09}:
\begin{theorem}\label{th.Galois-spin}
Put ${\bf k}=k\ge 2$ if $n=1$ and ${\bf k}=(k+j,k),\ k\ge 3, j\ge 0$ if $n=2$. 
Then, there exists a continuous 
semi-simple Galois representation 
$$\rho_{F}:G_{\QQ}\to \GL_{2^n}(K_\frkp)$$
such that 
\begin{itemize}
\item [(1)] $\rho_{F}$ is unramified at any prime number $\ell \not=p$, and 
\[\det(I_{2^n}-X \rho_{F}(\mathrm {Frob}_{\ell}^{-1}))=L_{\ell}(X,F,\mathrm{Sp});\]
\item [(2)] $\rho_{F}$ is crystalline at $p$ of Hodge-Tate weights given by 
$$\begin{array}{cc}
\{-(k-1),0\} & (n=1), \\
\{-(2k+j-3),-(k+j-1),-(k-2),0\} & (n=2). 
\end{array}
$$
In particular, if $p-1>w_{{\rm Sp}}$, then $\rho_{F}$ is short crystalline at $p$, where $w_{{\rm Sp}}=k-1$ or $2k+j-3$ according as $n=1$ or $2$; 
\item [(3)] if $n=1$, $\rho_{F}$ is absolutely irreducible. 
\end{itemize}
\end{theorem}
\begin{remark}
In the case of $n=1$, our representation $\rho_F$ is that in \cite{Kato04}, and it is the dual of those defined in \cite{Deligne68}, \cite{Ribet}. In the case of $n=2$, $\rho_F$ is  the dual of those defined in \cite{Laumon05}, \cite{Wei09}.
\end{remark}
 Let us discuss  the irreducibility of $\rho_{F}$ for $n=2$. 
Let $\Pi=\Pi_F$ be the cuspidal representation of $\GSp_2(\AAA_\QQ)$ associated to $F$. 
As discussed in \cite[Section 2]{Kim-Wakatsuki-Yamauchi20}, since $F$ is of level one, $\Pi$ falls into either of the following cases:
\begin{enumerate}
\item $\Pi$ is CAP and it has to be $j=0$;
\item $\Pi$ is a symmetric cubic lift from a cuspidal representation of $\GL_2(\AAA_\QQ)$;
\item $\Pi$ is genuine. 
\end{enumerate}
In particular, $\Pi$ cannot be endoscopic because $F$ is holomorphic and of level one.  
\begin{theorem}\label{th.irred} 
Keep the notation as above. Suppose that $j\ge 1,\ k\ge 4$, and 
$p-1\ge {\rm max}\{7,2k+j-1\}$. 
If $\Pi$ is not CAP, then $\rho_{F}$ is absolutely irreducible.
\end{theorem}
\begin{proof}
We may assume that $F$ is not CAP, then $\Pi$ is tempered at any finite place of $\QQ$ by \cite{Laumon05} and 
\cite{Wei09}. 
Suppose that $\rho_{F}$ is not absolutely irreducible. By using self-duality of $\Pi$ and 
the fact that $F$ is of level one, we have 
\begin{itemize}
\item $\rho_{F}\simeq \chi^{-(2k+j-3)}_p\oplus \chi^{-(k+j-1)}_p\oplus 
\chi^{-(k-2)}_p\oplus  \textbf{1}$;
\item $\rho_{F}\simeq \chi^{-(2k+j-3)}_p\oplus \textbf{1}\oplus \tau_1$;
\item $\rho_{F}\simeq \chi^{-(k+j-1)}_p\oplus \chi^{-(k-2)}_p\oplus \tau_2$;
\item $\rho_{F}\simeq  \tau_3\oplus \tau_4$, 
\end{itemize}
where each $\tau_i\ (1\le i \le 4)$ is an irreducible representation of $G_{\QQ}$ to ${\rm GL}_2(K_\frkp)$ which 
is unramified outside $p$ and crystalline at $p$. Except for the last case, by observing Satake parameters at sufficiently large 
$l$ in conjunction with the Ramanujan bound that is $4l^{\frac{2k+j-3}{2}}$, $\Pi$ cannot be tempered and it contradicts the assumption. 

For the remaining case, put $\tau=\tau_3$ or $\tau_4$ and let $\overline{\tau}$ be 
the residual representation of $\tau$. By matching of Hodge-Tate weights and the assumption of 
weights $k,j$, up to the twist of a power of $p$-adic cyclotomic character, 
we may assume that the Hodge-Tate of $\tau$ is of form $\{-a,0\}$ with $2\le a\le 2k+j-3<p-1$. Since $\tau$ is unramified 
outside $p$ and crystalline at $p$, it cannot be dihedral, otherwise $\tau$ must have a non-trivial conductor. 
Therefore, ${\rm Sym}^2\tau$ is irreducible and being non-dihedral implies more strongly that ${\rm Sym}^2\tau|_{\Gal(\bar \QQ/\QQ(\zeta_p))}$ 
is irreducible. 
Further, $\overline{\tau}|_{G_{\QQ_p}}$ cannot be a twist of the form 
$\left(\begin{array}{cc}
\overline{\chi}^{-1}_p & \ast \\
0 & 1
\end{array}\right)
$ by \cite[Theorem 2.4 and Theorem 2.5]{Edixhoven92}. Here we use the assumption $2\le a$ and also $a<p-1$. 
Applying \cite[Theorem 1.2]{Calegari12} (note that the assumption there on $p$  
is fulfilled in our setting), $\tau$ has to be odd. Then as explained right after \cite[Theorem 1.1]{Calegari12}, 
$\tau$ is modular. Hence $F$ has to be endoscopic but it is impossible as we have classified the possible types of $\Pi$. 

Summing up, we have a contradiction in either case when $\rho_{F}$ is not absolutely irreducible.  
\end{proof}
\begin{theorem}\label{th.irred2}
Let $1\le n \le 5$. 
Let $\Pi$ be regular, algebraic, cuspidal, self-dual automorphic representation of $\mathrm{PGL}_n(\AAA_\QQ)$ which is unramified 
at any finite place of $\QQ$. 
Further, we assume that if $n=4$, ${\rm Weights}(\pi)=\{a,b,-b,-a\}\subset \frac{1}{2}+\ZZ$ with 
$b>\frac{1}{2}, a-b>1$ and if $n=5$, ${\rm Weights}(\pi)=\{m_1,m_2,0,-m_2,-m_1\}\subset \ZZ$ with $m_1>m_2+1>2$.  For 
each rational prime $p$, let $\rho_{\Pi}=\rho_{\Pi,p}:G_{\QQ}\to 
\GL_n(K_\frkp)$ be the $\frkp$-adic representation 
attached to $\Pi$. Then $\rho_{\Pi,p}$ is absolutely irreducible.
\end{theorem}
\begin{proof}
Let us first remark that by \cite{Ca1,Ca2}, $\Pi$ is tempered everywhere. 
It is well-known when $n\le 3$. First we consider the case $n=4$. 
By self-duality, since $\Pi$ is level one, $\Pi$ descends to an everywhere unramified cuspidal representation of $\GSp_2(\AAA_\QQ)$ or $\mathrm{GO}(2,2)(\AAA_\QQ)$. In the former case, by \cite{Wei08}, there exists a cuspidal representation $\Pi'$ of $\GSp_2(\AAA_\QQ)$ such that 
the finite parts of $\Pi_{f}$ and $\Pi'_f$ are isomorphic and $\Pi'$ is associated to a 
Hecke eigen Siegel cusp form $F$ of level 1 with the weight $(k_1,k_2)=(a+b+1,a-b+2)$. 
Since $\Pi$ is cuspidal and tempered everywhere, $\Pi$ is neither CAP nor endoscopic. 
Put $j=k_1-k_2=2b-1\ge 1$ and $k=k_2\ge 4$. Applying Theorem \ref{th.irred} to $F$, we have the claim for the former case. 
For the latter case, $\Pi$ can be described as the convolution product of two cuspidal representations $\pi_1,\pi_2$ of $\GL_2(\AAA_\QQ)$ 
associated to two elliptic cusp forms of level one. Since $\Pi$ is cuspidal, they never 
be isomorphic each other.
Further, $\pi_1,\pi_2$ cannot be dihedral. Therefore, the convolution 
$\rho_{\pi_1,p}\otimes \rho_{\pi_2,p}$ has to be irreducible.  

Next we assume $n=5$. The self-duality shows that $\Pi$ can be descended to a cuspidal representation $\Pi_1$ 
of $\GSp_2(\AAA_\QQ)$ corresponding to the exceptional isomorphism $\PGSp_2(\CC)\simeq \mathrm{SO}(5)(\CC)$ (cf. Section 8 of \cite{Kim-Yamauchi17}). 
Since $\Pi$ is tempered, so is $\Pi_1$. Further, $\Pi_1$ is unramified everywhere up to a quadratic twist. 
By twisting if necessary, we may assume that $\Pi_1$ is unramified everywhere.  
It follows from the cuspidality and the temperedness of $\Pi$ that $\Pi_1$ is neither CAP nor endoscopic. Similarly, as before, 
there exists a cuspidal representation $\Pi'$ of $\GSp_2(\AAA_\QQ)$ such that 
the finite parts $\Pi_{1,f}$ and $\Pi'_f$ are isomorphic and $\Pi'$ is associated to a 
Hecke eigen Siegel cusp form $F$ of level 1 with the weight $(k_1,k_2)=(m_1+1,m_2+2)$ satisfying 
$j:=k_1-k_2\ge 1$ and $k=k_2\ge 4$ by the assumption.  
The proof of Theorem \ref{th.irred} shows that $\Pi'$ is either genuine or a symmetric cubic lift of 
a regular cuspidal representation $\pi$ of $\GL_2(\AAA_\QQ)$ which is unramified everywhere (this can be checked by observing 
local conductors). 
In the latter case, $\Pi$ is a symmetric fourth lift of $\pi$. Since the connected component of 
the Zariski closure of the image $\rho_{\pi,p}$ contains $\SL_2$, $\rho_{\Pi,p}={\rm Sym}^4(\rho_{\pi,p})$ 
is irreducible. 
In the former case, the Zariski closure of the image $\rho_{\Pi',p}$ contains $\mathrm{Sp}_2$. 
Hence $\rho_{\Pi,p}=\wedge^2(\rho_{\Pi',p})/\rho_{\omega',p}$ is irreducible where $\omega'$ is the 
central character of $\Pi'$ and $\rho_{\omega',p}:G_{\QQ}\to \GL_1(K_\frkp)$ is the corresponding 
character.  
\end{proof}

Let $G$ be a compact group and $L$ a field of characteristic $0$. Let $\rho:G \to \GL_m(L)$ be a semi-simple representation, and $\rho^\vee$ the contragredient representation of $\rho$. If $\rho$ is self-dual, then there exists an element $T \in \GL_m(K)$ such that $\rho^\vee=T^{-1}\rho T$. Moreover,  if $\rho$ is 
absolutely irreducible, then $T$ is symmetric or alternating. We say that the sign of $\rho$ is $+1$ (resp. $-1$) if $T$ is symmetric (resp. alternating) (cf. \cite{Bellaiche-Chenevier11}).

\begin{theorem} \label{th.sign-Galois-st}
Let $F$ be a Hecke eigenform in $S_{\bf k}(\varGamma^{(n)})$ with ${\bf k}=(k_1,\ldots,k_n)$ and 
let  $\psi_F$
be the associated $A$-parameter. Suppose that $\psi_F$ is stable tempered or  that $\psi_F=\tau_1 \boxplus 1$ with $\tau_1$ an irreducible unitary cuspidal self-dual automorphic representation of $\GL_{2n}(\AAA_\QQ)$.  Let $\rho$ be the Galois representation attached to $\psi_F$ in the former case, and to $\tau_1$ in the latter case.
Then $\rho$ is self-dual. Moreover, for any irreducible factor $\rho'$ of $\rho$ such that $\rho'$ is self-dual, the sign of $\rho'$  is $+1$.
\end{theorem}
\begin{proof}
The first part is due to \cite[Theorem 2.1.1]{Barnet-Lamb-Gee-Taylor14}. The second part is due to \cite[Corollary 1.3]{Bellaiche-Chenevier11}.
\end{proof}
\begin{remark} The second part  of the above theorem is trivial if the dimension of $\rho'$ is odd.
\end{remark}

For our later purpose, we provide the following lemma.
\begin{lemma}\label{lem.Ribet's-criterion}
Let $f(z)=\sum_{m=1}^{\infty} a(m,f){\bf e}(mz)$ be a primitive form in $S_k(\SL_2(\ZZ))$. Let $\frkp$ be a prime ideal of $\QQ(f)$ of degree one with  $p:=p_\frkp>2k$.
Let  $\rho_f:G_{\QQ} \to \GL_2(K_\frkp)$ be  the Galois representation attached  as stated before.  Suppose that $\frkp$ divides none of  $\zeta(1-k), \ a(2,f), \ a(2,f)+2^{k/2}, a(2,f)-2^{k/2}$. Then, $\rho_f(\Gal(\bar \QQ/\QQ(\zeta_{p^{\infty}})))$ contains $\SL_2(\ZZ_p)$ with a suitable choice of a lattice of the representation space $V_f$ of $\rho_f$, where $\QQ(\zeta_{p^{\infty}})=\cup_{l=1}^{\infty} \QQ(\zeta_{p^l})$.
\end{lemma}
\begin{proof}
By  the proof of \cite[Lemma 5.3]{Ribet75} and \cite[Remark 5]{Ribet75}, $\bar \rho_f(G_{\QQ})$ contains $\SL_2(\FF_p)$ 
(see also the divisibility conditions at line 3 and line 8 of \cite[p.265]{Ribet75}).  Since 
$$\bar \rho_f(G_{\QQ})/\bar \rho_f(\Gal(\bar \QQ/\QQ_{\zeta(p^{\infty}}))$$ is abelian, 
$\bar \rho_f(\Gal(\bar \QQ/\QQ_{\zeta(p^{\infty}}))$ also contains the commutator 
subgroup $[\SL_2(\FF_p),\SL_2(\FF_p)]$ which is, in fact, $\SL_2(\FF_p)$. Thus the assertion follows from \cite[IV, 3.4, Lemma 3]{Serre68}.
\end{proof}
\section{Main results}

Now we state Harder's conjecture. 
\begin{conjecture}(Harder's conjecture \cite{Harder03})
\label{conj.Harder2}
Let $k$ and $j$ be non-negative integers such that $j$ is even and $k \ge 3$. 
Let $f(z)=\sum a(m,f){\bf e}(mz)\in S_{2k+j-2}(\SL_2(\ZZ))$ be a primitive form, and let  $\frkp$ be a prime ideal of $\QQ(f)$. Fix a period $c_{\pm}(f)$ appropriately. Suppose that $p_\frkp >  2k+j-2$ and that $\frkp$
divides ${\bf L}(k+j,f;c_{s(k+j)}(f))$.
Then, there exists a Hecke eigenform
$F \in S_{(k+j,k)}(\varGamma^{(2)})$, and a prime $\frkp' \mid\frkp$ in (any
field containing) $\QQ(f)\QQ(F)$ such that, for all primes $\ell$
\begin{align*}
L_\ell(X,F,{\rm Sp}) \equiv L_\ell(X,f)(1-\ell^{k-2}X)(1-\ell^{k+j-1}X)
\mod{\frkp'}. \tag{H}
\end{align*}
In particular, 
$$\lambda_F(T(\ell))\equiv \ell^{k-2}+\ell^{j+k-1}+a(\ell ,f) \mod{\frkp'}.$$
\end{conjecture}
We say that (H) is Harder's congruence for $f$.
We also propose a modification of the above conjecture, which can be formulated independent of the choice of $c_{\pm}(f)$.
\begin{conjecture}\label{conj.modified-Harder}
Let $k$ and $j$ be non-negative integers such that $j$ is even and $k \ge 3, j \ge 4$. Let $f$ be as in Conjecture \ref{conj.Harder2}. Suppose that a prime ideal $\frkp$ of $\QQ(f)$ satisfies $p_\frkp >2k+j-2$ and that $\frkp$ divides ${\displaystyle {\bf L}(k+j,f) \over \displaystyle {\bf L}(k_j,f)}$, where $k_j=k+j/2$ or $k+j/2+1$ according as $j \equiv 0 \mod 4$ or $j \equiv 2 \mod 4$. Then the same assertion as Conjecture \ref{conj.Harder2} holds.
\end{conjecture}
Conjecture \ref{conj.Harder2} does not concern the congruence between the Hecke eigenvalues of two Hecke eigenforms, and it is not easy to confirm it. To treat the conjecture more accessibly, we reformulate it 
in the case $k$ is even.   
Let $k$ and $n$ be positive even integers.
For a primitive form $f$ in $S_{2k-n}(\SL_2(\ZZ))$ let $\scri_n(f)$ be the Duke-Imamoglu-Ikeda lift of $f$ to $S_k(\varGamma^{(n)})$. That is, $\scri_n(f)$ is a Hecke eigenform such that
\[L(s,\scri_n(f),\St)=\zeta(s)\prod_{i=1}^n L(s+k-i-1,f).\]
In the case $n=2$, $\scri_2(f)$ is called the Saito-Kurokawa lift.  We note that the Klingen-Eisenstein series $[\scri_2(f)]^{(k,k,l,l)}(Z)$ belongs to $M_{(k,k,l,l)}(\varGamma^{(4)})$ if  $k \ge l \ge 6$ (cf. \cite[Proposition 2.1]{A-C-I-K-Y23}).
We also have the following lift, which is a special case of \cite[Theorem 4.3]{A-C-I-K-Y23}.
\begin{theorem} \label{th.atobe1}
 Let $F$ be a Hecke eigenform in $S_{k+l-4,k-l+4}(\varGamma^{(2)})$.
 Then there exists a Hecke eigenform $\scra^{(I)}_4(F)$ in $S_{(k,k,l,l)}(\varGamma^{(4)})$ such that 
\[L(s,\scra^{(I)}_4(F),\St)=\zeta(s)L(s+k-1,F,{\rm Sp})L(s+k-2,F,{\rm Sp}).\]
\end{theorem}
Let ${\bf k}=(k_1,\ldots,k_n)$ be a non-increasing sequence of non-negative integers. 
Let $F$ and $G$ be Hecke eigenforms in $M_{{\bf k}}(\varGamma^{(n)})$.
Let $K$ be an algebraic number field containing  $\QQ(F)$ and
 $\frkp$ a prime ideal of $K$. We say that $G$ is Hecke congruent to $F$ modulo $\frkp$ if there is a prime ideal $\frkp'$ of the composite field $K\cdot\QQ(G)$ lying above $\frkp$ such that
\[\lambda_G(T) \equiv \lambda_F(T) \mod {\frkp'} \text{ for any } T \in {\bf L}_n^{({\bf k})}, \]
and write
\[G \equiv_{\mathrm{ev}} F \mod \frkp.\] 
Then the following conjecture is a special case of \cite[Conjecture 4.5]{A-C-I-K-Y23}.
\begin{conjecture} \label{conj.main-conjecture}
Let $k$ and $j  $ be positive even integers such that $k \ge 4$ and $j  \equiv 0 \mod 4$, and put 
${\bf k}=\Bigl({j \over 2}+k, {j \over 2}+k,{j \over 2}+4,{j \over 2}+4 \Bigr)$.
Let $f(z) \in S_{2k+j-2}(\SL_2(\ZZ))$ be a primitive form.  Let $\frkp$ be a prime ideal of $\QQ(f)$ such that $p_{\frkp} > 2k+j-2$ and suppose that $\frkp$
divides ${\displaystyle {\bf L}(k+j,f) \over \displaystyle {\bf L}(j/2+k,f)}$. Then, there exists a Hecke eigenform
$F \in S_{(k+j,k)}(\varGamma^{(2)})$ such that 
\[\scra_4^{(I)}(F) \equiv_{\ev} [\scri_2(f)]^{{\bf k}}  \mod{\frkp}.\]

\end{conjecture}

Our main result in this paper is to prove Conjecture \ref{conj.main-conjecture} in the case $k$ is even and $j \equiv 0 \mod 4$.
Let $\Omega(f)_{\pm}=\Omega_{\mathrm{Kato}}(f)_\pm$ be the period of $f$ arising from the cohomology of an integral model of the Kuga-Sato variety, which will be explained in Section 8.
Let $(\rho_f,V_f)$ be the Galois representation of $G_{\QQ}$ attached to $f$ in Theorem \ref{th.Galois-spin}.
For positive integers $m, l$ such that $m \ge 4$ and a Hecke eigenform $F \in  S_k(\varGamma^{(2)})$, put
\[\calc_{m,l}(F)=\prod_{j=3}^{[m/2]}\zeta(1+2j-2l){\bf L}(l-2,F,\St).\]
For each $A \in \calh_m(\ZZ)_{>0}$, let $\frkd_A$ be the discriminant of $\QQ(\sqrt{(-1)^{m/2} \det A})/\QQ$ and 
$\chi_A=\Big({\frkd_A \over }\Big)$  the Kronecker character corresponding to $\QQ(\sqrt{(-1)^{m/2} \det A})/\QQ$.
For a prime number $p$, let $\zeta_p$ be the $p$-th root of unity, and let $\calc_p$ be the $p$-Sylow subgroup of the ideal class group of $\QQ(\zeta_p)$. 
Let $\bar\chi:G_{\QQ} \to \FF^\times_p$ be the 
mod $p$ cyclotomic character, and $\omega$ the Teichm\"uller lift of $\bar \chi$ which acts on $\calc_p$ via the restriction 
$G_{\QQ} \to\Gal(\QQ(\zeta_p)/\QQ)$. 
We also denote by  $\calc_p^{\omega^i}$ the $\omega^i$-isotypical part of $\calc_p$. 
\begin{theorem} \label{th.main-result}
Let $k$ and $j$ be positive even integers such that $k \ge 4$ and $j \equiv 0 \mod 4$, and ${\bf k}= (k+j/2,k+j/2,j/2+4,j/2+4)$.
Let $f$ be a primitive form in $S_{2k+j-2}(\SL_2(\ZZ))$ and
$\frkp$ be a prime ideal of $\QQ(f)$ and $p=p_{\frkp}$. 
Suppose that the following conditions hold:
\begin{itemize}
\item[(C.1)] $p >2k+j-2$;
\item[(C.2)] $\frkp$ divides ${\bf L}(k+j,f:\Omega_+(f))$;
\item[(C.3)] $\frkp$ does not divide 
\begin{align*}
a(A_1,\scri_2(f))\calc_{8,\frac{j}{2}+2}(\scri_2(f))\overline{a(A,[\scri_2(f)]^{{\bf k}})}
\end{align*}
 for some $A_1 \in \calh_2(\ZZ)_{>0}$ such that $\frkp \nmid \frkd_{A_1}$ and  $A \in \calh_4(\ZZ)_{>0}$;
\item[(C.4)] $\frkp$ does not divide $\frkD_f^{(p)}$;
%\item[(C.5)] $\bar \rho_f$ is absolutely irreducible.%
\item[(C.5)]  $\rho_f(\Gal(\bar \QQ/\QQ(\zeta_{p^{\infty}}))$ contains $\SL_2(\ZZ_p)$ (with a suitable choice of a lattice of $V_f$);
%\item[(C.7)] $f$ is $p$-ordinary.%
\item[(C.6)]  $\#\mathcal C_p^{\omega^{-j}}=\#\mathcal C_p^{\omega^{-j-2}}=\#\mathcal C_p^{\omega^{-j/2}}=1$;
\item[(C.7)] $p$ divides none of the zeta values $\zeta(-1-j), \zeta(1-j), \zeta(1-j/2)$ and $\zeta(-1-j/2)$;
\item [(C.8)] $\frkp$ divides none of the $L$-values 
${\bf L}(k+j/2,f;\Omega_+(f)), \ \ {\bf L}(k+j-1,f;\Omega_-(f)), {\bf L}(k+j-2,f;\Omega_+(f))$ and ${\bf L}(k+j+1,f;\Omega_-(f))$.
\end{itemize}
Then there is a Hecke eigenform $F_0$ in $S_{(k+j,k)}(\varGamma^{(2)})$ such that
\[\scra_4^{(I)}(F_0)  \equiv_\ev [\scri_2(f)]^{{\bf k}} \mod \frkp.\]
\end{theorem}
\begin{remark}\label{rem.confirm-assumption}
\begin{itemize}
\item[(1)] We have an algorithm for computing \begin{align*}
a(A_1,\scri_2(f))\calc_{8,\frac{j}{2}+2}(\scri_2(f))\overline{a(A,[\scri_2(f)]^{{\bf k}})}
\end{align*}
and we can easily check whether the condition (C.3) holds or not
(cf. Proposition \ref{prop.fc-klingen1}). 
%\item[(2)] The condition (C.4)  can easily be checked. %
%\item[(2)] The condition (C.5) holds if $p>2k+j-2$ and $p$ does not divide $\zeta(3-2k-j)$%
\item [(2)] The condition (C.5) holds for almost all $p$ and this can easily be checked (cf. Lemma \ref{lem.Ribet's-criterion}).
We also note that $\bar \rho_f$ is absolutely irreducible by this condition.
%\item[(5)] The condition (C.7) can easily be checked.%
\item [(3)]  If Vandiver's conjecture holds, the condition (C.6)  holds.  Vandiver's conjecture holds for $p <12 000 000$ (cf. \cite{B-C-E-M-S01}).
\item [(4)]  If $p$ is a regular prime, the conditions (C.6) and (C.7) hold.
 \end{itemize}
\end{remark}
To confirm our conjecture rigorously, we rewrite the above theorem. 
\begin{theorem} \label{th.main-result2} 
Let the notation be as in Theorem \ref{th.main-result}.
In addition to the assumptions \text{(C.1), (C.3)--(C.7)}, suppose that the following two conditions hold:
\begin{itemize}
\item[(C.2')] $\frkp$ divides ${\bf L}(k+j,f)/{\bf L}(k+j/2,f)$,
\item [(C.8')] $\frkp$ divides neither 
${\bf L}(k+j/2, k+j-1;f)$ nor ${\bf L}(k+j-2,k+j+1;f)$.
\end{itemize}
Then the same assertion as Theorem \ref{th.main-result} holds.
\end{theorem}
\begin{remark}\label{rem.confirm-assumption2}

\begin{itemize}
\item[(1)] The condition (C.2') does not depend on the choice of periods $c_{\pm}(f)$.
\item [(2)] The condition (C.8') does not depend on the choice of the periods, and it can easily be checked by Shimura's or Zagier's method. In particular, suppose that
$\frkp$ divides neither ${\bf L}(k+j-2,f)/{\bf L}(k+j/2,f)$ nor ${\bf L}(k+j-1,f)/{\bf L}(k+j+1,f)$. 
Then if  $\frkp$ does not divide ${\bf L}(k+j/2, k+j-1;f)$, 
the condition (C.8') holds. 

 \end{itemize}
\end{remark}

The proofs of Theorem \ref{th.main-result} and \ref{th.main-result2} will be given in Sections 8 and 9, respectively.  By \cite[Theorem 4.5]{A-C-I-K-Y23}, we can prove Harder's conjecture in the case $k$ is even and $j \equiv 0 \mod 4$.

\begin{theorem} \label{th.Harder-congruence}
Let the notation and the assumption be  as in Theorem \ref{th.main-result} or as in Theorem \ref{th.main-result2}. Then there exists a Hecke eigenform $F 
 \in S_{(k+j,k)}(\varGamma^{(2)})$, and a prime $\frkp' \mid\frkp$ in (any
field containing) $\QQ(f)\QQ(F)$ such that, for all prime numbers  $\ell$
\[L_{\ell}(X,F,{\rm Sp}) \equiv L_{\ell}(X,f)(1-{\ell}^{k-2}X)(1-{\ell}^{k+j-1}X)
\mod{\frkp'}.\]
\end{theorem}
Let $\varpi$ be the uniformizer of $K_{\frkp}$. 
By \cite[Proposition 4.3]{Dummigan-Ibukiyama-Katsurada11}, we have 
\begin{theorem} \label{th.Bloch-Kato-torsion}
Let the notation and the assumption be as in Theorem \ref{th.main-result} or as in Theorem \ref{th.main-result2}. Let $T$ be  a $G_{\QQ}$-stable lattice of $V_f$. Let $\mathrm{Sel}(\QQ,T(k+j))$ be the Bloch-Kato Selmer group associated with $T(k+j)$. (The definition of $\mathrm{Sel}(\QQ,T(k+j))$ will be given in Section 7.)
Then $\mathrm{Sel}(\QQ,T(k+j))$ has a $\frkp$-torsion element.
\end{theorem}

\section{Congruence for Klingen-Eisenstein lifts}

The following theorem is due to  \cite[Theorem 6.13]{A-C-I-K-Y23}.
\begin{theorem} \label{th.main-congruence}
Let $k$ and $l$ be positive even integers such that $l \le k$, and put ${\bf k}=(k,k,l,l)$. 
We define a subspace $\widetilde M_{\bf k}(\varGamma^{(4)})$ of $M_{\bf k}(\varGamma^{(4)})$ as
\[\widetilde M_{\bf k}(\varGamma^{(4)})=\begin{cases} \{G \in  M_{\bf k}(\varGamma^{(4)}) \ | \ \Phi^4_2(G) \in S_k(\varGamma^{(2)}) \} & \text{ if } l<k \\
M_{\bf k}(\varGamma^{(4)}) & \text{ if } k=l. \end{cases}
\]
Let $f$ be a primitive form in $S_{2k-2}(\SL_2(\ZZ))$ and $\frkp$ a prime ideal of $\QQ(f)$ such that
\begin{itemize}
\item[(1)] $p_\frkp > 2k-2$;
\item[(2)] $\frkp$ divides ${\bf L}(k+l-4,f)/{\bf L}(k,f)$;
\item[(3)] $\frkp$ divides neither $\frkD_f$ nor $\zeta(3-2k)$;
\item[(4)] $\frkp$ divides neither $\calc_{8,l}(\scri_2(f))a(A_1,\scri_2(f))\overline{a(A,[\scri_2(f)]^{\bf k})}$ nor $\frkd_{A_1}$ for some $A_1 \in \calh_2(\ZZ)_{>0}$ and $A \in \calh_4(\ZZ)_{>0}$.
\end{itemize}
Then there exists a Hecke eigenform $G$ in $\widetilde M_{{\bf k}}(\varGamma^{(4)})$ such that
$G$ is not a constant multiple of $[\scri_2(f)]^{{\bf k}}$ and 
\[G \equiv_{\mathrm {ev}} [\scri_2(f)]^{{\bf k}}  \mod{\frkp}.\]
\end{theorem}

Let ${\bf k}=(k,k,l,l)$ and we write $k$ for $(k,k)$ with $k,l$ positive even integers such that $k \ge l$. To confirm the condition (4) in Theorem \ref{th.main-congruence}, we give a formula for computing ${\bf L}(l-2,F,\St)a(N_1,F)a(N,[F]^{{\bf k}})$.

For $T \in \calh_m(\ZZ_p)$, let $F_p^*(T,X)$ be the polynomial associated with the Siegel series $b(s,T)$ in \cite[Section 6]{A-C-I-K-Y23}. Moreover, let $\calm_{k,l}=M_k(\varGamma^{(2)})$ or $S_k(\varGamma^{(2)})$ according as $k=l$ or not, and 
let $\{f_j \}_{j=1}^d$ be a basis of $\calm_{k,l}$ consisting of Hecke eigenforms.

Let $k>l$ and put ${\bf k}'=(k-l,k-l,0,0)$. 
Let $V$ be a $2 \times 4$ matrix  of variables.  For 
 $A_0 \in \calh_2(\ZZ)_{>0}, A_1 \in \calh_4(\ZZ)_{>0}$ and  $R \in M_{24}(\ZZ)$, put  $r(R)=r(A_0,A_1,R)=\rank  \begin{pmatrix} A_0 & R/2 \\ {}^t\!R/2 & A_1 \end{pmatrix}$ and
\begin{align*}
&\calz(A_0,A_1,R,l) \\
&=
\begin{cases}
 L\Bigl({4-l,\chi_{\bigl(\begin{smallmatrix} A_0 & R/2 \\ {}^t\!R/2 & A_1 \end{smallmatrix}\bigr)}}\Bigr)
 & \text{ if } r(A_0,A_1,R)=6 ,\\
\zeta(7-2l) & \text{ if } \ r(A_0,A_1,R)=5 ,\\
\zeta(7-2l)L(3-l,\chi_{A_0}) & \text{ if } \ r(A_0,A_1,R)=4.
\end{cases}
\end{align*}
Moreover, put
\begin{align*}
&P\Bigl(\begin{pmatrix} A_0 & R/2 \\ {}^t \!R/2 & A_1 \end{pmatrix}\Bigr)(V) =(k-l)! \sum_{a,b,c \ge 0 \atop a+2b+2c=k-l}{(-1)^b2^{-a} \over a!b!c! }\left(l+c-{3 \over 2}\right)_{a+b+c} \\
&\times |R \ {}^t\, \!V|^a  (|VA_1{}^t\, \!V||A_0| )^b
\begin{vmatrix} A_0 & R \ {}^t\, \!V/2 \\ V \ {}^t\!R/2& VA_1 {}^tV
\end{vmatrix}^c ,
\end{align*}
and 
\begin{align*}
&\epsilon_{k,{\bf k}}(A_0,A_1)(V)=\sum_{R \in M_{24}(\ZZ) \atop 
\left(\begin{smallmatrix} A_0 & R/2 \\ {}^t\!R/2 &A_1 \end{smallmatrix} \right) \ge 0} 2^{[r(R)+1)/2]}\calz(A_0,A_1,R,l)P\Bigl(\begin{pmatrix} A_0 & \!R/2 \\ {}^t\!R/2 & A_1 \end{pmatrix}\Bigr) (V)\\
& \times\prod_{p } F_p^\ast\Bigl(\left(\begin{smallmatrix} A_0 & R/2 \\ {}^t\!R/2 &A_1 \end{smallmatrix} \right),p^{l-r(R)-1}\Bigr).
\end{align*}
Here $(x)_m=x(x+1)\cdots(x+m-1)$ is the ascending Pochhammer symbol.

For $T \in \calh_2(\ZZ)_{>0},\ N \in \calh_4(\ZZ)_{>0}, \ $ and a positive integer $m$, we define $\epsilon_{k,{\bf k}}(m,T,N)$ inductively as follows:
$$ \epsilon_{k,{\bf k}}(1,T,N)=\epsilon_{k,{\bf k}}(T,N),$$
and for $m >1$,
\begin{align}
 \epsilon_{k,{\bf k}}(m,T,N) =& \epsilon_{ k,{\bf k}}(mp^{-1},pT,N)+p^{2k-3} \epsilon_{k,{\bf k}}(mp^{-1},T/p,N)\nonumber\\
&+p^{k-2}\sum_{D \in \GL_2({\ZZ}) U_p \GL_2({\ZZ})/\GL_2({\ZZ})} \epsilon_{k,{\bf k}}(mp^{-1},T[D]/p,N) ,
\nonumber
\end{align}
where $p$ is a prime factor of $m$ and $U_p=\begin{pmatrix} 1 & 0 \\ 0 & p \end{pmatrix}$. 
Then the following proposition follows from \cite[Corollary 7.3 and Theorem 7.5]{A-C-I-K-Y23}.
 \begin{proposition}\label{prop.fc-klingen1}
 For $N_1 \in \calh_2(\ZZ)_{>0},N \in \calh_4(\ZZ)_{>0}$  let $e_m= \epsilon_{ k,{\bf k}}(m,N_1,N)$.  
Let $F$ be a Hecke eigenform in $S_k(\varGamma^{(2)})$ and 
let $\{F_j \}_{j=1}^d$ be a basis of $\calm_{k,l}$ consisting of Hecke eigenforms such that $F_1=F$.
For  positive integers  $m_1,\ldots,m_d$ put $\Delta=\Delta(m_1,\ldots,m_d)=\det (\lambda_{j,m_i})_{1 \le i,j \le d}$.
Let $\frkp$ be a prime ideal of $\QQ(F)$ such that $p_\frkp>2k-2$. Suppose that $\frkp$ divides neither $\zeta(9-2l)$ nor $\begin{vmatrix} e_1 & \lambda_{1,2} &\hdots&  \lambda_{1,d} \\
               \vdots & \vdots         &\vdots & \vdots          \\
                e_d    &  \lambda_{d,2}         &\hdots &  \lambda_{d,d} \end{vmatrix}
$. Then $\frkp$ does not divide 
$\calc_{8,l}(F)\overline{a(N,[F]^{{\bf k}})}a(N_1,F)$.
\end{proposition}
We have an explicit formula for $F_p^*(T,X)$ for any half-integral matrix $T$ over ${\ZZ}_p$ (cf. \cite{Katsurada99}, \cite{Ikeda-Katsurada18},\cite{Ikeda-Katsurada22}), and an algorithm for computing it (cf. \cite{Lee18}). 
Therefore, by using  Proposition \ref{prop.fc-klingen1}, we can  check the indivisibility of the Fourier coefficients of  the Klingen-Eisenstein series by a prime ideal in Theorem \ref{th.main-congruence}.

\section{Preliminary results}
In this section, we provide several preliminary results for proving Theorem \ref{th.main-result}. Let $K$ be an algebraic number field and $\frkO$ be its ring of integers. For a prime ideal $\frkp$ of $\frkO$ put $\FF=\frkO_{\frkp}/\frkp\frkO_{\frkp}$.
Let $\rho:G_{\QQ} \to \mathrm{Aut}_{K_\frkp}(V) \cong \GL_n(K_\frkp)$ be a Galois representation. Then for a $G_{\QQ}$-stable lattice $T$ of $V$, we denote by $\rho_T$ the representation $\rho_T:G_{\QQ} \to \mathrm{Aut}_{\frkO_\frkp}(T) \cong \GL_n(\frkO_\frkp)$. From now on we simply say $T$ is a lattice of $V$ if $T$ is a $G_{\QQ}$-stable lattice of $V$. We sometimes drop the suffix $T$ and simply write $\rho$ for $\rho_T$ if there is  no fear of confusion.
From now on, we assume that $p_\frkp$ is odd.
Throughout this section, for a $\frkp$-adic representation $\rho$ let $\bar \rho$ denote the mod $\frkp$ representation of $\rho$. 
Let ${\bf u}=(u_1,\ldots,u_n)$ be a basis of $T$, we denote by $A_{T,{\bf u}}=(a_{T,{\bf u};i,j})_{1 \le i,j \le n}$ be the representation matrix of $\bar \rho_T$ such that
\[(\overline {\sigma(u_1)}, \ldots,\overline{\sigma(u_n)}) = (\overline{u_1},\ldots,\overline{u_n}) (a_{T,{\bf u};i,j}(\sigma)) \text{ for any } \sigma \in G_{\QQ}.\]
We sometimes write $A_{T}=(a_{T;i,j})$ for $A_{T,{\bf u}}=(a_{T,{\bf u};i,j})$, and identify $\bar \rho_T$ with $A_T$. 

\begin{lemma} \label{lem.Fontaine-Laffaille-Messing}
Let $p=p_{\frkp}$ and  $k$  a positive even integer such that $k<p$. For  a primitive form $f$  in $S_k(\SL_2(\ZZ))$, let $\rho_f:G_\QQ \to GL_2(K_\frkp)$ be the Galois representation in Theorem 4.4.
\begin{itemize}
\item[(1)] Let $a,b$ be  integers such that $-p+1<a< b <p-1$.  Suppose that 
\[\bar \rho_f^{\rm ss} = \bar \chi^a \oplus \bar \chi^b.\]
Then $(a,b)=(1-k,0)$.
\item[(2)] Let $g$ be a primitive form in $S_l(\SL_2(\ZZ))$ with $l \le k$ and $i$ an integer such that $1-p < i<p-1$.
Suppose that  
\[\bar \rho_g^{\rm ss} = \bar \rho_f(i)^{\rm ss}. \]
Then $k=l$ and $i=0$.
\end{itemize}
\end{lemma}
\begin{proof}
By \cite[Theorem 1.2]{Buzzard-Gee09} and  its remark, 
$\overline{\rho}_{f}^{ss}|_{I_p}$ should be 
\[\overline{\chi}^{1-k} \oplus 1\]
 or 
\[{\omega}_2^{1-k} \oplus {\omega}_2^{p(1-k)}\]
 with ${\omega}_2$ the fundamental character of level 2. Thus the assertion (1) holds.
Similarly  
$\overline{\rho}_{g}^{ss}|_{I_p}$ should be 
\[\overline{\chi}^{1-l} \oplus 1\]
 or 
\[{\omega}_2^{1-l} \oplus {\omega}_2^{p(1-l)}.\]
Thus the assertion (2) holds.
\end{proof}
We review the definition of Bloch-Kato Selmer group attached to a Galois representation. For a topological  group $G$ and $G$-module $M$, let $\mathit{H}^1(G,M)$ be the first continuous  cohomology group of $G$ with coefficients in $M$.  
In particular, if $L$ is a field, we write $\mathit{H}^1(L,M)$ for $\mathit{H}^1(\Gal(\bar L/L),M)$, where $\bar L$ is the separable closure of $L$.  Let $\rho$  be a $\frkp$-adic Galois representation of $G_{\QQ}$, and $V$ be its representation space. Let $T$ be a $G_{\QQ}$-stable lattice of $V$.  For a prime number $\ell$, we define
$H_f^1(\QQ_{\ell},V)$ as
\begin{align*}
H_f^1(\QQ_{\ell},V)
&=\begin{cases}
 \mathrm{Ker}\left(\mathit{H}^1(\QQ_{\ell},V) \to\mathit{H}^1(I_{\ell},V)\right) & \text{ if } \ell \not=p, \\
\mathrm{Ker}\left(\mathit{H}^1(\QQ_p,V) \to \mathit{H}^1(\QQ_p,V 
\otimes_{\QQ_p} B_{\rm crys})\right) & \text{ if } \ell =p,
\end{cases}
\end{align*}
where $B_{\rm crys}$ is the Fontaine's ring of $p$-adic periods (cf. \cite{Fontaine94}). We then put
\[\mathit{H}_f^1(\QQ_{\ell},T \otimes_{\ZZ_{\ell}}\QQ_{\ell}/\ZZ_{\ell})=\mathrm{Im}(\mathit{H}_f^1(\QQ_{\ell},V) \to \mathit{H}^1(\QQ_{\ell},T \otimes_{\ZZ_{\ell}} \QQ_{\ell}/\ZZ_{\ell}))\]
and define the Selmer group $\mathrm{Sel}(\QQ,T)$ for $T$ as 
\[\mathrm{Sel}(\QQ,T)=\mathrm{Ker} \Bigl(\mathit{H}^1(\QQ,T \otimes_\ZZ\QQ/\ZZ) \to \bigoplus_{\ell} \mathit{H}^1(\QQ_{\ell},T \otimes_{\ZZ_{\ell}} \QQ_{\ell}/\ZZ_{\ell})/\mathit{H}_f^1(\QQ_{\ell},T \otimes_{\ZZ_{\ell}} \QQ_{\ell}/\ZZ_{\ell}) \Bigr).\] 
When $\ell=p$, we denote by $H_f^1(\QQ_p,T)$ the pullback of $H_f^1(\QQ_p,V)$ under 
the natural map $H^1(\QQ_p,T) \to  H^1(\QQ_p,V)$ induced 
from the natural inclusion $T\subset V$. Then we define 
$H_f^1(\QQ_p,T\otimes_{\ZZ_p}\FF_p):=i_\ast H_f^1(\QQ_p,T)$ to be the image of $H_f^1(\QQ_p,T)$ 
under the map $i:H^1(\QQ_p,T) \to  H^1(\QQ_p,T\otimes\FF_p)$ induced by 
the reduction map $T\to   T\otimes\FF_p$.

Let $R_i:G_{\QQ} \to\GL_{n_i}(\FF) \ (i=1,2)$ be a Galois representation, and 
let $Z^1(R_2,R_1)$ be the set of mappings 
$A:G_{\QQ} \to M_{n_1,n_2}(\FF)$ such that
\[ A(\sigma \tau)=R_1(\sigma)A(\tau)+A(\sigma)R_2(\tau) \quad \text{ for any } \sigma,\tau \in G_{\QQ},\]
and  $B^1(R_2,R_1)$ the subset of $Z^1(R_2,R_1)$ consisting of the mapping 
\[\sigma \mapsto UR_2(\sigma)-R_1(\sigma)U  \quad \text{ with some } U \in M_{n_1,n_2}(\FF), \]
and put $\mathrm{Ext}(R_2,R_1)=Z^1(R_2,R_1)/B^1(R_2,R_1)$.
For $A_1,A_2$,  we write $A_1 ='_{R_2,R_1} A_2$ if $A_1 - A_2 \in B^1(R_2,R_1)$. 
We sometimes write $A_1='A_2$ if there is no danger of confusion. Let $(\rho,V)$ be a Galois representation of $G_\QQ$ over $K_\frkp$ of dimension $n$.
For two Galois representations $(\rho_1,V_1),(\rho_2,V_2)$ of  $G_{\QQ}$ over $\FF$ of dimensions $n_1$ and $n_2$, respectively, such that $n_1+n_2=n$ and $\rho_1$ is a subrepresentation of $\bar \rho$, let $\call_{V_1,V_2}$ be the set of all lattices $T \subset V$ such that
\[\bar \rho_T =\begin{pmatrix} \rho_1 & B_T \\ O & \rho_2 \end{pmatrix},\]
where $B_T \in Z^1(\rho_2,\rho_1)$. We sometimes write $\call_{\rho_1,\rho_2}$ instead of $\call_{V_1,V_2}$. The following two lemmas and two corollaries are easy to see, but for  readers' convenience, we give  proofs to them.
\begin{lemma} \label{lem.lattice-change} 
Let $K$ be an algebraic number field and $\frkO$ be its ring of integers. For a prime ideal $\frkp$ of $\frkO$ and $\FF=\frkO_{\frkp}/\frkp\frkO_{\frkp}$. Let $\rho:G_{\QQ} \to \mathrm{Aut}_{K_\frkp}(V) \cong \GL_n(K_{\frkp})$ be a $\frkp$-adic Galois representation.
\begin{itemize}
\item[(1)] Let $T$ be  a lattice of $V$ such that  
\[\bar \rho_T=\begin{pmatrix} \rho_{1} & A_{12} \\ O &  \rho_{2} \end{pmatrix}\]
with $\rho_i:G_{\QQ} \to \GL_{n_i}(\FF)$ a $\frkp$-adic Galois representation such that  $n_1+n_2=n$ and
$A_{12} \in Z^1(\rho_2,\rho_1)$.
Then there is a lattice $T'$ such that
\[\bar \rho_{T'}=\begin{pmatrix}  \rho_{2} & A_{12}' \\ O &  \rho_{1} \end{pmatrix}\]
with $A_{12}'  \in Z^1( \rho_1, \rho_2)$.
\item[(2)] Let $T$ be a lattice  of the following form:
\[\bar \rho_T=\begin{pmatrix}  \rho_{1}  & A_{12} & A_{13} & A_{14} \\ O &  \rho_{2} &O & A_{24} \\
O & O & \rho_{3} & A_{34} \\
O & O & O &  \rho_{4}
\end{pmatrix}
\]
with $\rho_i:G_{\QQ} \to \GL_{n_i}(\FF)$ a $\frkp$-adic Galois representation such that  $n_1+n_2+n_3+n_4=n$. Then there is a lattice $T'$ such that 
\[\bar \rho_{T'}=\begin{pmatrix}  \rho_{1} & A_{13} & A_{12} & A_{14} \\ O &  \rho_{3} &O & A_{34} \\
O & O &  \rho_{2} &A_{24} \\
O & O & O &  \rho_{4}
\end{pmatrix}
.\]
\end{itemize}
\end{lemma}
\begin{proof} (1) We take $T'$ so that
\[\rho_{T'}= \begin{pmatrix} O & \varpi^{-1} {\bf 1}_{n_1} \\{\bf 1}_{n_2} & O \end{pmatrix}^{-1}\rho_T \begin{pmatrix} O & \varpi^{-1} {\bf 1}_{n_1} \\{\bf 1}_{n_2} & O \end{pmatrix}.\] Then $T'$ satisfies the required condition.

(2) We take $T'$ so that 
\[\rho_{T'}=\Bigg(\begin{smallmatrix} {\bf 1}_{n_1}  & O & O & O \\
O  & O & {\bf 1}_{n_2} & O \\
O & {\bf 1}_{n_3} & O & O \\
O & O & O & {\bf 1}_{n_4}
\end{smallmatrix}\Bigg)^{-1}\rho_T \Bigg(\begin{smallmatrix} {\bf 1}_{n_1}  & O & O & O \\
O  & O & {\bf 1}_{n_2} & O \\
O & {\bf 1}_{n_3} & O & O \\
O & O & O & {\bf 1}_{n_4}
\end{smallmatrix}\Bigg).\] Then $T'$ satisfies the required condition.
\end{proof}
\begin{lemma} \label{lem.irreducibility-criterion} Let $\rho$ be as in Lemma \ref{lem.lattice-change}. Suppose that $\rho$ is irreducible and  that
\[\bar \rho^{\rm ss} =V_1 \oplus V_2\]
with $(\rho_1,V_1),(\rho_2,V_2)$ Galois representations of $G_{\QQ}$ over $\FF$. 
 Then, there exists a lattice $T \in \call_{V_1,V_2}$ such that
\[\bar \rho_T =\begin{pmatrix} \rho_1 & A_{12} \\ O & \rho_2 \end{pmatrix}\]
with $A_{12} \not='0$.
\end{lemma}
\begin{proof}
Put $n_i=\dim_{\FF} V_i$ for $i=1,2$, and let $\call_{V_1,V_2}$ be the set of lattices of $\rho$ defined above. Since $\rho$ is irreducible, for any element $T \in \call_{V_1,V_2}$, $\rho_T$ can be written as
\[\rho_T=\begin{pmatrix} R_{11} & \varpi^{r_T}R_{12} \\
\varpi R_{21} & R_{22}\end{pmatrix} ,\]
where $R_{ij}$ is a mapping from $G_{\QQ}$ to $M_{n_i,n_j}(\frkO_{\frkp})$ and $r_T$ is a non-negative integer such that
\[\bar R_{11}=\rho_1, \bar R_{22}=\rho_2 \text{ and } \bar R_{12} \not= O.\]
Since $\rho$ is irreducible,  the set $\calr=\{r_T \ | \ T \in 
\call_{V_1,V_2}$\} is bounded. Put $r_0=\max \calr$ and take a lattice $T_0$ so that
\[\rho_{T_0}=\begin{pmatrix} R_{11} & \varpi^{r_0}R_{12} \\
\varpi R_{21} & R_{22}\end{pmatrix}\]
and take a lattice $T$ so that
\[\rho_T=\Big(\begin{smallmatrix} {\bf 1}_{n_1} & O \\ O & \varpi {\bf 1}_{n_2} \end{smallmatrix} \Big)^{-1}\rho_{T_0}\Big(\begin{smallmatrix} {\bf 1}_{n_1} & O \\ O &\varpi  {\bf 1}_{n_2} \end{smallmatrix} \Big).\]
Then we have
\[\bar \rho_T=\begin{pmatrix} \rho_1 & \bar R_{12} \\
O &\rho_2\end{pmatrix}.\]
Suppose that $\bar R_{12} ='0$. Then
\[\bar R_{12}=U\rho_2-\rho_1U \ \text{ with } U \in M_{n_1,n_2}(\FF).\]
Take $U' \in M_{n_1,n_2}(\frkO_\frkp)$ so that $\bar U'=U$. Then
\[\overline{\Big(\begin{smallmatrix} {\bf 1}_{n_1} & U' \\ O &  {\bf 1}_{n_2} \end{smallmatrix} \Big)^{-1}\rho_T \Big(\begin{smallmatrix} {\bf 1}_{n_1} & U' \\ O & {\bf 1}_{n_2} \end{smallmatrix} \Big)}=\begin{pmatrix} \rho_1 &O \\ O & \rho_2 \end{pmatrix}.\]
This implies that there is an element $T'$ such that $r_{T'} \ge r_0+1$, which is a contradiction. Hence $\bar R_{12} \not='0$, and $T$ satisfies the required condition.
\end{proof}
\begin{corollary} \label{cor.no-factor} Let $\rho$ be as in Lemma \ref{lem.lattice-change}
 and $\tau$ be a semi-simple subrepresentation of $\rho$.
Let 
\[\bar \tau^{\rm ss} =R_1 \oplus \cdots \oplus  R_r,\]
where $R_1,\ldots,R_r$ are Galois representations of $G_{\QQ}$ over $\FF$. 
Suppose that there is an integer $2 \le i_0 \le r$ such that 
 $\mathrm{Ext}(R_j, R_i)=0$ for any $1 \le i \le i_0-1$
and  $i_0 \le j \le r$. Then, $\tau$ is not irreducible, and  
\[\tau=\tau_1 \oplus \tau_2, \]
where $\tau_1$ and $\tau_2$ are subrepresentations of $\tau$ such that 
\[\bar \tau_1^{\rm ss}=R_1 \oplus \cdots \oplus  R_{i_0-1} \text{ and } \bar \tau_2^{\rm ss}=R_{i_0} \oplus \cdots \oplus  R_r.\]
\end{corollary}
\begin{proof} Put $n_1=\dim_\FF R_1+\cdots+\dim_\FF R_{i_0-1}$ and $n_2=\dim_\FF R_{i_0}+\cdots+\dim_\FF R_r$. In view of Lemma \ref{lem.lattice-change},
we can take a lattice $T$ of $\tau$ so that
\[\tau_T=\begin{pmatrix} \tau_{11} & \tau_{12} 
\\ \varpi \tau_{21} & \tau_{22} \end{pmatrix},\]
where $\tau_{ij}$ is a mapping from $G_\QQ$ to $M_{n_i,n_j}(\frkO_\frkp)$ such that
\[\bar \tau_{11}=\begin{pmatrix} R_1 & * & * \\ 
    O& \ddots & * \\
    O & O & R_{i_0-1} \end{pmatrix} \text{ and }  \bar \tau_{22}=\begin{pmatrix} R_{i_0} & * & * \\ 
    O& \ddots & * \\
    O & O & R_r \end{pmatrix}. \]
By assumption, we have $\bar \tau_{12} ='_{\bar \tau_{22},\bar \tau_{11}} 0$ for any $T \in \call_{\bar \tau_{11},\bar \tau_{22}}$. Thus, by Lemma \ref{lem.irreducibility-criterion}, we prove the first part of the assertion. Moreover, by a careful analysis of the proof of Lemma  \ref{lem.irreducibility-criterion}, we can take a lattice $T$ so that
$\tau_T=\begin{pmatrix} \tau_{11} & O 
\\ * & \tau_{22} \end{pmatrix}$. Since $\tau$ is semi-simple, we prove the second part of the assertion.
\end{proof}

\begin{corollary} \label{cor.no-factor2} Let $\rho$ be as in Lemma \ref{lem.lattice-change},  and $(W,\tau)$ be a subrepresentation of $\rho$, and let 
$R_1,\ldots,R_r$ be Galois representations of $G_{\QQ}$ over $\FF$. 
  Let $\calt_\tau$ be the set of lattices $T$ of $W$ such that $\bar \tau_T$ is of the following form:
\[\bar \tau_T=\begin{pmatrix} R_1 & a_{T;1,2} & \hdots & a_{T;1,j} & \hdots  & a_{T;1,r} \\
0 & \ddots & \ddots & \ddots & \ddots & \vdots \\
0 & 0 & \ddots & \ddots & \ddots & \vdots \\
0 & 0 & 0 & R_j & \hdots & a_{T;j,r} \\
0 & 0 & 0 & 0 & \ddots & \vdots \\
0 & 0 & 0 & 0 & 0 & R_r \end{pmatrix}.\]
\begin{itemize}
\item[(1)]
Suppose that there is an integer $2 \le i_0 \le r$ such that 
 $\mathrm{Ext}(R_{i_0},R_j)=0$ for any $ i_0+1 \le j \le r$, and
\[\begin{pmatrix} a_{T;1,i_0} \\ \vdots \\ a_{T;i_0-1,i_0} \end{pmatrix} =' \begin{pmatrix} 0 \\ \vdots \\ 0 \end{pmatrix}\]
for any lattice $T \in \calt_{\tau}$.
 Then,
$\tau$ is not irreducible.
\item[(2)]
Suppose that there is an integer $2 \le i_0 \le r$ such that 
 $\mathrm{Ext}(R_j,R_{i_0})=0$ for any $1 \le j \le i_0-1$, and
\[
%\begin{pmatrix}
(a_{T;i_0,i_0+1},  \ldots , a_{T;i_0,r})
%\end{pmatrix} 
=' 
%\begin{pmatrix} 
(0,  \ldots  ,0)
%\end{pmatrix}
\]
for any lattice $T \in \calt_{\tau}$.
 Then,
$\tau$ is not irreducible.
\end{itemize}

\end{corollary}

\begin{proof} The assertions can be proved using the same argument as in the proof of Corollary \ref{cor.no-factor}.
\end{proof}

The following lemma is due to \cite[Lemma 7.7]{Brown07} (but no proof is given there). 
Though it seems well-known for experts, for readers' convenience, we give a detailed proof. A point we have to clarify here is 
to understand the relation between elements of Bloch-Kato Selmer groups 
and extension classes for crystalline representations over 
some coefficient rings. This is well-known (cf. \cite[Proposition 1.26]{Nekovar}) if 
the coefficient ring is a $p$-adic field.
\begin{lemma}\label{lem.extension-FL}Let $V$ be a short crystalline representation 
over $K_{\frkp}$
of $G_{\QQ_p}$ with a $G_{\QQ_p}$-stable $\frkO_{\frkp}$-lattice $T$. 
Let $\bar{X}$ be an $\FF[G_{\QQ_p}]$-subquotient  of $\bar{T}:=T\otimes_{\frkO_{\frkp}}\FF$ where $\FF=
\frkO_{\frkp}/\frkp$ such that there exists an exact sequence of $\FF[G_{\QQ_p}]$-modules:
$$0 \to \bar{T}_1 \to \bar{X} \to \bar{T}_2 
\to 0.$$
Then, we have the following:
\begin{enumerate}
\item
enlarging $K_{\frkp}$ if necessary, there exist short crystalline representations $X$ and 
$V_i,\ i=1,2$ over $K_{\frkp}$ of $G_{\QQ_p}$ 
with $\frkO_{\frkp}[G_{\QQ_p}]$-stable lattices $T_X$ and $T_i$ respectively such that 
\begin{enumerate}
\item 
$T_X\otimes_{\frkO_{\frkp}}\FF=\bar{X}$ and $T_i\otimes_{\frkO_{\frkp}}\FF=\bar{T_i}$, 
and 
\item $0 \to T_1 \to T_X \to T_2 
\to 0$;
\end{enumerate}
\item $\bar{X}$ gives a class of $H^1_f(\QQ_p,(\bar{T}_2)^\vee\otimes_{\FF}\bar{T}_1)$ 
with respect to 
the reduction map $(T_2)^\vee\otimes_{\frkO_{\frkp}}T_1\to  ((T_2)^\vee\otimes_{\frkO_{\frkp}}T_1)_{\frkO_{\frkp}}\otimes\FF$. 
In particular, if $T_i=\chi^{l_i}$ for any $l_i\in \ZZ,\ i=1,2$ then 
the class belongs to $H^1_f(\QQ_p,\FF(l_1-l_2))$.
\end{enumerate}
\end{lemma}
\begin{proof}
Let ${\rm MF}_{W(\FF),{\rm tor}}^{f,[0,p-2]}$ be the category whose object is a 
Fontaine-Laffaille module $M$ over the Witt vectors $W(\FF)$ which are of finite length and 
${\rm Fil}^0M=M,\ {\rm Fil}^{p-1}M=0$. By \cite[1.10(b) and Proposition 1.8]{Fontaine-Laffaille82}, ${\rm MF}_{W(\FF),{\rm tor}}^{f,[0,p-2]}$ is an abelian category 
(see \cite[Section 1]{Wintenberger84} for more properties).
We denote by ${\rm Rep}_{{\rm tor},[-(p-2),0]}(G_{\QQ_p})$ the 
category of torsion crystalline representations of $G_{\QQ_p}$ of 
weight in $[-(p-2),0]$ (see the lines after Theorem 3.1.3.2 of \cite{Breuil-Messing} 
for the definition). 
 By taking the (Pontryagin) dual of the contravariant functor $T^\ast_{{\rm cris}}$ in \cite[Theorem 3.1.3.3]{Breuil-Messing}, 
 one can define an equivalence ${T^\ast}^\vee_{{\rm cris}}:{\rm MF}_{W(\FF),{\rm tor}}^{f,[0,p-2]} \to {\rm Rep}_{{\rm tor},[-(p-2),0]}(G_{\QQ_p})$.  
Since $V$ is short crystalline, $\bar{T}$ belongs to 
${\rm Rep}_{{\rm tor},[-(p-2),0]}(G_{\QQ_p})$. 
Since ${\rm MF}_{W(\FF),{\rm tor}}^{f,[0,p-2]}$ is abelian, so is 
${\rm Rep}_{{\rm tor},[-(p-2),0]}(G_{\QQ_p})$. Therefore,  $\bar{X}$ also belongs to ${\rm Rep}_{{\rm tor},[-(p-2),0]}(G_{\QQ_p})$.  
Since there is an equivalence of categories between ${\rm Rep}_{{\rm tor},[-(p-2),0]}(G_{\QQ_p})$ and ${\rm MF}_{W(\FF),{\rm tor}}^{f,[0,p-2]}$, it has the same property as  \cite[Proposition 1.6.3]{Wintenberger84}. Hence, there exists a crystalline lattice 
$T_X$ such that $ T_X \otimes_{\frkO_{\frkp}} \FF \in {\rm Rep}_{{\rm tor},[-(p-2),0]}(G_{\QQ_p})$ 
and $ T_X\otimes_{\frkO_{\frkp}}\FF=\bar{X}$. 
Applying the proof of \cite[Proposition 1.6.3]{Wintenberger84}, one can first construct 
$0\to   T_1 \to  T_X$ in ${\rm Rep}_{{\rm tor},[-(p-2),0]}(G_{\QQ_p})$ 
such that it is a lift of $0\to  \bar{T}_1 \to  \bar{X} $ and $ T_2:=T_X/T_1$ is torsion free. Recall ${\rm MF}_{W(\FF),{\rm tor}}^{f,[0,p-2]}$ 
is closed under taking quotients, and so is ${\rm Rep}_{{\rm tor},[-(p-2),0]}(G_{\QQ_p})$. Therefore, $T_2$ yields an 
object in ${\rm Rep}_{{\rm tor},[-(p-2),0]}(G_{\QQ_p})$ as 
$T_2 \otimes_{\frkO_{\frkp}}\FF=\bar{T}_2$. Put $V_i=T_i \otimes_{\frkO_{\frkp}}K_{\frkp}$ for $i=1,2$. Then, $X, T_X$ and $V_i, T_i (i=1,2)$ satisfy the conditions (a),(b).  Hence, we have the first claim. 

For the second claim, let $\pi$ be a uniformizer of $\frkO_{\frkp}$ and put $S=(T_2)^\vee\otimes_{\frkO_{\frkp}}T_1$ for simplicity.  
The exact sequence $0\to  
S\stackrel{\times \pi}{\to} S \to  S/\pi S \to 0$ 
yields an injection $H^1(\QQ_p,S)\otimes_{\frkO_{\frkp}}\FF\to  
H^1(\QQ_p,S/\pi S)$.
It is well-known that $H^1(\QQ_p,(V_2)^\vee\otimes_{K_{\frkp}}V_1)=H^1(\QQ_p,S)\otimes_{
\frkO_{\frkp}}K_{\frkp}$ (cf. \cite[Proposition B.2.4]{Rubin00}). 

By the first claim and \cite[Proposition 1.26]{Nekovar}, $X=T_X\otimes_{\frkO_{\frkp}}K_{\frkp}$ belongs to 
$H^1_f(\QQ_p,(V_2)^\vee\otimes_{K_{\frkp}}V_1)$. Then, $X$ is pulled back to a class 
of $T_X$ in $H^1_f(\QQ_p,S)$ under the natural inclusion $S\hookrightarrow S\otimes_{\frkO_{\frkp}}K_{\frkp}=(V_2)^\vee\otimes_{K_{\frkp}}V_1$. The class $\bar{X}$ obviously comes from $T_X$ under 
the injection $H_f^1(\QQ_p,S)\otimes_{\frkO_{\frkp}}\FF\hookrightarrow 
H^1(\QQ_p,S/\pi S)$. 
Namely, it belongs to 
$H^1_f(\QQ_p,(\bar{T}_2)^\vee\otimes_{\FF}\bar{T}_1)$. 
\end{proof}
\begin{remark}\label{fund-char} In Lemma \ref{lem.extension-FL},  (2), 
the fundamental characters appear in $\bar{X}|_{I_p}$, and they  can be described in terms of 
$\mathrm{HT}(V)$ by using Fontaine-Laffaille theory (cf. \cite[Proposition 3]{Wortmann}). 
In particular, if $T_i=\chi^{l_i},\ l_i\in\ZZ$ $(i=1$ or 2$)$, then we can choose  
$l_i$ from $\mathrm{HT(}V)$ so that $-(p-2)\le l_i\le 0$.  
\end{remark}

\begin{lemma} \label{lem.non-trivial-extension1}
Let $\frkp$ be a prime ideal of a number field $K$ and $p=p_\frkp$. Let 
$\rho: G_{\QQ} \to GL_r(K_{\frkp})$ be  a $\frkp$-adic Galois representation such that 
\begin{itemize}
\item[(a)] $\rho$ is unramified outside $p$;
\item[(b)] $\rho$ is short crystalline at $p$.
\end{itemize}
Let 
\[\bar \rho=\begin{pmatrix}                     
 * & * &* & * \\
 0 & \bar \chi^{l_1} & h &* \\
 0 & 0 & \bar \chi^{l_2} &* \\
 0 & 0 & 0 & *
\end{pmatrix}\]
be the $\mathrm{mod}$ $\frkp$ representation of $\rho$ with $l_1,l_2 \in  \ZZ$ such that $\{l_1,l_2\} \subset \mathrm{HT}(\rho)$ and $h \not=' 0$. Then we have the following:
\begin{itemize}
\item[(1)] Suppose that $l_1<l_2$. Then $\calc_{p}^{\omega^{l_1-l_2}}$ is non-trivial.
\item[(2)] Suppose  that $l_1-l_2$ is a positive even integer. Then, $p$ divides $\zeta(1-l_1+l_2)$.
\end{itemize}
\end{lemma}
\begin{proof} (1) The assertion may be proved using the same argument as in \cite[Section 8]{Brown07}. But, for reader's convenience, we here give an outline of the proof. Put $\tau=\begin{pmatrix}
\bar \chi^{l_1} & h \\ 0 & \bar \chi^{l_2} \end{pmatrix}$. Then we have
$\tau|_{\Gal(\bar \QQ/\QQ(\zeta_p))}=\begin{pmatrix}
\bar 1 & \widetilde h \\ 0 & \bar 1 \end{pmatrix}$ with $\widetilde h=h|_{\Gal(\bar \QQ/\QQ(\zeta_p))}$.
It is easily shown that $\widetilde h$ is a non-trivial homomorphism from $\Gal(\bar \QQ/\QQ(\zeta_p))$ to $\FF$, where $\FF=\frkO_\frkp/\frkp$.
Hence, the splitting field $L$ of $\widetilde h$ is an algebraic extension of $p$-power degree. It is easily shown that
$\Gal(L/\QQ)$ acts on $\Gal(L/\QQ(\zeta_p))$ via $\omega^{l_1-l_2}$. By the assumption (a), $L/\QQ(\zeta_p)$ is unramified outside $p$.
We claim that it is unramified also at $p$. We note that $h|_{D_p} \in H^1(\QQ_p,\FF(l_1-l_2))$ gives 
 an extension of $\FF[D_p]$-modules:
\[0 \to \FF(l_1) \to \bar X \to \FF(l_2) \to 0.\]
By the assumption (b) and  Lemma \ref{lem.extension-FL}, it belongs to  $H_f^1(\QQ_p,\FF(l_1-l_2))$. 
Since we have $l_1-l_2<0$, we have $H_f^1(\QQ_p,K_{\frkp}(l_1-l_2))=0$ (cf. \cite[Example 3.9]{Bloch-Kato90}), and hence $H_f^1(\QQ_p,(K_{\frkp}/\frkO_{\frkp})(l_1-l_2))=0$. Hence, by \cite[Proposition 7.9]{Brown07}, 
we have $h|_{D_p}=0$, and in particular  $\widetilde h|_{I_p(L/\QQ(\zeta_p))}=0$, where 
$I_p(L/\QQ(\zeta_p))$ is an inertia subgroup of $L/\QQ(\zeta_p)$ at $p$. This completes the claim.
Thus we can construct an unramified extension $L$ of $\QQ(\zeta_p)$ of $p$-power degree such that $\Gal(L/\QQ)$ acts on $\Gal(L/\QQ(\zeta_p))$ via $\omega^{l_1-l_2}$. Hence, by the class field theory, we prove that 
$\calc_p^{\omega^{l_1-l_2}}$ is  non-trivial.

(2)  Similarly to (1), $h$  gives a non-split extension of $\FF[G_{\QQ}]$-modules
\[0 \to \FF(l_1) \to \bar X \to \FF(l_2) \to 0,\]
and belongs to $H^1(\QQ,\FF(l_1-l_2))$.
 Let $\ell \not=p$. Then, by (a) we can prove that $h|_{D_{\ell}} \in \mathit{H}_f^1(\QQ_{\ell},\FF(l_1-l_2))$ by using the same argument as in \cite[Section 8]{Brown07}. 
Moreover, by (b) and Lemma \ref{lem.extension-FL}, we have $h|_{D_p} \in H_f^1(\QQ_p,\FF(l_1-l_2))$.  
Now we have the following isomorphism
\[\iota:\mathit{H}^1(\QQ,\FF(l_1-l_2)) \cong \mathit{H}^1(\QQ,\FF_p(l_1-l_2))^{ \oplus d}  \]
with $d=[\FF:\FF_p]$. Let $(h_1,\ldots,h_d)=\iota(h)$. 
Then for any $i=1,\ldots,d$, we have
\[h_i|_{D_{\ell}} \in \mathit{H}_f^1(\QQ_{\ell},\FF_p(l_1-l_2)) \text{ for any }\ell .\]
Hence, by \cite[Proposition 7.9]{Brown07}, we have
\[h_i \in \mathrm{Sel}(\QQ,\ZZ(l_1-l_2)).\]
This implies that $\mathrm{Sel}(\QQ,\ZZ(l_1-l_2))$
has $p$-torsion element.  
Since $l_1-l_2$ is a positive even integer,  $\mathrm{Sel}(\QQ,\ZZ(l_1-l_2))$ coincides with the Tate-Shafarevich group $\mathcyr{Sh}(\ZZ(l_1-l_2))$. Thus the assertion follows from
\cite[(6.5)]{Bloch-Kato90} remarking that $p$ does not divide $\#H^0(\QQ,\QQ/\ZZ(l_1-l_2))$. 
\end{proof}
Let $F$ and $G$ be Hecke eigenforms in $M_{\bf k}(\varGamma^{(m)})$, and $\frkp$ be a prime ideal of $\QQ(F)$. Then, we write
\[\rho_{G,\St} \equiv \rho_{F,St} \mod \frkp \]
if there exists a prime ideal $\widetilde \frkp$ of $\QQ(F)\cdot\QQ(G)$ lying above $\frkp$ such that
\[\rho_{G,\St} \equiv \rho_{F,St} \mod {\widetilde \frkp}. \]
In this case, we also write
\[\bar \rho_{G,\St} = \bar \rho_{F,\St}.\]

We characterize some cuspidal Hecke eigenforms  in terms of Arthur's parameters.
\begin{proposition} \label{prop.A-parameters}
Let $k,l$ be positive even integers such that $k \ge l \ge 6$.
\begin{itemize}
\item[(1)] Let $G$ be a tempered form in $S_k(\varGamma^{(2)})$ or in $S_{(k,k,l)}(\varGamma^{(3)})$. 
Then $G$ is stable. Let $G$ be a tempered form in $S_{(k,k,l,l)}(\varGamma^{(4)})$. Then, $G$ is stable or $\psi_G=\tau_1 \boxplus 1$ with
$\tau_1$  an irreducible unitary self-dual automorphic cuspidal representation of $\mathrm{PGL}_8(\AAA_\QQ)$.
\item[(2)] Let $G$ be a Hecke eigenform in  $S_{(k,k,l,l)}(\varGamma^{(4)})$ such that $\psi_G=\pi[2] \boxplus 1$ with
$\pi$ an irreducible unitary self-dual cuspidal automorphic representation of $\mathrm{PGL}_4(\AAA_\QQ)$. Then, $G$ is the lift of type $\scra^{(I)}$, that is
there exists a Hecke eigenform $F$ in $S_{(k+l-4,k-l+4)}(\varGamma^{(2)})$ such that $G=\scra_4^I(F)$.
\end{itemize}
\end{proposition}
\begin{proof}
The assertions follow from \cite[Theorem A.1]{A-C-I-K-Y23}.
\end{proof}

Let $g$ be a primitive form in $S_k(\SL_2(\ZZ))$. Let $K$ be a number field containing $\QQ(g)$ and $\frkp$ be a prime ideal of $K$ and put $p=p_{\frkp}$. Let $\rho_g$  be the $\frkp$-adic Galois representation of $G_{\QQ}$ in Theorem \ref{th.Galois-spin}, and $V$ be its representation space. Let $T$ be a $G_{\QQ}$-stable lattice of $V$.  

To state Kato's result on the Selmer group, we define the period $\Omega (g)_\pm=\Omega_{\mathrm{Kato}}(g)_{ \pm}$ following \cite{Kato04}. For a positive integer $L \ge 3$, let $Y(L)$ be the  open modular curve over $\QQ$  for $\varGamma(L)$ and $X(L)$ its compactification. For  integers $M,N \ge 1$, take an integer $L$ such that $L \ge 5$ divisible by $M$ and $N$, and define $Y(M,N)$ as
\[Y(M,N)=G \backslash Y(L)\]
with 
\[G=\left\{\begin{pmatrix} a & b \\ c & d \end{pmatrix} \ \middle| \ a \equiv 1\mod M, b \equiv 0 \mod M, c \equiv 0 \mod N, d \equiv 1 \mod N \right\}.\]
We note that $Y(M,N)$ is independent of the choice of $L$. We also note that $Y(N,N)=Y(N)$ and $Y(1,N)$ is the modular curve $Y_1(N)$ for $\varGamma_1(N)$. Let $X(M,N)$ be the smooth compactification of $Y(M,N)$. Let $\lambda: E \to Y(N)$ be the universal elliptic curve, and let $\bar \lambda: \bar E \to X(N)$ be the smooth N\'eron model of $E$ over $X(N)$. Let $\Omega^1_{\bar E/X(N)}$ denote the space of $1$-forms of $\bar E$ over $X(N)$, and put $\mathrm{coLie}(\bar E)=\bar \lambda_*(\Omega^1_{\bar E/X(N)})$. We then define the space $M_k(X(N))$ of algebraic modular forms weight $k$ for $X(N)$ (or for $\varGamma(N))$ as 
\[M_k(X(N))=\Gamma(X(N), \mathrm{coLie}(\bar E)^{\otimes k}),\]
and its subspace $S_k(X(N))$, called  the space of algebraic cusp forms of weight $k$ for $X(N)$ as
\[S_k(X(N))=\Gamma(X(N),  \mathrm{coLie}(\bar E)^{\otimes k-2} \otimes_{{\calo(X(N))}} \Omega^1_{X(N)/\QQ}).\]
For a curve $X$ of the form $X=G \backslash X(N)$ with $N \ge 3$ and $G$ a subgroup of $\GL_2(\ZZ/N\ZZ)$, we define $M_k(X)$ and $S_k(X)$ as the $G$-fixed part of $M_k(X(N))$ and $S_k(X(N))$, respectively. Let $\lambda: E \to Y(M,N)$ be the universal elliptic curve and define a local system $\calh^1$ on $Y(M,N)(\CC)$ as
\[\calh^1=R^1\lambda_*(\ZZ).\]
For a commutative ring $A$, let
\[V_{k,A}(Y(M,N))=\mathit{H}^1(Y(M,N)(\CC), \mathrm{Sym}_\ZZ^{k-2}(\calh^1) \otimes A),\]

\[V_{k,A,c}(Y(M,N))=\mathit{H}_c^1(Y(M,N)(\CC), \mathrm{Sym}_\ZZ^{k-2}(\calh^1) \otimes A)\]
and 
\[V_{k,A}(X(M,N))=\mathit{H}^1(X(M,N)(\CC), j_*\mathrm{Sym}_\ZZ^{k-2}(\calh^1) \otimes A),\]
where $j$ is the inclusion map $Y(M,N) \to X(M,N)$.
Let
\[\mathrm{per}_{1,N}:M_k(X_1(N))  \to V_{k,\CC}(Y_1(N))\]
be the period map in \cite[4.10]{Kato04}.
We note that $S_k(X_1(N)) \otimes \CC$ can be identified with $S_k(\varGamma_1(N))$. Let $g(z)=\sum_{m=1}^\infty a(m,g) {\bf e}(mz)$ be a new form in $S_k(\varGamma_1(N))$. Let $S(g)$ and $V_{\QQ(g)}(g)$ be the  quotient $\QQ (g)$-vector  spaces of $M_k(X_1(N))\otimes \mathbb{Q}(g)$ and $V_{k,\QQ}(Y_1(N))\otimes \mathbb{Q}(g)$, respectively, defined in \cite[6.3]{Kato04}, and put $V_\CC(g)=V_{\QQ(g)}(g) \otimes \CC$. 
Let $\iota : V_{k,\mathbb{Z}}(Y_1(N))\to V_{k,\mathbb{Z}}(Y_1(N))$ be the map induced by the complex conjugation on $Y_1(N)(\mathbb{C})$ and on $E(\mathbb{C})$.
We denote the $\mathbb{C}$-linear automorphism on $V_{k,\mathbb{C}}(Y_1(N))$ induced by $\iota$ by the same letter $\iota$.
For $x\in V_{k,\mathbb{C}}(Y_1(N))$, we put
$\displaystyle x^\pm =\frac{1}{2}(1+\iota)(x)$.
%The complex conjugation $\iota$ on $V_{k,\QQ}(Y_1(N))$ induces a $\CC$-linear map $\iota:V_{\CC}(g) \to V_{\CC}(g)$.
Then we have
\[\dim_{\QQ(g)} V_{\mathbb{Q}(g)}(g)=2, \ \dim_{\QQ(g)} V_{\mathbb{Q}(g)}(g)^+=\dim_{\QQ(g)} V_{\mathbb{Q}(g)}(g)^-=1.\]
Then the map $\mathrm{per}_{1,N}:M_k(X_1(N))  \to V_{k,\CC}(Y_1(N))$ induces a $\QQ(g)$-linear map
\[\mathrm{per}_g: S(g) \to V_{k,\CC}(g).\]
Let $N' \ge 3$ be an integer such that $N\mid N'$ and $p\nmid N'$. Let $X$ be an integral model of the Kuga-Sato variety $\mathrm{KS}_{k-2}(N')$,
which is a proper smooth scheme over $\ZZ[1/N']$. We will give a brief explanation of the Kuga-Sato varieties in Appendix A.
% such that $X \otimes_{\ZZ[1/N']} \QQ \cong \mathrm{KS}_k(N')$, where $\mathrm{KS}_k(N')$ is the Kuga-Sato variety of weight $k$ of level $N'$ (cf. \cite[p. 203]{Kato04}). 
 Let $D$ be the $\frkO_{\frkp}$-lattice generated by the image of
\[\mathit{H}^{k-1}(X,\Omega_{X/\ZZ [1\slash N']}^{\ge k-1})(\widetilde\varepsilon) \to S(g),\]
and $T$ be the image of
\[\mathit{H}^{k-1}(\mathrm{KS}_{k-2}(N') \otimes_{\QQ} \bar \QQ_p,\frkO_{\frkp})(\widetilde\varepsilon) \to V_{\QQ(g)_{\frkp}}(g).\]
(For the precise notation, see \cite[pp. 249--250]{Kato04}.)
Let $\omega$ be an $\frkO_{\frkp}$-basis of $D$, and $\gamma^\pm$ be an $\frkO_{\frkp}$-basis of $T^{\pm}$. Then we have
\[\mathrm{per}_g(\omega)^{\pm}=\Omega^\pm \gamma^\pm\]
with $\Omega^\pm \in \CC^\times$. (In \cite{Kato04}, $\Omega^\pm$ is written simply as $\Omega$.)
We then define $\Omega (g)_{\pm}=\Omega_{\mathrm{Kato}}(g)_{\pm}$ as
\[\Omega (g)_{\pm}=(2\pi \sqrt{-1})^{1-k} \Omega^{\mp}.\]
 The following theorem is due to Kato \cite{Kato04}. 
\begin{theorem} \label{th.Kato}
Let $g$ be a primitive form in $S_l(\SL_2(\ZZ))$. Let $K$ be a number field containing $\QQ(g)$ and $\frkp$ be a prime ideal of $K$.
Let $(\rho_g,V)$ be the $\frkp$-adic Galois representation attached to $g$ in Theorem \ref{th.Galois-spin}.
Let $m$ be a  positive integer such that $m \le l-1$ and $m \not=l/2$.
Let $T$ be a $G_{\QQ}$-stable lattice of $V(m)$. 
Suppose that the following three conditions hold:
\begin{itemize}
\item[(1)] The residual representation $\rho_g$ mod $\frkp$ is absolutely irreducible;
\item[(2)] $\rho_g(\Gal(\bar \QQ/\QQ(\zeta_{p^{\infty}}))$ contains $\SL_2(\ZZ_p)$ (with a suitable choice of a lattice of $V$);
\item[(3)] $\mathrm{Sel}(\QQ,T)$ has a $\frkp$-torsion element. 
\end{itemize}
Then $\frkp$ divides ${\bf L}(m,g;\Omega(g)_{(-1)^m})$.
\end{theorem}
\begin{proof}
We define $\mathcal{S}(\QQ,T)$ as
\[\mathcal{S}(\QQ,T)=\mathrm{Ker} \Bigl(H^1(\ZZ[1/p],T \otimes \QQ/\ZZ) \to H^1(\QQ_p,T \otimes \QQ/\ZZ)/H_f^1(\QQ_p,T \otimes \QQ/\ZZ) \Bigr).\]
Then $\mathrm{Sel}(\QQ,T)$ is a subgroup of $\mathcal{S}(\QQ,T)$.
By \cite[Proposition 14.21 (2)]{Kato04}, we have
\begin{align*}
\# (\mathcal{S}(\QQ,T))=&\mu^{-1}\cdot[\frkO_\frkp:{\bf L}(m,g;\Omega(g)_{(-1)^m})]\\
& \times \#H^0(\QQ, T \otimes \QQ/\ZZ) \cdot \#H^0(\QQ,T^*(1) \otimes \QQ/\ZZ),
\end{align*}
where $\mu=[H^1(\ZZ[1/p],T):z]\cdot \#(H^2(\ZZ[1/p],T)^{-1}$
with $z$ the element of $H^1(\ZZ[1/p],V_{K_\frkp}(g)(m))$ in \cite[Proposition 14.16]{Kato04}.
By the assumption (1), and by the proof of \cite[Proposition 14.11]{Kato04}, we see that $\#H^0(\QQ, T \otimes \QQ/\ZZ) = \#H^0(\QQ,T^*(1) \otimes \QQ/\ZZ)=1$. Moreover, by the assumption (2) and \cite[Theorem 14.5 (3)]{Kato04}, we have $\mu \ge 1$.
Thus the assertion has been proved remarking that $\mathrm{Sel}(\QQ,T) \subset \mathcal{S}(\QQ,T)$.
\end{proof}

  \begin{lemma} \label{lem.non-trivial-extension2}
Let $\frkp$ be a prime ideal of a number field $K$ and $p=p_\frkp$. Let 
$\rho: G_{\QQ} \to GL_r(K_{\frkp})$ be a Galois representation. 
Let $g$ be a primitive form in $S_l(\SL_2(\ZZ))$ with  $l<p/2$. Let $(\rho_g,V_g)$ be the Galois representation in Theorem \ref{th.Galois-spin}.
Let $T_g$ be a $G_{\QQ}$-stable lattice of $V_g$.

\begin{itemize}
\item[(1)] 
 Suppose that $\rho$ is unramified outside $p$ and short crystalline at $p$ and that 
 \[\bar \rho=\begin{pmatrix}                      
* & * &* & * \\
 0 & \bar \rho_g(l_1) & a &* \\
 0 & 0 & \bar \chi^{l_2} &* \\
 0 & 0 & 0 & *
\end{pmatrix}\]
with  $l_1,l_2$ integers such that $\{1-l+l_1,l_1, l_2 \} \subset  \mathrm{HT}(\rho),  1 \le l_1-l_2 \le l-1, l_1-l_2 \not=l/2$ and $a \not='0$. Then
$\mathrm{Sel}(\QQ,T_g(l_1-l_2))$ has a $\frkp$-torsion element.
\item[(2)] 
Suppose that $\rho^\vee \otimes \chi^{-l}$ is unramified outside $p$ and  short crystalline at $p$ and that 
\[\bar \rho=\begin{pmatrix}                    
 * & * &* & * \\
 0 & \bar \chi^{l_1}& a &* \\
 0 & 0 & \bar  \rho_g(l_2)   &* \\
0 & 0 & 0 & *
\end{pmatrix}\]
with $l_1, l_2$  integers such that $\{-l-l_1, -l-l_2, -1-l_2 \} \subset  \mathrm{HT}(\rho \otimes \chi^{-l}), 1 \le l-1-l_2+l_1 \le l-l_2+l_1-1 \le l-1,  l-1-l_2+l_1 \le l-l_2+l_1-1 \not=l/2$ and  $a \not=' 0$. Then
$\mathrm{Sel}(\QQ,T_g(l+l_1-l_2-1))$  has a $\frkp$-torsion element.
\item[(3)]  Suppose that $\rho$ is unramified outside $p$ and short crystalline at $p$ and that
\[\bar \rho=\begin{pmatrix}                     
 * & * &* & * \\
 0 & \bar \rho_g^\vee(-l) & a &* \\
 0 & 0 & \bar  \rho_g   &* \\
 0 & 0 & 0 & *
\end{pmatrix}\]
with  $a \in M_2(\FF)$.  Moreover suppose that  $\{-l,1-l,-1,0 \} \subset  \mathrm{HT}(\rho)$ and $a \not='0$ and $a=-\bar\rho_g^\vee {}^ta\bar\rho_g$. Then, we have $\# \calc_p^{\omega^{-1}} \not=1$.

\end{itemize}

\end{lemma}
\begin{proof} (1)  Similarly to Lemma \ref{lem.non-trivial-extension1}, $a$ gives a non-trivial extension of $\FF[G_{\QQ}]$-modules:
\[0 \to \bar V_g(l_1) \to \bar X \to  \FF(l_2) \to 0.\]
This gives rise to a $\frkp$-torsion element $h$ of $H^1(\QQ,(V_g/T_g)(l_1-l_2))$.
Since $\rho$ is unramified outside $p$ and short crystalline at $p$, by using the same argument as in (2) of Lemma  \ref{lem.non-trivial-extension1} (or Brown \cite{Brown07}) combined with Lemma \ref{lem.extension-FL}, we see that 
$h$ belongs to $\mathrm{Sel}(\QQ,T_g(l_1-l_2))$. This proves the assertion.

(2)  Similarly to (1), $a$ gives a non-trivial extension
\[0 \to \FF(l_1) \to \bar X \to \bar V_g(l_2) \to 0.\]
Hence we have the following non-trivial extension
\[0 \to \bar V_g(l_2)^\vee \otimes \bar \chi^{-l} \to \bar X^\vee \otimes \chi^{-l} \to  \FF(l_1)^\vee \otimes \chi^{-l} \to 0.\]
We note that  $\bar V_g(l_2)^\vee =\bar V_g(l-1-l_2)$ and 
$\FF(l_1)^\vee =\FF(-l_1)$. 
Hence we have the following non-trivial extension
\[0 \to \bar V_g(-l_2-1) \to \bar X^\vee \otimes \bar \chi^{-l}   \to \FF(-l_1-l) \to 0.\]
This gives rise to a $\frkp$-torsion element $h$ of $H^1(\QQ,(V_g/T_g)(l+l_1-l_2-1))$.
Since $\rho \otimes \chi^{-l}$ is unramified outside $p$ and short crystalline at $p$, 
 in a way similar to (1), we see that $h$ belongs to 
$\mathrm{Sel}(\QQ,T_g(l+l_1-l_2-1))$.

(3) Put $\rho_0=\rho_g(l/2)$. Then, for any $\sigma, \tau \in G_{\QQ}$ we have
\[a(\sigma \tau)=\bar \rho_0^\vee(\sigma)a(\tau)+a(\sigma)\bar \rho_0(\tau).\]
Put $b=a\bar\rho_0^{-1}$. Then, we have ${}^tb=-b$, and put
$b=\begin{pmatrix} 0 & \beta \\ -\beta & 0 \end{pmatrix}$.
Then, by a simple computation, we have
\[\beta(\sigma \tau) =\beta(\sigma) + \beta(\tau) \bar \chi^{-1}(\sigma).\] 
This implies that $\beta$ gives rise to a non-trivial extension of $\FF[G_{\QQ}]$-modules:
\[0 \to \FF(-1) \to \bar X \to \FF \to 0.\]
Then by  using the same argument as in (1) of Lemma \ref{lem.non-trivial-extension1}, we can prove the assertion.
\end{proof}

\begin{lemma} \label{lem.det-of-irreducible-factor}
Let ${\bf k}=(k_1,\ldots,k_n)$ be a non-increasing sequence of integers such that $k_n \ge n+1$.  Let $G$ be a tempered form in $S_{\bf k}(\varGamma^{(n)})$ and
$\rho_{G,\St}$ the $\frkp$-adic Galois representation in Theorem \ref{th.Galois-spin}, Let $\rho'$ be an irreducible factor of $\rho_{G,\St}$ of dimension $r$. Then $\wedge^r \rho'=1$.
\end{lemma}
\begin{proof}
Let $\rho''$ be the complement of $\rho'$, that is
\[\rho_{G,\St}=\rho' \oplus \rho''.\]
Then for any prime number $\ell \not=p_{\frkp}$, we have
\[\det (1_r-\rho'(\mathrm{Frob}_{\ell}^{-1})X) \det(1_{n-r}-\rho''(\mathrm{Frob}_{\ell}^{-1})X)=
L_{\ell}(X,G,\St).\]
Let $\alpha_0,\alpha_1,\ldots,\alpha_n$ be the $\ell$-Satake parameters of $G$ in Section 3. Then we have
\[\det \rho'(\mathrm{Frob}_{\ell}^{-1})={\ell}^i=\alpha_{i_1}^{j_1}\cdots \alpha_{i_r}^{j_r}\]
with $i \in \ZZ, 1 \le i_1 < \cdots <i_r \le n$ and $j_k=\pm1 $ for $k=1,\ldots,r$. Since $G$ is tempered, the  Ramanujan conjecture holds (cf. \cite[Theorem 8.2.17]{Chenevier-Lannes19}). Hence 
\[{\ell}^i=|\alpha_{i_1}^{j_1}\cdots \alpha_{i_r}^{j_r}|=1.\]
Hence, by the Chebotarev density theorem, we  prove the assertion.
\end{proof}
\begin{lemma}\label{lem.no-irreducible-factor}
Let $G$ be a Hecke eigenform in $S_{\bf k}(\varGamma^{(n)})$ with ${ \bf k}=(k_1,\ldots,k_n)$, and $\rho_{G,\St}$ the Galois representation attached to $G$. Suppose that $\rho_{G,\St}$ has an  irreducible factor $\tau$  such that $\bar \tau=\bar \chi^m$ with $m$ an integer  such that $(1-p)/2 \le m \le (p-1)/2$. Then we have $(1-n)/4 <m <(n-1)/4$.
\end{lemma}
\begin{proof}
Let 
\[\psi(\pi_{G,\St})=\oplus_{i=1}^r \pi_i[d_i]\]
be the global Arthur parameter associated with the standard representation $\pi_{G,\St}$  stated in Section 4.
By Arthur's multiplicity formula, there is an integer $i_0$ such that
\begin{align*}
 & d_{i_0}=1 \text { and } n_{i_0} \equiv 1 \mod 2, 
\text { and }  n_id_i  \equiv 0 \mod 4 \text { for any } i \not=i_0 \tag A
\end{align*}
(cf. \cite[Appendix A]{A-C-I-K-Y23}).
By Theorem \ref{th.Galois-st}, for any prime number $\ell \not=p$ we have
\begin{align*}
&\det(I_{2n+1}-X\rho_{F,\St}({\rm Frob}_\ell^{-1}) )          \\
&=\prod_{i=1}^r\prod_{j=0}^{d_i-1}
\det(I_{n_i}-\ell^{j+\frac{1-d_i}{2}}X\rho_{\pi_i,p}({\rm Frob}_\ell^{-1})).\\
\end{align*}
Suppose that $\tau$ is  an irreducible factor of $\rho_{G,\St}$ such that $\bar \tau=\bar \chi^m$ with $m$ an integer such that  $(1-p)/2 \le m \le (p-1)/2$. Since $\tau$ is unramified outside $p$ and crystalline at $p$, we have $\tau=\chi^m$. 
Since $1-X\tau({\rm Frob}_\ell^{-1})$ is a factor of $\det(I_{2n+1}-X\rho_{F,\St}({\rm Frob}_\ell^{-1}) )$, we have
\begin{align*}
1-X\tau({\rm Frob}_\ell^{-1})  &=1-X{\ell}^{-m}\\
&=1-X\ell^{j+\frac{1-d_i}{2}}\alpha_{t_0}(\ell,\pi_i)
\end{align*}
for some $1 \le i \le r, \ 1 \le j \le d_i$, and $1 \le t_0 \le n_i$, where $\{\alpha_t(\ell,\pi_i)\}_{1 \le t \le n_i}$ is the $\ell$-Satake parameters of $\pi_i$. Since $\pi_i$ satisfies the Ramanujan conjecture,  $\alpha_{t_0}(\ell,\pi_i)=1$ and $d_i$ is odd. Hence,  by (A), $n_i$ is a multiple of $4$. Hence we have
\[|m| \le (d_i-1)/2 \le ((2n+1)/4-1)/2=n/4-3/8.\]
 This completes the proof.
\end{proof}

\section{Proof of Theorem \ref{th.main-result}}
In this section, we  prove our first main result. 

%Let $(\rho,V)$ be the $\frkp$-adic Galois representation of $G_{\QQ}$ over $K$ of dimension $n$.%

\begin{theorem} \label{th.noncongruence}
Let $f$ be a primitive form in $S_{2k-2}(\SL_2(\ZZ))$ with $k$ even, and
$\frkp$ a prime ideal of $\QQ(f)$ and put $p=p_{\frkp}$. Suppose that the following two conditions hold:
\begin{itemize}
\item[(a)]  $p > 2k-2$;
\item[(b)] $\rho_f(\Gal(\bar \QQ/\QQ(\zeta_{p^{\infty}})))$ contains $\SL_2(\ZZ_p)$ (with a suitable choice of a lattice of $V_f$).
%\item[(c)]  $\frkp$ does not divide $\frkD_f$.%
\end{itemize}
\begin{itemize}
\item[(1)] For  any non-cuspidal Hecke eigenform $G \in M_k(\varGamma^{(2)})$ we have
 \[\rho_{G,\St} \not\equiv \rho_{\scri_2(f),\St} \ \mod \frkp.\]
\item[(2)] Suppose that $\frkp$ divides neither  $\frkD_f^{(p)}$ nor ${\bf L}(k,f,\Omega_+(f))$. Then, for any Hecke eigenform $G \in S_k(\varGamma^{(2)})$ such that $G \not\in \langle \scri_2(f) \rangle_\CC$, we have  
\[\rho_{G,\St} \not\equiv \rho_{\scri_2(f),\St} \ \mod \frkp.\]
\item[(3)] Let $G$ be a Hecke eigenform in $S_{(k,k,l)}(\varGamma^{(3)})$ with $k \ge l \ge 6$ and $l$ even.
Suppose that the following holds:
\begin{itemize}
\item[(3.1)] $\#\calc_p^{\omega^{6-2l}}=1$;
\item[(3.2)] $p$ does not divide $\zeta(3-l)$;
\item[(3.3)] $\frkp$ divides none of the $L$-values  
\[{\bf L}(k,f;\Omega_+(f)), {\bf L}(k+l-3,f;\Omega_-(f)) \text{ and } {\bf L}(k+l-5,f;\Omega_-(f)).\]
\end{itemize}
Then, 
\[\rho_{G,\St} \not\equiv \rho_{[\scri_2(f)]^{(k,k,l)},\St} \ \mod \frkp.\]
\item[(4)] Let $G$ be a Hecke eigenform
in $S_{(k,k,l,l)}(\varGamma^{(4)})$ with $k \ge l \ge 6$ and $l$ even.
Suppose that the following holds:
\begin{itemize}
\item[(4.1)]  \[\#\calc_p^{\omega^{4-l}}=\#\calc_p^{\omega^{6-2l}}=\#\calc_p^{\omega^{8-2l}}=1;\]
\item[(4.2)]  $p$ divides  none of the values
\[\zeta(3-l), \zeta(7-2l), \zeta(9-2l);\]
\item[(4.3)]  $\frkp$ divides none of the $L$-values 
\[{\bf L}(k,f;\Omega_+(f)), \ {\bf L}(k+l-3,f;\Omega_-(f)), \  {\bf L}(k+l-5,f;\Omega_{-}(f)) \text{ and } {\bf L}(k+l-6,f;\Omega_+(f)).\] 
\end{itemize}
Then, 
\[\rho_{G,\St} \not\equiv \rho_{[\scri_2(f)]^{(k,k,l,l)},\St}  \mod \frkp\]
if $G$ is not the lift of type $\scra^{(I)}$.
\end{itemize}
Throughout (1) $\sim$ (4), we note that $\QQ(\scri_2(f))=\QQ([\scri_2(f)]^{(k,k,l)})=\QQ([\scri_2(f)]^{(k,k,l,l)})=\QQ(f)$.
\end{theorem}
\begin{proof}

First we prove (1).  Let 
 $g=\Phi_1^2(G)$ be a Hecke eigenform in $M_k(\SL_2(\ZZ))$ and suppose that 
\[\rho_{G,\St} \equiv \rho_{\scri_2(f),\St} \mod \frkp.\]
Then, 
\[\bar \rho_{G,\St}^{\mathrm{ss}} = \rho_{g,\St}^{\mathrm{ss}} \oplus \bar \chi^{k-2} \oplus \bar \chi^{2-k}.\]
On the other hand, 
\[\bar \rho_{\scri_2(f),\St}^{\mathrm{ss}} =\bar \rho_f(k-1) \oplus \bar \rho_f(k-2) \oplus \bar 1,\]
and $\bar \rho_f(k-1)$ and $ \bar \rho_f(k-2)$ are irreducible. This proves the assertion.

Next  we prove (4).  Suppose that 
\[\rho_{G,\St} \equiv \rho_{[\scri_2(f)]^{(k,k,l,l)},\St} \ \mod \frkp\]
for a  Hecke eigenform $G$ in $S_{(k,k,l,l)}(\varGamma^{(4)})$. 
Let $(\rho,V)$ be the $\frkp$-adic Galois representation defined by $\rho=\rho_{G,\St} \otimes \chi^{1-k}$.
By Theorem \ref{th.Galois-st}, $\rho$ is unramified outside $p$ and short crystalline at $p$. 
We have $\rho_f^\vee(2-2k)=\begin{pmatrix} 0 & 1 \\ -1 & 0 \end{pmatrix}\rho_f(-1)\begin{pmatrix} 0 & -1 \\ 1 & 0 \end{pmatrix} \cong \rho_f(-1)$.
Hence we have  
\begin{align*}
\bar \rho^{{\rm ss}} &=\bar \rho_f \oplus \bar \rho_f^\vee(2-2k) \oplus \bar \chi^{l-k-2} \oplus \bar \chi^{4-l-k}  \oplus \bar \chi^{l-k-3}  \oplus \bar \chi^{5-l-k} \oplus \bar  \chi^{1-k}\\
&=(\bar \rho_{[\scri_2(f)]^{(k,k,l,l)},\St} \otimes \bar \chi^{1-k})^{\rm ss}.
\end{align*}
Put  
\[\calf=\{\bar \rho_f, \ \bar \rho_f^\vee(2-2k), \ \bar \chi^{l-k-2}, \  \bar \chi^{4-l-k}, \ \bar \chi^{l-k-3}, \ \bar \chi^{5-l-k}, \bar \chi^{1-k} \}.\]

Let $\tau$ be an irreducible factor of $\rho$ with the greatest dimension. Then,  the following assertion  (I.i) holds for some $1 \le i \le 9$:

(I.i) We have $\dim \tau=i$ and 
\[\bar \tau^{ss}=R_1 \oplus \cdots \oplus R_r \ \text{ with } r \ge 1 \text{ and } R_j \in \calf.\]
Since $\bar \rho_f$ and $\bar \rho_f^\vee(2-2k)$ are irreducible, (I.1) does not hold. Suppose that we have $\dim \tau=2$. Then, $\bar \tau^{ss}=
\bar \chi^a \oplus \bar \chi^b$  with $a,b \in \{1-k,\pm(l-4)+1-k, \pm(l-3)+1-k\}, a \not =b$ or $\bar \rho_f$ or $\bar \rho_f^\vee(2-2k)$. By the assumptions (4.1) and (4.2), and Lemma \ref{lem.non-trivial-extension1}, we see that the first case does not hold. In view of Lemma \ref{lem.no-irreducible-factor}, there is no non-trivial one dimensional irreducible factor of $\rho$. Hence, in the second and the third cases, there is another irreducible factor $\tau'$ of $\rho$ such that
 $\bar \tau'^{ss}=\bar \chi^a \oplus \bar \chi^b$ with $a,b \in \{0,\pm(l-4)+1-k, \pm(l-3)+1-k\}, a \not =b$, which is reduced to the first case. Therefore, (I.2) does not hold.
Similarly, by the assumptions (4.1), (4.2), (4.3), Lemmas \ref{lem.non-trivial-extension1} and \ref{lem.non-trivial-extension2}, and Corollary \ref{cor.no-factor}, we easily see that  (I.3) does not  hold.

Suppose that (I.9) holds. Then, for a suitable lattice $T$ of $V$, $\bar \rho_T$ is expressed as 
\[\bar \rho_T=\begin{pmatrix}  
a_{11} & a_{12} & a_{13} & a_{14} & a_{15} & a_{16}  & a_{17} \\                     
0 & a_{22} & a_{23} & a_{24} & a_{25} & a_{26} & a_{27} \\
0 & 0      & a_{33} & a_{34} &a_{35} & a_{36} & a_{37}\\
0 & 0 & 0 & a_{44} & a_{45} & a_{46} & a_{47}\\
0 & 0 & 0 & 0 & a_{55} & a_{56} & a_{57}\\
0 & 0 & 0 & 0 & 0 & a_{66} & a_{67} \\
0 & 0 & 0 & 0 & 0 & 0       & a_{77}
\end{pmatrix}\]
with $a_{ii} \in \calf  \quad (i=1,\ldots,7)$.
%For $\rho_1, \rho_2 \in \calf$, we write $\rho_1 < \rho_2 \ \text { in } T$ if $\rho_1=a_i, \rho_2=a_j$ with $i<j$. %
First we claim that there exists a lattice $T_0$ of $V$ satisfying the following conditions:
\begin{align*}
\bar \rho_{T_0}=\begin{pmatrix}  
\bar \chi^{l-k-3} & a_{12} & a_{13} & a_{14} & a_{15} & a_{16}  & a_{17} \\                     
0 & \bar  \chi^{4-l-k} & a_{23} & a_{24} & a_{25} & a_{26} & a_{27} \\
0 & 0      & \ \bar \rho_f^\vee(2-2k) & a_{34} &a_{35} & a_{36} & a_{37}\\
0 & 0 & 0 & \bar \rho_f & a_{45} & a_{46} & a_{47}\\
0 & 0 & 0 & 0 &\bar \chi^{l-k-2}   & a_{56} & a_{57}\\
0 & 0 & 0 & 0 & 0 &\bar \chi^{5-l-k}  & a_{67} \\
0 & 0 & 0 & 0 & 0 & 0       & \bar \chi^{1-k}
\end{pmatrix} \tag{N0} \end{align*}
and 
\begin{align*}
a_{45} &\not=' 0 \tag{N1},\\
a_{56} &\not='0 \tag {N2},\\
\begin{pmatrix} a_{47} \\ a_{57} \\ a_{67} \end{pmatrix}&\not='\begin{pmatrix} 0 \\ 0 \\ 0 \end{pmatrix} \tag{N3},\\
a_{34} & \not='0. \tag{N4}
\end{align*}
By  Lemma \ref{lem.lattice-change}, there exists a lattice $T_1$ such that
\[\bar \rho_{T_1} =\begin{pmatrix}  
* & \hdots & * & \hdots & * & \hdots  & * \\                     
0 & \ddots & \vdots & \vdots & \vdots & \vdots & \vdots \\
0 & 0      & * & * & * & \vdots & *\\
0 & 0 & 0 & * & \vdots & \vdots & \vdots\\
0 & 0 & 0 & 0 & * & \vdots & \vdots\\
0 & 0 & 0 & 0 & 0 & a_{66} & a_{67} \\
0 & 0 & 0 & 0 & 0 & 0       & \bar \chi^{5-l-k}
\end{pmatrix}\]
with $a_{66} \in \calf \backslash \{\bar \chi^{5-l-k} \}$.
We have 
\[\mathrm{Ext}(\bar \chi^{5-l-k},a)=0\]
for any $a \in \calf \backslash \{\bar \chi^{5-l-k}, \bar \chi^{l-k-2} \}$. Therefore, we have $a_{67} ='0$ if $a \not=\bar \chi^{l-k-2}$.
Therefore, by  using (2) of Lemma \ref{lem.lattice-change} repeatedly, we can show that there exists a lattice $T_2$ such that
\[\bar \rho_{T_2} =\begin{pmatrix}  
* & \hdots & * & \hdots & * & \hdots  & * \\                     
0 & \ddots & \vdots & \vdots & \vdots & \vdots & \vdots \\
0 & 0      & * & * & * & \vdots & *\\
0 & 0 & 0 & * & \vdots & \vdots & \vdots\\
0 & 0 & 0 & 0 & * & \vdots & \vdots\\
0 & 0 & 0 & 0 & 0 & \bar \chi^{l-k-2} & * \\
0 & 0 & 0 & 0 & 0 & 0       & \bar \chi^{5-l-k}
\end{pmatrix}.\]
Repeating this process, we see that there exists a lattice $T_3$ such that
\[\bar \rho_{T_3} =\begin{pmatrix}  
\bar \chi^{1-k} & * & * & * & * & *  & * \\                     
0 & \bar \chi^{l-k-3} & * & * & * & * & * \\
0 & 0      & \bar \chi^{4-l-k} & * & * & * & *\\
0 & 0 & 0 & \bar \rho_f^\vee(2-2k) & * & * & *\\
0 & 0 & 0 & 0 & \bar \rho_f & * & *\\
0 & 0 & 0 & 0 & 0 & \bar \chi^{l-k-2} & * \\
0 & 0 & 0 & 0 & 0 & 0       & \bar \chi^{5-l-k}
\end{pmatrix}.\]
By (1) of Lemma \ref{lem.lattice-change}, we see that there exists a lattice $T_4$ satisfying the condition (N0).
Suppose that we have $a_{45} ='0$  for any lattice $T$ satisfying the condition (N0). 
Since we have $\mathrm{Ext}(a,\bar \chi^{l-k-2})='0$ for any $a \in \calf \setminus \{\bar \rho_f \}$, by Corollary \ref{cor.no-factor2}
$\rho$ is not irreducible. Hence there exists a lattice $T_5$ satisfying the conditions (N0) and (N1).
Using the same argument as above repeatedly, we can show  that  there exists a lattice $T_6$ satisfying the conditions (N0), (N1), (N2) and (N3). 
Suppose that we have $a_{34} ='0$ for any lattice $T$ satisfying the conditions (N0), (N1), (N2) and (N3). Then, by using the same argument as above, we can show that 
\[(a_{ij})_{1 \le i \le 4, 3 \le j \le 7} ='O,
\]
and we derive a contradiction. This completes the claim.

Since $\rho_{G,\St}$ is self-dual, we have
$\rho^\vee=\rho \otimes \chi^{2k-2}.$ Hence we have 
\begin{align*}
\rho_{T_0}^\vee =U^{-1}\chi^{2k-2} \otimes  \rho_{T_0} U 
\end{align*}
with $U \in \GL_9(K_{\frkp})$. 
We may suppose that $U \in M_9(\frko_{\frkp})$ and
$\bar U \not= O$, and 
\begin{align*}
\bar U \bar \rho_{T_0}^\vee =\bar \chi^{2k-2} \otimes \bar \rho_{T_0} \bar U \tag{SD}.
\end{align*}
By Theorem \ref{th.sign-Galois-st}, we have ${}^tU=U$. Moreover, $\bar \rho_{T_0}^\vee$ is of the following form
\begin{align*} \bar \rho_{T_0}^\vee&=\begin{pmatrix} 
\bar \chi^{-l+k+3} & 0 & 0 & 0 & 0 & 0 & 0\\                     
*  & \bar \chi^{-4+l+k} & 0 & 0 & 0 & 0 & 0\\
* & * & \bar \rho_f (2k-2)& 0 & 0 &0  & 0\\
* & *& a_{43}^* & \bar \rho_f^\vee &0 &0 & 0\\
* & *& *& *&\bar \chi^{-l+k+2} &0 & 0\\
* & *& *& *& * & \bar \chi^{-5+l+k}&0 \\
* & * & * & * & * & * & \bar \chi^{k-1}
\end{pmatrix} \tag{$N^\vee$}\\
& \text{ with } a_{43}^*=-\rho_f^\vee \ {}^t\! a_{34}\rho_f(2k-2).
\end{align*}

Write $\bar U$ in a block form
\[
\begin{array}{cccccccccl}
&\hskip -1pt
\overbrace{\hphantom{u_{11}}}^1 \;
\overbrace{\hphantom{u_{12}}}^1  \; 
\overbrace{\hphantom{u_{13}}}^{2} \; 
\overbrace{\hphantom{u_{14}}}^{2} \;
\overbrace{\hphantom{u_{15}}}^1  \; 
\overbrace{\hphantom{u_{16}}}^1 \; 
\overbrace{\hphantom{u_{17}}}^1
\\ & 
\bar U=\left(
\begin{array}{cccccccl}
u_{11} & u_{12} & u_{13} &  u_{14} & u_{15} & u_{16} & u_{17}\\
u_{21} & u_{22} & u_{23}  & u_{24} & u_{25} &u_{26} & u_{27} \\
u_{31} &u_{32} & u_{33} & u_{34} & u_{35} & u_{36} & u_{37} \\
u_{41} &u_{42} & u_{43} & u_{44} & u_{45} & u_{46} & u_{47} \\
u_{51} &u_{52} & u_{53} & u_{54} & u_{55} &u_{56} & u_{57} \\
u_{61} &u_{62} & u_{63} & u_{64} & u_{65} & u_{66} & u_{67}\\
u_{71} &u_{72} & u_{73} & u_{74} & u_{75} & u_{76} & u_{77}\\
\end{array} \hskip -1pt
\right) \hskip -5pt
\begin{array}{l}
 \left.\vphantom{B_{11}} \right\} \text{\footnotesize$1$} \\
 \left.\vphantom{B_{11}} \right\} \text{\footnotesize$1$} \\
 \left.\vphantom{B_{11}} \right\} \text{\footnotesize$2$} \\
\left.\vphantom{B_{11}} \right\} \text{\footnotesize$2$} \\
 \left.\vphantom{B_{11}} \right\} \text{\footnotesize$1$} \\
 \left.\vphantom{B_{11}} \right\} \text{\footnotesize$1$} \\
\left.\vphantom{B_{11}} \right\} \text{\footnotesize$1$
} 
\end{array}
\end{array}
\]
with $u_{ji}={}^tu_{ij}$.
Then, by comparing the last columns of the both-hand sides of (SD), we have
\[R_1 \begin{pmatrix} u_{17} \\ \vdots \\ u_{67} \end{pmatrix} +u_{77}\begin{pmatrix} a_{17} \\ \vdots \\ a_{67} \end{pmatrix} =\bar \chi^{1-k} \begin{pmatrix} u_{17} \\ \vdots \\ u_{67} \end{pmatrix},\]
where 
\[R_1=\begin{pmatrix}  
\bar \chi^{l-k-3} & a_{12} & a_{13} & a_{14} & a_{15} & a_{16}    \\                  
0 & \bar  \chi^{4-l-k} & a_{23} & a_{24} & a_{25} & a_{26}  \\
0 & 0      & \ \bar \rho_f^\vee(2-2k) & a_{34} &a_{35} & a_{36} \\
0 & 0 & 0 & \bar \rho_f& a_{45} & a_{46} \\
0 & 0 & 0 & 0 &\bar \chi^{l-k-2}   & a_{56} \\
0 & 0 & 0 & 0 & 0 &\bar \chi^{5-l-k}   \\
\end{pmatrix}. \]
If $u_{77}\not=0,$ then 
 we have 
\[\begin{pmatrix} a_{17} \\ \vdots \\ a_{67} \end{pmatrix} ='\begin{pmatrix} 0 \\ \vdots \\ 0 \end{pmatrix}.\]
This contradicts the assumption (N3).  Hence $ u_{77}=0$. 
Since $\bar \rho_f(i)$ is irreducible for $i=0,-1$, by Theorem \ref{th.sign-Galois-st}, we have
\[\begin{pmatrix} u_{17} \\ \vdots \\ u_{67} \end{pmatrix}=\begin{pmatrix} 0 \\ \vdots \\ 0 \end{pmatrix} \text{ and } (u_{71},\ldots,u_{76})=(0,\ldots,0).\]
Repeating this process, we may suppose that we have 
\[\bar U=\begin{pmatrix} 
u_{11} & u_{12} & u_{13} & u_{14} & u_{15} & \vpi^{i_1} & 0 \\                     
u_{21} & u_{22} & u_{23} & u_{24} &  \vpi^{i_2} & 0 & 0\\
u_{31} & u_{32} & u_{33}   & \vpi^{i_3} {\bf 1}_2& 0 &0 & 0 \\
u_{41} & u_{42} &\vpi^{i_3}{\bf 1}_2 & 0     &0 &0 & 0\\
u_{51} &\vpi^{i_2}&0 &0 &0&0 & 0\\
\vpi^{i_1}&0&0& 0 & 0& 0 & 0 \\
0 & 0 & 0 & 0 & 0 & 0 & 0
\end{pmatrix} \]
with $i_m=0$ or $1$ for $m=1,2,3$ and $u_{ji}={}^tu_{ij}$.
Suppose that $i_3=0$. Then, without loss of generality we may assume $u_{33}=0$. Then we have  
\[a_{34}=-\bar \rho_f ^\vee{}^t\! a_{34}\bar \rho_f .\]
Hence, by (3) of Lemma \ref{lem.non-trivial-extension2}, we have $a_{34}='0$, which contradicts the assumption (N4). Hence $i_3=1$. 
Moreover, we have $u_{33}=0$.
Then, comparing the second columns of the both-hand sides 
of (SD), we have
\[R_2 \begin{pmatrix} u_{12} \\ u_{22} \\ u_{32} \\ u_{42} \end{pmatrix} +\varpi^{i_2}\begin{pmatrix} a_{15} \\ a_{25} \\ a_{35} \\ a_{45} \end{pmatrix} =\begin{pmatrix} u_{12} \\ u_{22} \\ u_{32} \\ u_{42} \end{pmatrix} \bar \chi^{l-k-2}\]
with
\[R_2=\begin{pmatrix}  
\bar \chi^{l-k-3} & a_{12} & a_{13} & a_{14}     \\                  
0 & \bar  \chi^{4-l-k} & a_{23} & a_{24}   \\
0 & 0      & \ \bar \rho_f^\vee(2-2k) & a_{34}  \\
0 & 0 & 0 & \bar \rho_f \\
\end{pmatrix},\]
and in particular we have
\[\bar \rho_f u_{24} +\varpi^{i_2} a_{45}=\bar \chi^{l-k-2}u_{24}.\]
If $i_2=0$, we have  $a_{45} ='0$, which contradicts the assumption (N1). Hence, $i_2=1$ and we also have
$\begin{pmatrix} u_{12} \\ u_{22} \\ u_{32} \\ u_{42} \end{pmatrix}=\begin{pmatrix} 0 \\ 0 \\ 0 \\ 0 \end{pmatrix}$
and $(u_{21},u_{22},u_{23},u_{24})=(0,0,0,0)$.
Then, using the same argument as above, we can prove that 
$i_1=1$, and 
$\begin{pmatrix} u_{11} \\ u_{21} \\ u_{31} \\ u_{41} \\ u_{51} \end{pmatrix}=\begin{pmatrix} 0 \\ 0 \\ 0 \\ 0 \\ 0\end{pmatrix}$
and $(u_{11},u_{12},u_{13},u_{14},u_{15})=(0,0,0,0,0)$.
This implies $\bar U=O$, which contradicts the assumption.

Next suppose that (I.8) holds.  Then, by using the same argument as above, we can show that there is a lattice $T$ of $\tau$ such that
\[\bar \tau_T=\begin{pmatrix}  
\bar \chi^{l-k-3} & a_{12} & a_{13} & a_{14} & a_{15} & a_{16}   \\                     
0 & \bar  \chi^{4-l-k} & a_{23} & a_{24} & a_{25} & a_{26}  \\
0 & 0      & \ \bar \rho_f^\vee(2-2k) & a_{34} &a_{35} & a_{36} \\
0 & 0 & 0 & \bar \rho_f & a_{45} & a_{46} \\
0 & 0 & 0 & 0 &\bar \chi^{l-k-2}   & a_{56} \\
0 & 0 & 0 & 0 & 0 &\bar \chi^{5-l-k}  
\end{pmatrix}\]
and 
\begin{align*}
a_{34} \not=' 0, \ a_{45} &\not='0, \  a_{56} \not='0.
\end{align*}
By Theorem \ref{th.sign-Galois-st}, we have  
\[U\tau_T^\vee= \chi^{2k-2} \otimes \tau_T U\]
with $U \in \GL_8(K_\frkp) \cap \Sym_8(\frkO_\frkp)$.
Then, using the same argument as above we can derive a contradiction.

Thirdly, suppose that (I.7) holds. Then we have
\[\rho=\tau \oplus \tau_1,\]
where $\tau_1$ is a subrepresentation of $\rho$ such that $\dim \tau_1=2$. Since $\bar \rho_f$ and $\bar \rho_f^\vee(2k-2)$ are irreducible and not self-dual, we have
\[\bar \tau^{\rm ss}=\bar \chi^{l-i+1-k} \oplus \bar \chi^{i-l+1-k} \text{ with } i=3 \text{ or } 4.\] By Corollary \ref{cor.no-factor} and Lemma \ref{lem.no-irreducible-factor}, we see that $G$ is not a cusp form. This derives a contradiction. Using the same argument, we can prove that none of (I.6) and (I.5) holds.

Suppose that (I.4) holds. Then, without loss of generality we may assume that we have 
\[\bar \tau^{\rm ss} = \bar \rho_f \oplus \bar \chi^{l_1} \oplus \bar \chi^{l_2}\]
with $l_1,l_2 \in \{1-k, \pm(l-4)+1-k,\pm(l-3)+1-k \}$. 
 Then, again by Corollary \ref{cor.no-factor} and Lemma \ref{lem.non-trivial-extension2}, we easily see that we have 
\[\bar \tau^{\rm ss} =\bar \rho_f \oplus \bar \chi^{l-k-2} \oplus \bar \chi^{5-l-k}.\]
Put $\sigma=\tau \otimes \chi^{k-1}$. Then, $\sigma$ is an irreducible factor of $\rho_{G,\St}$,
and 
\[\bar \sigma^{\rm ss}= \bar \rho_f(k-1) \oplus \bar \chi^{l-3} \oplus \bar \chi^{4-l}.\]
Since we have 
\[(\bar \sigma^\vee)^{\rm ss} \cong \bar \rho_f(k-2) \oplus \bar \chi^{3-l} \oplus \bar \chi^{l-4} \not \cong \bar \sigma^{ss},\]
and $\rho_{G,\St}$ is self-dual, we have 
\[\rho_{G,\St}=\sigma \oplus \sigma^\vee \oplus 1.\]
By Lemma \ref{lem.det-of-irreducible-factor}, $G$ is not tempered, and $\psi_G=\pi[2] \boxplus 1$ with $\pi$ an
 irreducible unitary self-dual cuspidal automorphic representation of $\mathrm{PGL}_4(\AAA_\QQ)$.
 Hence, by Proposition \ref{prop.A-parameters}, $G$ must be the lift of type $\scra^{(I)}$. This completes (4). 
Similarly we can prove  (3). 

Finally we prove (2). Suppose that
\[\rho_{G,{\rm St}} \cong \rho_{\scri_2(f),{\rm St}} \mod \frkp.\]
Then we have 
\[\bar \rho_{G,{\rm St}}^{\rm ss}=\bar \rho_f(k-1) \oplus \bar \rho_f(k-2) \oplus \bar 1.\]
First $G$ is not tempered. Then, $G=\scri_2(g)$ with a primitive form $g$ in $S_{2k-2}(\SL_2(\ZZ)$.
By the assumption, we have $g \not=f$ and
\[\rho_{\scri_2(f),\St} \equiv \rho_{\scri_2(g),\St} \mod \frkp.\]
Then, 
\[\bar \rho_f(k-1) \oplus \bar \rho_f(k-2) \oplus \bar 1 = \bar \rho_g(k-1) \oplus \bar \rho_g(k-2) \oplus \bar 1.\]
This implies that
\[(\bar \chi + \bar 1)\bar \rho_f=(\bar \chi +\bar 1)\bar \rho_g\]
in $\scra_{\widetilde \frkp}$, where $\scra_{\widetilde \frkp}$ is the Grothendieck ring of the Galois representations of $G_{\QQ}$ in \cite[Section 4]{A-C-I-K-Y23}. Hence, by \cite[Lemma 4.7]{A-C-I-K-Y23}, we have $\bar \rho_f =\bar \rho_g$ in $\scra_{\widetilde \frkp}$. This implies that $\frkp$ divides $\frkD_f^{(p)}$ and derives a contradiction.

Next suppose that $G$ is tempered. Then, by Theorem \ref{th.irred2},
$\rho_{G,\rm{St}}$ is irreducible. Then, using the same argument as in the proof of (4) combined  with   Lemma \ref{lem.non-trivial-extension2}, we can derive a contradiction.

\end{proof}

\begin{lemma} \label{lem.cong-phi-operator}
Let ${\bf k}=(k_1,\ldots,k_r,k_{r+1},\ldots,k_n)$ and ${\bf l}=(k_1,\ldots,k_r)$ with $k_1 \ge \cdots \ge k_r \ge k_{r+1} \ge\cdots \ge k_n$.
Let $F$ and $G$ be Hecke eigenforms in $M_{\bf k}(\varGamma^{(n)})$ such that 
 such that $\Phi_r^n(G)$ is a Hecke eigenform in $S_{\bf l}(\varGamma^{(r)})$. Let $\frkp$ be a prime ideal of $\QQ(F)$ and suppose that
$G \equiv_{\ev} F \mod \frkp$. 
Then, we have 
\[\rho_{\Phi_r^n(G),\St} \equiv \rho_{\Phi_r^n(F ),\St} \mod \frkp.\]
\end{lemma}
\begin{proof} 
By the assumption, for any prime number $\ell \not=p$, we have
\[L_{\ell}(X,G,\St) \equiv L_{\ell}(X,F,\St) \mod \widetilde \frkp,\]
where $\widetilde \frkp$ is a prime ideal of $\QQ(F)\QQ(G)$ lying above $\frkp$. 
By definition, we have
\[L_{\ell}(X,F,\St) =L_{\ell}(X,\Phi_r^n(F),\St)\prod_{i=r+1}^n(1-{\ell}^{k_i-i}X)(1-{\ell}^{i-k_i}X),\]
and
\[L_{\ell}(X,G,\St) =L_{\ell}(X,\Phi_r^n(G),\St)\prod_{i=r+1}^n(1-{\ell}^{k_i-i}X)(1-{\ell}^{i-k_i}X).\]
Hence we have
\[L_{\ell}(X,\Phi_r^n(G),\St) \equiv L_{\ell}(X,\Phi_r^n(F),\St) \mod \widetilde \frkp.\]
Thus, by the Chebotarev density theorem, we have
\[\rho_{\Phi_r^n(G),\St} \equiv \rho_{\Phi_r^n(F),\St} \mod \frkp.\]

\end{proof}

By Theorem \ref{th.noncongruence} and Lemma \ref{lem.cong-phi-operator}, we easily obtain the following theorem.
\begin{theorem} \label{th.noncongruence-Phi}
Let the notation and the assumption be as in Theorem \ref{th.noncongruence}, and put ${\bf k}=(k,k,l,l)$.
\begin{itemize}
\item [(1)] Let $k=l$. 
Then, we have
\[E_{4,k} \not\equiv_{\ev} [\scri_2(f)]^{{\bf k}} \mod \frkp ,\]
and for any Hecke eigenform $g$ in $S_k(\SL_2(\ZZ))$, we have
\[[g]^{\bf k} \not\equiv_{\ev} [\scri_2(f)]^{{\bf k}} \mod \frkp.\]
\item [(2)] For each $2 \le r \le 3$, put ${\bf k}_r=(k,k,\overbrace{l,\ldots,l}^{r-2})$.
Let $G$ be a non-cuspidal Hecke eigenform in $M_{\bf k}(\varGamma^{(4)})$, and suppose that $\Phi_r^4(G)$ is a Hecke eigenform  in $S_{{\bf k}_r}(\varGamma^{(r)})$ and that $\Phi_2^4(G)$ is not a constant multiple of $\scri_2(f)$. Then
\[G\not\equiv_{\ev} [\scri_2(f)]^{{\bf k}} \mod \frkp.\]
\item [(3)] Let $G$ be a Hecke eigenform in $S_{\bf k}(\varGamma^{(4)})$ and suppose that $G$ is not the lift of type $\scra^{(I)}$. Then,
\[G\not\equiv_{\ev} [\scri_2(f)]^{{\bf k}} \mod \frkp.\]
\end{itemize}
\end{theorem}

\bigskip

\noindent
{\bf Proof of Theorem \ref{th.main-result}}

Now we prove Theorem \ref{th.main-result}. By the  conditions (C.1), (C.2), (C.3), (C.4), (C.8) and Theorem \ref{th.main-congruence}, there exists a Hecke eigenform $G$ in $\widetilde M_{\bf k}(\varGamma^{(4)})$ such that $G$ is not a constant multiple of $[\scri_2(f)]^{{\bf k}}$ and 
\[G \equiv_{\ev} [\scri_2(f)]^{\bf k} \text { mod } \frkp.\]
We prove that $G$ is the lift of type $\scra^{(I)}$.
 Suppose  that $G$ is neither the  lift of type $\scra^{(I)}$ nor
the Klingen-Eisenstein lift of the Saito-Kurokawa lift of $f$. 
If $G$ is non-cuspidal, by Theorem \ref{th.noncongruence-Phi}, we have
\[G \not\equiv [\scri_2(f)]^{\bf k} \mod \frkp.\]
If $G$ is cuspidal, by  Theorem \ref{th.noncongruence} (4), we have
\[G \not\equiv [\scri_2(f)]^{\bf k} \mod \frkp. \]
This completes the proof.

\section{Proof of Theorem \ref{th.main-result2}}

To prove Theorem \ref{th.main-result2}, we provide the following lemma.
\begin{lemma} \label{lem.comparison-of-periods} Let $g$ be a primitive form in $S_l(\SL_2(\ZZ))$.
Fix a prime $p$ satisfying $p>l-2$. Let $\mathfrak{p}$ be a prime ideal of the ring of integers of $\mathbb{Q}(g)$ dividing $p$.
Assume $\mathfrak{p}\nmid \frkD_g$. Then the product of periods $\Omega (g)_+ \Omega (g)_-$ coincides with the Petersson inner product $( g,g )$ up to $\mathfrak{p}$-adic unit.
\end{lemma}
%The proof is given in Appendix A.
The proof of this lemma was essentially obtained by Dummigan with a brief explanation (cf. \cite[\S 5]{Dummigan09} and \cite[Lemma 1.4]{Dummigan22}) combined with a result of Hida \cite[Theorem 5.20]{Hida00}.
But, for reader's convenience, we give a detailed proof of it in Appendix A.

\bigskip

\noindent
{\bf Proof of Theorem \ref{th.main-result2}}

We have
\begin{align*}
&{\bf L}(k+j+1,f;\Omega_{-}(f))\cdot {\bf L}(k+j-1,f;\Omega_{-}(f))\\ & \times {\bf L}(k+j/2,f;\Omega_+(f)) \cdot {\bf L}(k+j-2,f;\Omega_+(f))\\
&= {\bf L}(k+j/2,k+j-1;f) \cdot {\bf L}(k+j-2,k+j+1;f) \\
& \times \Bigl({ (f,  f ) \over \Omega_{+}(f)\Omega_{-}(f)}\Bigr)^2,
\end{align*}
and by the condition (C.4), (C.8'), and Lemma \ref{lem.comparison-of-periods}, the right-hand side  is not divisible by $\frkp$.
By \cite[Proposition 14.16 and Proposition 14.21]{Kato04}, it is easy to see that the values 
${\bf L}(k+j+1,f;\Omega_{-}(f)),{\bf L}(k+j-1,f;\Omega_{-}(f)), {\bf L}(k+j/2,f;\Omega_{+}(f))$ and ${\bf L}(k+j-2,f;\Omega_{+}(f))$ are all $\frkp$-integral under our assumptions, and hence they are not divisible by $\frkp$.
This implies that the condition (C.8) is satisfied. Moreover, by (C.2'), the condition (C.2) is satisfied. Thus the assertion follows from Theorem \ref{th.main-result}.

\section{Examples}
In \cite{A-C-I-K-Y23}, we confirmed Harder's congruence in the case $(k,j)=(10,4), (14,4)$ and $(4,24)$ without using Theorem \ref{th.main-result2}. In this section, we give some other examples using it.

Let  
\[\phi^\pm(z)=E_6(z)^3 \Delta(z)+(5856 \pm 96 \sqrt{51349})E_6(z)\Delta(z)^3,\]
where $E_6(z)$ is the Eisenstein series of weight $6$ for $\SL_2(\ZZ)$ defined by 
\[E_6(z)=1+240 \sum_{m=1}^\infty \left(\sum_{0<d | m} d^5\right) {\bf e}(mz),\]
and $\Delta(z)$ is Ramanujan's delta function defined by
\[\Delta(z)={\bf e}(z) \prod_{m=1}^{\infty} (1-{\bf e}(mz))^{24}.\]
Then $\phi^+$ and $\phi^-$ are primitive forms in $S_{30}(\SL_2(\ZZ))$.
Put 
\[\alpha^\pm =4320 \pm 96\sqrt{51349}.\] 
Then 
\[a(2,\phi^\pm)=\alpha^\pm, \ a(3,\phi^\pm)=-552\alpha^\pm-99180. \]
 We also have 
\[\lambda_{\scri_2(\phi^\pm)}(T(2))=\alpha^\pm+49152.  \]
Put $f=\phi^+$. We give a list of the prime numbers appearing in $N_{\QQ(f)/\QQ}{\bf L}(k+j,f)/{\bf L}(k+j/2,f)$ greater than $30$.\\

\noindent{\bf Table 1}\\
\begin{tabular}{|l|l|l|l|l|l|l|} 
\hline
$p$&  4289  &  67021  & 24251 &\ 1657 & 593 &\ 97\\
\hline 
$(k,j)$ &  (14,4) &  (12,8) &    (10,12) &  (8,16) & (6,20) & (4,24) \\
\hline
\end{tabular}\\

We note all the prime numbers  in  Table 1 except $593$ satisfy the conditions (C.5), (C.6) and (C.7) of Theorem \ref{th.main-result}. 
We note that all the primes except $67021$ and $593$ are regular primes, and $67021$ also satisfies the conditions (C.6) and (C.7). Since $593$ divides $\zeta(-21)$,
it  does not satisfy the condition (C.7).\footnote{This was pointed out by Nobuki Takeda.}

(1)  We have confirmed Harder's congruence in the case
$(k,j)=(14,4)$ and $(4,24)$ as stated before.

(2) Let $(k,j)=(12,8)$. Then the prime number $6701$ divides
$N_{\QQ(f)/\QQ}({\bf L}(20,f)/{\bf L}(16,f))$ and it splits  in $\QQ(f)$. Hence there exist  prime ideals $\frkq, \frkq'$ of $\QQ(f)$ such that $(6701)=\frkq \frkq'$ and
$\frkq$ divides ${\bf L}(20,f)/{\bf L}(16,f)$
and  satisfies (C.4) of 
Theorem \ref{th.main-result}, and (C.8') of Theorem \ref{th.main-result2}. Put ${\bf k}=(16,16,8,8)$. 
Let $N=\left(\begin{matrix} 
1    &1/2&0&1/2 \\
1/2& 1 & 0 & 0 \\
0 & 0 & 1& 1/2 \\
1/2 &0 & 1/2 & 1
\end{matrix} \right), \ N_1=\left(\begin{matrix} 1 & 0 \\ 0  & 1 \end{matrix}\right)$, 
and $V=\begin{pmatrix} 1 & 0 & 1 & 0 \\ 1 & 1 & 0 & 3 \end{pmatrix}$. Put 
\begin{align*}
&e_1= \epsilon_{16,{\bf k}}(N_1,N)(V), \\ &e_2=\epsilon_{16,{\bf k}}(2N_1,N)(V)+2^{14} \epsilon_{16,{\bf k}}(N_1,N)(V)
\end{align*}
and 
\[\alpha_{16,{\bf k}}(N_1,N)= \begin{vmatrix}  e_1 & 1 \\
e_2 & \lambda_{\scri_2(f')}(T(2))
\end{vmatrix}.\]
By a computation with \cite{Mathematica20}, we have
\[\epsilon_{16,{\bf k}}(N_1,N)=2170426272885/2816,
\quad 
 \epsilon_{16,{\bf k}}(2N_1,N)=526133454959385/22
,\]
and hence
\begin{align*}
\alpha_{16,{\bf k}}(N_1,N)=15795 (26020274165 + 412236709 \sqrt{51349})/88.
\end{align*}
Hence
\[N_{\QQ(f)/\QQ}(\alpha_{16,{\bf k}}(N_1,N)) \equiv 2310 \mod{6701}.\]
In view of Proposition \ref{prop.fc-klingen1}, this implies that $f$ satisfies the condition (C.3) of Theorem \ref{th.main-result}.
Hence, by Theorem \ref{th.main-result2} we prove that 
 there exists a Hecke eigenform $G$ in $S_{(20,12)}(\varGamma^{(2)})$ such that  
\[\scra_4^{(I)}(G)^{{\bf k}}
 \equiv_{\mathrm{ev}} [\scri_2(f)]^{{\bf k}} \mod \frkq .\]
 
(3) In a way similar to above, we can show that 
there exists a Hecke eigenform $G$ in $S_{(k+j,j)}(\varGamma^{(2)})$ such that  
\[\scra_4^{(I)}(G)^{{\bf k}}
 \equiv_{\mathrm{ev}} [\scri_2(f)]^{{\bf k}} \mod \frkq \]
in the case $(k,j)=(10,12)$ and $(8,16)$.

Summarizing these results, in view of Theorem \ref{th.Harder-congruence}, we have the following result:
\begin{theorem}
\label{th.examples}
Harder's conjecture holds for $(k,j)=(12,8),(10,12)$ and $(8,16)$.
\end{theorem}

\appendix
\section{Comparison of periods}

In this appendix we prove the coincidence of the periods (Lemma 9.1).
We first show that the integral structures of de Rham cohomology of Kuga-Sato varieties coincides
with the integral structure coming from the $q$-expansion principle.

We recall some notations about modular curves and Kuga-Sato varieties.
Fix a positive integer $N$ greater than $3$.
Let $X$ be the modular curve $X(N)$ (or $X_1(N)$) over $\mathbb{Q}_p$ and $Y$ the open modular curve $Y(N)$ (or $Y_1(N)$) over $\mathbb{Q}_p$.
Let $Z$ be the cusps of $X$.
We also consider the integral models $\mathcal{Y}\subset \mathcal{X}$ of $Y \subset X$ over $\mathbb{Z}_p$.
Let $\lambda : E\to Y$ (resp. $\lambda : \mathcal{E} \to \mathcal{Y}$) be the universal elliptic curve over $Y$ (resp. $\mathcal{Y}$).
Let $\overline{\lambda}:\overline{E}\to X$ (resp. $\overline{\lambda} : \overline{\mathcal{E}} \to \mathcal{X}$) be the universal generalized elliptic curve over $X$ (resp. $\mathcal{X})$.
Denote the $k$-fold fiber products by $E^k=E\times_Y \cdots \times_Y E$ and $\overline{E}^k=\overline{E}\times_X \cdots \times_X \overline{E}$.
We define $\mathcal{E}^k$ and $\overline{\mathcal{E}}^k$ similarly.
Let $\mathrm{KS}_{k}\to X$ (resp. $\mathcal{KS}_{k}\to \mathcal{X}$) be Deligne's desingularization of $\overline{E}^k$
(resp. $\overline{\mathcal{E}}^k$).
For the integral case, we refer \cite[Appendix]{BDP13} for the construction.

To compare the two integral structures, we use the parabolic de Rham cohomology.
We recall some notations.
Let $\omega =\overline{\lambda}_* \Omega^{\mathrm{reg}}_{\overline{E}\slash X}=e^*\Omega^1_{\overline{E}\slash X}$ be the relative differential,
where $e:X\to \overline{E}$ is zero section.
We consider the relative de Rham cohomology sheaf $\mathcal{L}=\mathbb{R}^1\pi_*\Omega^\bullet_{E\slash Y}$, which is a rank 2 vector bundle on $Y$.
The sheaf $\mathcal{L}$ extends to $X$. Denote it by $\mathcal{L}$ again.
The Gauss-Manin connection $\nabla : \mathcal{L}\to \mathcal{L}\otimes \Omega_Y^1$ extends to a connection with logarithmic poles
$\nabla : \mathcal{L}\to \mathcal{L}\otimes \Omega_X^1 (\log Z)$.
Note that $\omega^2 \cong \Omega_X^1(\log Z)$ and $\omega^2(- Z)\cong \Omega_X^1(\log Z)(-Z)=\Omega_X^1$.
We use similar notations for the integral case.
We put $S_{k}(\mathcal{X},\mathbb{Z}_p)=\mathit{H}^0(\mathcal{X},\omega^k(-\mathcal{Z}))$
and $S_{k}(X,\mathbb{Q}_p)=\mathit{H}^0({X},\omega^k(-{Z}))$.
These spaces are identified with the space of cusp forms with coefficients in $\mathbb{Z}_p$ and in $\mathbb{Q}_p$
by the theory of Katz geometric modular forms.
Define the parabolic de Rham cohomology by
$$\mathit{H}_{\mathrm{dR}, int}^i(\mathcal{E}^k)=\operatorname{Im}(\mathit{H}_{\mathrm{dR}, c}^i(\mathcal{E}^k)\to \mathit{H}^i_{\mathrm{dR}}({E}^k)),$$
where $\mathit{H}_{\mathrm{dR}}^i(\mathcal{E}^k)=
\mathbb{H}^i(\Omega^\bullet_{\mathcal{KS}_{k}\slash \mathbb{Z}_p}(\log (\mathcal{KS}_k \setminus \mathcal{E}^k)))$
and
$\mathit{H}_{\mathrm{dR},c}^i(\mathcal{E}^k)=
\mathbb{H}^i(\Omega^\bullet_{\mathcal{KS}_{k}\slash \mathbb{Z}_p})$.
We also define
$$
\mathit{H}^i_{\mathrm{dR}, c}(\mathcal{Y},\operatorname{Sym}^r\mathcal{L})
=\mathbb{H}^i(\operatorname{Sym}^r\mathcal{L}\otimes\Omega^\bullet_{\mathcal{X}}(\log \mathcal{Z})(-\mathcal{Z}))
=\mathbb{H}^i(\operatorname{Sym}^r\mathcal{L}\otimes\Omega^\bullet_{\mathcal{X}}),
$$
$$
\mathit{H}^i_{\mathrm{dR}}(\mathcal{Y},\operatorname{Sym}^r\mathcal{L})
=\mathbb{H}^i(\operatorname{Sym}^r\mathcal{L}\otimes\Omega^\bullet_{\mathcal{X}}(\log \mathcal{Z}))
$$
and
$$
\mathit{H}^i_{\mathrm{dR}, int}(\mathcal{Y},\operatorname{Sym}^r\mathcal{L})
=\operatorname{Im}\left(\mathit{H}^i_{\mathrm{dR}, c}(\mathcal{Y},\operatorname{Sym}^r\mathcal{L})\to \mathit{H}^i_{\mathrm{dR}}(\mathcal{Y},\operatorname{Sym}^r\mathcal{L})\right).
$$
%A standard argument shows the following result.
The following proposition follows from a standard argument.
\begin{proposition}\label{Deligne}
Assume $p>k-2$ and $p\nmid N$. Then we have
$\mathit{H}_{\mathrm{dR}, int}^{k-1}(\mathcal{E}^{k-2})(\varepsilon)\cong \mathit{H}^1_{\mathrm{dR}, int}(\mathcal{Y},\operatorname{Sym}^{k-2}\mathcal{L})$,
where $\varepsilon =\varepsilon_k$ is the projector defined in \cite{Scholl90}.
\end{proposition}
This is proved by Deligne's argument \cite[Lemme 5.3]{Deligne68} using Leray spectral sequence
for the log de Rham and the compact supported log de Rham cohomology
combined with Lieberman's trick (see also \cite[Proof of Lemma 2.2 (1)]{BDP13}).

\begin{remark}
To prove the above proposition, one needs the Leray spectral sequence for the log de Rham and the compact supported log de Rham cohomology.
These spectral sequences can be constructed in a way similar to  that in \cite[Remark 3.3]{Katz70} with minor modifications. 
\end{remark}

We note that $\mathit{H}_{\mathrm{dR}, c}^{k-1}(\mathcal{E}^{k-2})\to \mathit{H}^{k-1}_{\mathrm{dR}}({E}^{k-2})$ factors through $\mathit{H}_{\mathrm{dR}}^{k-1}(\mathcal{KS}_{k-2})$.
Therefore it is enough to compare the images of $S_{k}(\mathcal{X},\mathbb{Z}_p)=\mathit{H}^0(\mathcal{X},\omega^k(-\mathcal{Z}))$
and $\mathit{H}^1_{\mathrm{dR}, int}(\mathcal{Y},\operatorname{Sym}^{k-2}\mathcal{L})$ in $\mathit{H}^i_{\mathrm{dR}, int}({Y},\operatorname{Sym}^{k-2}\mathcal{L})$
under the following natural maps:
%$$
%\iota_{\mathbb{Z}_p}:S_{k}(\mathcal{X},\mathbb{Z}_p)=\mathit{H}^0(\mathcal{X},\omega^k(-\mathcal{Z}))\hookrightarrow 
%\mathit{H}^i_{\mathrm{dR, int}}(\mathcal{Y},\operatorname{Sym}^r\mathcal{L}),
%$$
%$$
%\iota_{\mathbb{Q}_p}:S_{k}(\mathcal{X},\mathbb{Q}_p)=\mathit{H}^0({X},\omega^k(-{Z}))\hookrightarrow 
%\mathit{H}^i_{\mathrm{dR, int}}({Y},\operatorname{Sym}^r\mathcal{L}),
%$$
%$$
%\varphi : \mathit{H}^0(\mathcal{X},\omega^k(-\mathcal{Z}))\to \mathit{H}^0{X},\omega^k(-{Z}))
%$$
%and
%$$
%\psi : \mathit{H}^i_{\mathrm{dR, int}}(\mathcal{Y},\operatorname{Sym}^r\mathcal{L})\to \mathit{H}^i_{\mathrm{dR, int}}({Y},\operatorname{Sym}^r\mathcal{L}).
%$$
%
\[
  \xymatrix{
    S_{k}(\mathcal{X},\mathbb{Z}_p)=\mathit{H}^0(\mathcal{X},\omega^k(-\mathcal{Z}))  \ar@{^{(}->}[r]^{\quad \quad \iota_{\mathbb{Z}_p}}  \ar[d]_{\varphi} &  \mathit{H}^1_{\mathrm{dR}, int}(\mathcal{Y},\operatorname{Sym}^{k-2}\mathcal{L}) \ar[d]^{\psi} \\
     S_{k}({X},\mathbb{Q}_p)=\mathit{H}^0({X},\omega^k(-{Z}))  \ar@{^{(}->}[r]^{\quad \quad \iota_{\mathbb{Q}_p}}  & \mathit{H}^1_{\mathrm{dR}, int}({Y},\operatorname{Sym}^{k-2}\mathcal{L}).
  }
\]

\begin{proposition}\label{comparison-periods}
For an element $f$ in $S_{k}(X,\mathbb{Q}_p)$,
$f$ belongs to $\operatorname{Im} \varphi$ if and only if $\iota_{\mathbb{Q}_p}(f)$ belongs to $\operatorname{Im} \psi$.
In other words, the integral structure coming from $\mathit{H}_{\mathrm{dR}}^{k-1}(\mathcal{KS}_{k-2})$ coincides with the integral structure induced by the $q$-expansion
principle.
\end{proposition}
\begin{proof}\footnote{This proof is suggested by Kai-Wen Lan.}
First, we assume that $\iota_{\mathbb{Q}_p}(f)$ belongs to $\operatorname{Im} \psi$.
We put
$$
K=\operatorname{Ker}\left(\mathit{H}^1_{\mathrm{dR}, c}(\mathcal{Y},\operatorname{Sym}^{k-2}\mathcal{L}) \to \mathit{H}^1_{\mathrm{dR}}(\mathcal{Y},\operatorname{Sym}^{k-2}\mathcal{L})\right)
$$
and
$$
K'=\operatorname{Ker}\left(\mathit{H}^1(\mathcal{X},\omega^{-k+2}(-\mathcal{Z}))\to \mathit{H}^1(\mathcal{X},\omega^{-k+2})\right).
$$
Then we have the following diagram:
\[
  \xymatrix{
&0 \ar[r] \ar[d] & K \ar@{^{(}->}[r] \ar@{^{(}->}[d] & K' \ar@{^{(}->}[d]\\
0 \ar[r]  & \mathit{H}^0(\mathcal{X},\omega^k(-\mathcal{Z}))  \ar[r]  \ar@{^{(}->}[d]_{\alpha} &  \mathit{H}^1_{\mathrm{dR}, c}(\mathcal{Y},\operatorname{Sym}^{k-2}\mathcal{L}) \ar[r] \ar[d]_{\beta}  & \mathit{H}^1(\mathcal{X},\omega^{-k+2}(-\mathcal{Z})) \ar[r] \ar[d]_{\gamma} &0 \\
0 \ar[r]   &  \mathit{H}^0(\mathcal{X},\omega^k)  \ar[r]  & \mathit{H}^1_{\mathrm{dR}}(\mathcal{Y},\operatorname{Sym}^{k-2}\mathcal{L}) \ar[r] & \mathit{H}^1(\mathcal{X},\omega^{-k+2}) \ar[r] & 0.
  }
\]
Since $\operatorname{Im} \alpha = \mathit{H}^0(\mathcal{X},\omega^k(-\mathcal{Z})) = S_{k}(\mathcal{X},\mathbb{Z}_p)$ and
$\operatorname{Im} \beta =  \mathit{H}^1_{\mathrm{dR}, int}(\mathcal{Y},\operatorname{Sym}^{k-2}\mathcal{L})$, we have a short exact sequence
$$
0\to S_{k}(\mathcal{X},\mathbb{Z}_p) \to  \mathit{H}^1_{\mathrm{dR}, int}(\mathcal{Y},\operatorname{Sym}^{k-2}\mathcal{L}) \overset{\delta}{\to} 
 \mathit{H}^1(\mathcal{X},\omega^{-k+2}(-\mathcal{Z}))\slash K \to 0.
$$
Let $s$ be an element in $\mathit{H}^1_{\mathrm{dR}, int}(\mathcal{Y},\operatorname{Sym}^{k-2}\mathcal{L})$ satisfying $\psi (s)=\iota_{\mathbb{Q}_p}(f)$.
To prove $f\in \operatorname{Im} \varphi$, it is enough to show $\delta (s)=0$.
It is easy to see that $\delta (s)$ is a torsion element.
We consider the similar exact sequence
$$
0\to H^0({X},\omega^k) \to  \mathit{H}^1_{\mathrm{dR}}({Y},\operatorname{Sym}^{k-2}\mathcal{L}) \overset{\delta'}{\to} 
 \mathit{H}^1({X},\omega^{-k+2}) \to 0.
$$
Then one has $\delta'(\iota_{\mathbb{Q}_p}(f))=0$.
Since $\mathit{H}^1(\mathcal{X},\omega^{-k+2})$ is torsion-free (cf. \cite[p.\ 58]{Scholl85}),
the image of $s$ under the map $\mathit{H}^1_{\mathrm{dR}}(\mathcal{Y},\operatorname{Sym}^{k-2}\mathcal{L}) \to \mathit{H}^1(\mathcal{X},\omega^{-k+2})$
is also zero.
This implies that $\delta (s)$ is a torsion element in $K'\slash K$.
By the snake lemma, we have
$K'\slash K \subset \operatorname{Coker} ( \mathit{H}^0(\mathcal{X},\omega^k(-\mathcal{Z})) \hookrightarrow  \mathit{H}^0(\mathcal{X},\omega^k))
\subset \mathit{H}^0(\mathcal{Z},\omega^k)$.
Since $\mathit{H}^0(\mathcal{Z},\omega^k)$ is a finite free $\mathbb{Z}_p$-module, $K'\slash K$ is a free $\mathbb{Z}_p$-module. Therefore $\delta (s)=0$.
This shows $f\in \operatorname{Im}\varphi$.
%The opposite direction is obvious.
The converse is immediate.
Hence this completes the proof.
\end{proof}
\begin{remark}
Scholl \cite{Scholl85} introduced a parabolic de Rham cohomology $L_k(\mathcal{X},\mathbb{Z}_p)$.
In fact, $L_k(\mathcal{X},\mathbb{Z}_p)$ can be identified with $\mathit{H}^1_{\mathrm{dR}, int}(\mathcal{Y},\operatorname{Sym}^{k-2}\mathcal{L})$.
Moreover, Scholl obtained a short exact sequence
$$
0\to \mathit{H}^0(\mathcal{X},\omega^k(-\mathcal{Z})) \to L_k(\mathcal{X},\mathbb{Z}_p) \to \mathit{H}^1(\mathcal{X},\omega^{-k})\to 0
$$
and $\mathit{H}^1(\mathcal{X},\omega^{-k})\cong S_{k}(\mathcal{X},\mathbb{Z}_p)^\vee$ is a free $\mathbb{Z}_p$-module (\cite[p.58]{Scholl85}).
Therefore it is easy to see that this fact also implies Proposition \ref{comparison-periods}.
\end{remark}

We choose an auxiliary integer $N$ with $N\geq 4$.
%Here we use the notations explained in \S 8.
We use the notations of \S 8.
We put $F=\mathbb{Q}(g)$ and let $S(g)$ be the quotient of $M_k(X_1(N))\otimes F$ by the $F$-submodule generated by
the images of the operators $T(m)\otimes 1 -1\otimes a(m,g)$ for all $m\geq 1$.
Similarly we denote by $V_F(g)$ the quotient of $V_{k,F}(Y_1(N))$ by the $F$-submodule generated by
the images of the operators $T(m)\otimes 1 -1\otimes a(m,g)$ for all $m\geq 1$.

Recall that $D$ is the $\frkO_{\frkp}$-lattice generated by the image of
\[H^{k-1}(\mathcal{KS}_{k-2},\Omega_{\mathcal{KS}_{k-2}/\ZZ_p}^{\ge k-1})( \varepsilon) \to S(g),\]
and $T$ is the image of
\[H^{k-1}(\mathrm{KS}_{k-2} \otimes_{\QQ_p} \bar \QQ_p,\frkO_{\frkp})( \varepsilon) \to V_{\QQ(g)_{\frkp}}(g).\]
It is easy to see that $T$ coincides with $V_{k,\frkO_{\mathfrak{p}}}(Y_1(N))\cap V_{\QQ(g)_{\frkp}}(g)$ by the same argument as in the proof of Proposition \ref{Deligne}
(cf. \cite[1.2.1 Theorem and 4.1.1 Proposition]{Scholl90}).
Then $D$ is a free $\frkO_{\frkp}$-module of rank 1 and $T$ is a free $\frkO_{\frkp}$-module of rank 2.
Moreover $T^\pm$ are free $\frkO_{\frkp}$-modules of rank 1.
By the $q$-expansion principle, the primitive form $g$ can be identified with an element $\omega_g$ in $S(g)$ and 
the element $\omega_g$ gives a basis of $D$ by Proposition \ref{comparison-periods}.
Let $\gamma^{\pm}$ be a basis of $T^\pm$.
Then Kato's period $\Omega^\pm$ is defined by $\mathrm{per}_g(\omega_g)^\pm = \Omega^{\pm}\gamma^\pm$.
Hence we have $\mathrm{per}_g(\omega_g)=\Omega^+ \gamma^+ +\Omega^-\gamma^-$.

For a commutative ring $A$, let $\langle \,  , \,  \rangle_B : V_{k,A}(Y_1(N))\times V_{k,A,c}(Y_1(N)) \to A$ be the Poincare duality pairing.
We define an anti-$\mathbb{C}$-linear operator $\iota' : V_{k,\mathbb{C}}(Y_1(N))\to V_{k,\mathbb{C}}(Y_1(N))$ by
$
\iota'(x\otimes y)=x\otimes \overline{y}  
$
for $x\in V_{k,\mathbb{R}}(Y_1(N))$ and $y\in \mathbb{C}$.
Then by \cite[(7.13.5)]{Kato04}, we have
$$
\langle \mathrm{per}_g(\omega_g), \iota' \mathrm{per}_g(\omega_g)\rangle_B =(-8\pi^2 \sqrt{-1})^{k-1} \int_{\Gamma_1(N)\backslash \mathfrak{H}}\overline{g(\tau)}g(\tau)y^{k-2}dx\wedge dy,
$$
where $\tau =x+\sqrt{-1}y$ with $x,y \in \mathbb{R}$.
On the other hand, it is easy to see
$$
\langle \mathrm{per}_g(\omega_g), \iota' \mathrm{per}_g(\omega_g)\rangle_B
=-2\Omega^+ \Omega^- \langle \gamma^+, \gamma^-\rangle_B.
$$
Since we defined $\Omega (g)_{\pm}=(2\pi \sqrt{-1})^{1-k}\Omega^{\mp}$, one has
$$
(2\sqrt{-1})^{k-1}(g,g)=-2 \Omega(g)_+\Omega(g)_- \langle \gamma^+, \gamma^- \rangle_B.
$$
Here the pairing
$$\langle \,  , \,  \rangle_B: V_{k,\frkO_{\mathfrak{p}}}(Y_1(N))\times V_{k,\frkO_{\mathfrak{p}},c}(Y_1(N)) \to \frkO_{\mathfrak{p}}$$
is perfect, since $p>k-2$ and $p\nmid N$.
The pairing $\langle \,  , \,  \rangle_B$ induces a perfect pairing
$$\langle \,  , \,  \rangle_{B,int}: V_{k,\frkO_{\mathfrak{p}},int}(Y_1(N))\times V_{k,\frkO_{\mathfrak{p}},int}(Y_1(N)) \to \frkO_{\mathfrak{p}},$$
where
$$V_{k,\frkO_{\mathfrak{p}},int}(Y_1(N))=\operatorname{Im}\left( V_{k,\frkO_{\mathfrak{p}},c}(Y_1(N))\to V_{k,\frkO_{\mathfrak{p}}}(Y_1(N)) \right)$$
is the parabolic Betti cohomology.
Let $\langle \, , \, \rangle_T : T\times T \to \frkO_{\mathfrak{p}}$ be the restriction of $\langle \,  , \,  \rangle_{B,int}$ to $T\times T$.
Then the discriminant of $\langle \, , \, \rangle_T$ is given by $\langle \gamma^+, \gamma^- \rangle_B^2$.
Moreover it is well-known that $\mathfrak{p}$ divides the discriminant of $\langle \, , \, \rangle_T$ if and only if $\mathfrak{p}$ divides the congruence ideal for $g$ (cf. \cite[Proposition 2.7.16]{Bellaiche21}).
Hence the assumption $\mathfrak{p}\nmid \frkD_g$ implies that  $\langle \gamma^+, \gamma^- \rangle_B$ is a $\mathfrak{p}$-adic unit.
Therefore the quotient $(g,g) \slash (\Omega(g)_+\Omega(g)_-)$ is a $\mathfrak{p}$-adic unit.

\bigskip 
{\bf Comments on \cite{A-C-I-K-Y23}}

(1) page 1359, line -3: For ``the Grothendieck group'', read 
``the Grothendieck ring''.

(2) page 1384, line 9: For ``$2^a$'', read ``$2^{-a}$''.

(3) page 1380, line 9: Insert ``such that $p_\frkp>2k$'' after ``$\QQ(F)$''.

(4) page 1387: Let $N$ be as on page 1385, and let $N_1$ and $N_2$ be those on page 1387. For each $B \in \calh_2(\ZZ)_{>0}$, put
$$\eta(B)=Z(6,16)a(B,E_{2,16})a(N,E_{4,16}).$$
Moreover, for each $B=N_1,2N_1,3N_1,4N_1$ and $N_2$, let 
$g(B)$ be the value of the right-hand side of the equality on page 1387, lines 17, 18, 19, 20 and 21, respectively.
Then, by some reason, we subtracted $\eta(B)$ from $\varepsilon_{16,{\bf k}}(B,N)$. Therefore, for such a $B$, the correct value of $\varepsilon_{16,{\bf k}}(B,N)$ should be $g(B)+\eta(B)$.\footnote{This was pointed out by Chul-hee Lee.} Therefore, we have
$$e_i=g_i+f_i \text{ for any } i=1,2,3,4,$$
where
$$g_1=g(N_1), \ g_2=g(2N_1),  \ g_3=g(3N_1)+3^{14}g(N_1),$$
$$g_4=g(4N_1)+2^{29}g(N_1)+3 \cdot 2^{14}g(N_2),$$
and 
$$f(N_1)=\eta(N_1), \ f(2N_1)=\eta(2N_1), \ f_3=\eta(3N_1)+3^{14}\eta(N_2),$$
$$f_4=\eta(4N_1)+2^{29}\eta(n_1)+3\cdot 2^{14} \eta(N_2).$$
But, since 
$$f_2=\lambda_{E_{12,16}}(T(2))f_1, \ f_3=\lambda_{E_{12,16}}(T(3))f_1, \ f_4=(\lambda_{E_{12,16}}(T(2)))^2f_1,$$ 
the determinant on the right-hand side of the equality on page 1387, line 14 is unchanged if one replaces 
its first colomn  with $\Big(\begin{smallmatrix} g_1 \\g_2 \\ g_3 \\ g_4 \end{smallmatrix}\Big)$. This implies that the assertion on page 1387, line 2 from the bottom remains valid.

\end{document}